\documentclass[reqno]{amsart}
\usepackage{graphicx}
\usepackage{amsmath,amssymb,amsfonts}
\usepackage{color}
\usepackage{fullpage}
\usepackage{stix}

\usepackage[colorlinks=true, pdfstartview=FitV, linkcolor=blue, citecolor=blue, urlcolor=blue]{hyperref}
\numberwithin{equation}{section}

\begin{document}
% use this macro for differentials.
\newcommand{\dd}{\mathrm{d}}
% use this macro for the imaginary unit.
\newcommand{\ii}{\mathrm{i}}
% use this macro for exponentials.
\newcommand{\ee}{\mathrm{e}}
% put in macros for blue/red things.
%\newcommand{\LZeroBlue}{{\color{cyan}L^0_\mathrm{blue}}}
\newcommand{\LZeroBlue}{{L^0_\squarellblack}}
\newcommand{\LZeroBlueOne}{{L^{0,1}_\squarellblack}}
\newcommand{\LZeroBlueTwo}{{L^{0,2}_\squarellblack}}
\newcommand{\LZeroBlueThree}{{L^{0,3}_\squarellblack}}
\newcommand{\LInftyBlue}{{L^\infty_\squarellblack}}
\newcommand{\LInftyBlueOne}{{L^{\infty,1}_\squarellblack}}
\newcommand{\LInftyBlueTwo}{{L^{\infty,2}_\squarellblack}}
\newcommand{\LZeroRed}{{L^0_\squareurblack}}
\newcommand{\LZeroRedOne}{{L^{0,1}_\squareurblack}}
\newcommand{\LZeroRedTwo}{{L^{0,2}_\squareurblack}}
\newcommand{\LInftyRed}{{L^\infty_\squareurblack}}
\newcommand{\LInftyRedOne}{{L^{\infty,1}_\squareurblack}}
\newcommand{\LInftyRedTwo}{{L^{\infty,2}_\squareurblack}}
\newcommand{\LInftyRedThree}{{L^{\infty,3}_\squareurblack}}
\newcommand{\VZeroBlue}{{\mathbf{V}^0_\squarellblack}}
\newcommand{\VInftyBlue}{{\mathbf{V}^\infty_\squarellblack}}
\newcommand{\VZeroRed}{{\mathbf{V}^0_\squareurblack}}
\newcommand{\VInftyRed}{{\mathbf{V}^\infty_\squareurblack}}
\newcommand{\OurPower}[2]{{#1^{#2}_\squarellblack}}
\newcommand{\OurArg}[1]{{\arg_\squarellblack(#1)}}
\newcommand{\OurLog}[1]{{\log_\squarellblack(#1)}}
\newcommand{\OurLogAvg}[1]{{\langle\log_\squarellblack(#1)\rangle}}
\newcommand{\pEInftyRed}{{\partial E^\infty_\squareurblack}}
\newcommand{\pEZeroBlue}{{\partial E^0_\squarellblack}}
\newcommand{\lNaught}{{p}}
\newcommand{\lNaughtZeroBlue}{{p^0_\squarellblack}}
\newcommand{\lNaughtInftyRed}{{p^\infty_\squareurblack}}
\newcommand{\LensRedPM}{{\Lambda^\pm_\squareurblack}}
\newcommand{\LensBluePM}{{\Lambda^\pm_\squarellblack}}
\newtheorem{lem}{Lemma}
\newtheorem{prop}{Proposition}
\newtheorem{rhp}{Riemann-Hilbert Problem}
\newtheorem{rem}{Remark}
\newtheorem{theorem}{Theorem}
\newtheorem{conjecture}{Conjecture}

\title{Rational Solutions of the Painlev\'e-III Equation}
\author{Thomas Bothner}
\address{Department of Mathematics, University of Michigan, 2074 East Hall, 530 Church Street, Ann Arbor, MI 48109-1043, United States}
\email{bothner@umich.edu}
\author{Peter D. Miller}
\address{Department of Mathematics, University of Michigan, 2074 East Hall, 530 Church Street, Ann Arbor, MI 48109-1043, United States}
\email{millerpd@umich.edu}
\author{Yue Sheng}
\address{Department of Mathematics, University of Michigan, 2074 East Hall, 530 Church Street, Ann Arbor, MI 48109-1043, United States.  Current address:  Department of Mathematics, University of Pennsylvania, 209 South 33rd Street, Philadelphia, PA 19104-6395, United States}
\email{yuesheng@sas.upenn.edu}
\date{\today}


\thanks{TB acknowledges support by the AMS and the Simons Foundation through a travel grant and the work of PDM is supported by the National Science Foundation under grant DMS-1513054.  The authors are grateful to P. Clarkson, A. Its, C.-K. Law and W. Van Assche for useful conversations.}
\begin{abstract}
All of the six Painlev\'e equations except the first have families of rational solutions, which are frequently important in applications.  The third Painlev\'e equation in generic form depends on two parameters $m$ and $n$, and it has rational solutions if and only if at least one of the parameters is an integer.  We use known algebraic representations of the solutions to study numerically how the distributions of poles and zeros behave as $n\in\mathbb{Z}$ increases and how the patterns vary with $m\in\mathbb{C}$.  This study suggests that it is reasonable to consider the rational solutions in the limit of large $n\in\mathbb{Z}$ with $m\in\mathbb{C}$ being an auxiliary parameter.  To analyze the rational solutions in this limit, algebraic techniques need to be supplemented by analytical ones, and 
the main new contribution of this paper is to develop a Riemann-Hilbert representation of the rational solutions of Painlev\'e-III that is amenable to asymptotic analysis.  Assuming further that $m$ is a half-integer, we derive from the Riemann-Hilbert representation a finite dimensional Hankel system for the rational solution in which $n\in\mathbb{Z}$ appears as an explicit parameter.
\end{abstract}
\maketitle

\section{Introduction}
This paper is the first in a series concerned with the large degree asymptotic analysis of rational solutions $u_n(x;m)$ to the generic Painlev\'e-III equation parametrized by $n\in\mathbb{Z}$ and $m\in\mathbb{C}$.
The six Painlev\'e equations are best known for their transcendental solutions, and indeed their general solutions are frequently referred to as \emph{Painlev\'e transcendents}.  These transcendental solutions are modern special functions that have appeared in numerous applications, most famously in similarity solutions of nonlinear partial differential equations and in integrable probability.  However, all of the Painlev\'e equations except the first are actually families of ordinary differential equations indexed by complex parameters, and it is well-known that if the parameters take on certain special values, then the Painlev\'e equation admits particular solutions that are either finitely constructed from elementary special functions or rational functions.\smallskip

For example, the Painlev\'e-II equation $u''=2u^3+xu+m$ has a complex parameter $m$, and it is elementary that if $m=0$ then the equation admits the trivial rational solution $u(x)\equiv 0$.  With this solution in hand for $m=0$, one can apply the \emph{B\"acklund transformation} 
\begin{equation*}
	u(x)\mapsto\widehat{u}(x):=-u(x)-\frac{2m+1}{2u(x)^2+2u'(x)+x}
%	\ \ \ \ \  \textcolor{red}{\textnormal{do doublecheck!}}
\end{equation*} 
taking a solution of the equation with parameter $m$ into another solution of the same equation but with parameter $m\mapsto\widehat{m}:=m+1$.  The B\"acklund transformation obviously preserves rationality and with its help one quickly obtains a rational solution of the Painlev\'e-II equation for each integer value of $m$. It turns out that the integral values of $m$ are the only ones for which the equation admits a rational solution, and for each $m\in\mathbb{Z}$ there is exactly one rational solution, denoted $u_m(x)$, $m\in\mathbb{Z}$.  Motivated by applications, the family of functions $\{u_m(\cdot)\}_{m\in\mathbb{Z}}$ has recently been studied from the analytic perspective, i.e., from the point of view of asymptotic analysis in the limit of large integer $m$ \cite{BertolaB15,BuckinghamM14,BuckinghamM15,MillerS17}.

\subsection{The Painlev\'e-III equation, its symmetries and its rational solutions}\label{master}
The generic Painlev\'e-III equation
\begin{equation}
\frac{\dd^2u}{\dd x^2}=\frac{1}{u}\left(\frac{\dd u}{\dd x}\right)^2-\frac{1}{x}\frac{\dd u}{\dd x} + 
\frac{4\Theta_0 u^2 + 4(1-\Theta_\infty)}{x} + 4u^3-\frac{4}{u},
\label{eq:PIII}
\end{equation}
is the simplest of the Painlev\'e equations having a fixed singular point ($x=0$), and it involves two\footnote{In the most general form of the Painlev\'e-III equation one replaces the terms $4u^3-4u^{-1}$ on the right-hand side by $\gamma u^3 +\delta u^{-1}$ for arbitrary parameters $(\gamma,\delta)\in\mathbb{C}^2$.  Under the generic assumption that $\gamma\delta\neq 0$, a suitable rescaling of the dependent and independent variables results in the form \eqref{eq:PIII}.  There are two singular reductions:  one in which either $\gamma=0$ or $\delta=0$ but not both, which can be reduced by scaling to a one-parameter family of equations (or in the more special case that either $\Theta_0$ or $1-\Theta_\infty$ vanishes to an equation whose general solution is known in closed form), and one in which $\gamma=\delta=0$, which can be reduced by scaling to a unique form if $\Theta_0(1-\Theta_\infty)\neq 0$.  See \cite[\S32.2.2]{DLMF} and \cite[Section 2.2]{GamayunIL13}.} distinct complex parameters $\Theta_0$ and $\Theta_\infty$.  As we shall see, both of these features introduce new phenomena into the behavior of even the most elementary, rational solutions.\smallskip 

In order to study the rational solutions of \eqref{eq:PIII}, it will be convenient to represent the constant parameters $\Theta_0$ and $\Theta_\infty$ in the form
\begin{equation}
\Theta_0=n+m\quad\text{and}\quad\Theta_\infty=m-n+1.
\label{eq:Thetas-m-n}
\end{equation}
Equation \eqref{eq:PIII} has many symmetries, including the following elementary ones:
\begin{itemize}
%\item
%\emph{Odd reflection}:  if $u(x)$ satisfies \eqref{eq:PIII}, then so does $\widetilde{u}(x):=-u(-x)$ for the same values of $\Theta_0$ and $\Theta_\infty$.  
\item
\emph{Inversion}:  if $u(x)$ satisfies \eqref{eq:PIII}--\eqref{eq:Thetas-m-n}, then $u(x)\mapsto I[u](x):=1/u(x)$ satisfies \eqref{eq:PIII} with modified parameters $I:\Theta_0\mapsto \Theta_\infty-1=m-n$ and $I:\Theta_\infty\mapsto \Theta_0+1=m+n+1$ (corresponding to changing the sign of $n$ while holding $m$ fixed).  The mapping $I:(u(x),\Theta_0,\Theta_\infty)\mapsto (1/u(x),\Theta_\infty-1,\Theta_0+1)$ is an involution.
\item
\emph{Rotation}:  if $u(x)$ satisfies \eqref{eq:PIII}--\eqref{eq:Thetas-m-n}, then
$u(\cdot)\mapsto R[u](x):=-\ii u(-\ii x)$ satisfies \eqref{eq:PIII} with modified parameters $R:\Theta_0\mapsto \Theta_0=n+m$ and $R:\Theta_\infty\mapsto 2-\Theta_\infty=n-m+1$ (corresponding to swapping $m$ and $n$).  The mapping $R:(u(x),\Theta_0,\Theta_\infty)\mapsto (-\ii u(-\ii x),\Theta_0,2-\Theta_\infty)$ is the generator of a cyclic symmetry group of order $4$.  Note that $R^2$ fixes the parameters $(\Theta_0,\Theta_\infty)$ in \eqref{eq:PIII} but maps the solution $u(x)$ to its odd reflection $-u(-x)$.
\end{itemize}
%Of course two rotations is equivalent to odd reflection.  Odd reflection and inversion are involutions.
A nontrivial symmetry is the following B\"acklund transformation $u(x)\mapsto\widehat{u}(x)$, which was discovered by Gromak \cite{Gromak73}:
\begin{equation}
\widehat{u}(x):=\frac{xu'(x)+2xu(x)^2+2x-2(1-\Theta_\infty)u(x)-u(x)}{u(x)\cdot(xu'(x)+2xu(x)^2+2x+2\Theta_0u(x)+u(x))}
\label{eq:Backlund-n}
\end{equation}
solves \eqref{eq:PIII} for modified parameters $\Theta_0\mapsto\widehat{\Theta}_0:=\Theta_0+1=(n+1)+m$ and $\Theta_\infty\mapsto\widehat{\Theta}_\infty:=\Theta_\infty-1=m-(n+1)+1$, which amounts to incrementing $n$ for fixed $m$.

\begin{prop} Suppose now that \eqref{eq:PIII} has a solution $u(x)$ that is rational.  Then either $m\in\mathbb{Z}$ or $n\in\mathbb{Z}$ or both.
\end{prop}
\begin{proof}  
Indeed, assuming $u(x)=ax^p+O(x^{p-1})$ as $x\to\infty$ for $p\in\mathbb{Z}$ and $a\neq 0$, from \eqref{eq:PIII} we obtain a dominant balance only for $p=0$, yielding (from the last two terms on the right-hand side) $a^4=1$.  Continuing the Laurent expansion to the next order by writing $u(x)=a+bx^{-1}+O(x^{-2})$ as $x\to\infty$ with $a^4=1$, the calculation of $b$ only brings in the remaining terms in \eqref{eq:PIII} that are not proportional to derivatives of $u$, and we find $b=a^2(\Theta_\infty-1)/4-\Theta_0/4$.  Therefore, the sum of all finite residues of the assumed rational solution $u(x)$ must equal $b$ as well.  If $x=0$ is a pole of $u(x)$, then a similar dominant balance argument involving the terms $u''(x)$, $u'(x)^2/u(x)$, $u'(x)/x$, $u(x)^2/x$, and $4u(x)^3$ shows that it must be a simple pole of residue $-\Theta_0$.  Finally, if $x_0\neq 0$ is a pole of $u(x)$, then it must be a simple pole and a dominant balance involving $u''(x)$, $u'(x)^2/u(x)$, and $4u(x)^3$ shows that the residue is either $\tfrac{1}{2}$ or $-\tfrac{1}{2}$.  Letting $k\in\mathbb{Z}$ denote the difference between the number of nonzero poles of $u(x)$ with residues $\tfrac{1}{2}$ and $-\tfrac{1}{2}$, we therefore arrive at the identities
\begin{equation}
\frac{1}{2}k\mp \frac{1}{4}(\Theta_\infty-1)+\frac{1}{4}\Theta_0=
\begin{cases}\Theta_0 ,&\quad \text{if $x=0$ is a pole of $u$}\\
0,& \text{if $x=0$ is not a pole of $u$},
\end{cases}
\label{eq:Thetas-integral}
\end{equation}
where $a^2=\pm 1$.  Using \eqref{eq:Thetas-m-n} then shows that, if $x=0$ is not a pole of $u$, then $a^2=1$ implies $n=k\in\mathbb{Z}$, while $a^2=-1$ implies $m=-k\in\mathbb{Z}$.  On the other hand, if $x=0$ is a pole of $u$, then by inversion symmetry $I[u](x)=1/u(x)$ is a rational solution of \eqref{eq:PIII} analytic at the origin and corresponding to the modified parameters $I:\Theta_0\mapsto m-n$ and $I:\Theta_\infty\mapsto m+n+1$.  Applying \eqref{eq:Thetas-integral} to $I[u]$ with parameters replaced by their modified values then yields the same conclusion as in the case that $u$ is analytic at the origin, namely that $n=k\in\mathbb{Z}$ if $a^2=1$
and $m=-k\in\mathbb{Z}$ if $a^2=-1$.
\end{proof}

This argument shows that each rational solution of \eqref{eq:PIII} tends to one of four nonzero limits, $\pm 1$ or $\pm\ii$, as $x\to\infty$ and hence cannot be an odd function of $x$. Furthermore, it follows from odd reflection symmetry $R^2:u(x)\mapsto -u(-x)$ that for given parameters \eqref{eq:Thetas-m-n} with $m\in\mathbb{Z}$ or $n\in\mathbb{Z}$, the rational solutions come in distinct pairs permuted by odd reflection.\bigskip  

It turns out that if $m\in\mathbb{Z}$ or $n\in\mathbb{Z}$ there indeed exists a rational solution of \eqref{eq:PIII}--\eqref{eq:Thetas-m-n}.  If only one of $m$ and $n$ is integral, then there are exactly two rational solutions, while if both are integral there are exactly four rational solutions.  The existence and precise number of the rational solutions can be established by iterated B\"acklund transformations once the cases of $m=0$ or $n=0$ are analyzed.

Suppose\footnote{Taking $n=0$ in \eqref{eq:PIII}--\eqref{eq:Thetas-m-n} yields the so-called \emph{sine-Gordon reduction}: writing $u(x)=\ee^{-\ii\varphi(x)}$ and setting $n=0$ in \eqref{eq:PIII}--\eqref{eq:Thetas-m-n} gives
\[
\frac{\dd^2\varphi}{\dd x^2}+\frac{1}{x}\frac{\dd\varphi}{\dd x}=\frac{8m}{x}\sin(\varphi) + 8\sin(2\varphi).
\]
}
%
%either $m$ or $n$ (but not both) to be zero in \eqref{eq:PIII}--\eqref{eq:Thetas-m-n} includes the so-called Painlev\'e III ($D_8$) or Painlev\'e III$_3$ equation as a special case.  See \cite{GamayunIL13}.} 
$n=0$ and $m\not\in\mathbb{Z}$.  Then it is obvious that \eqref{eq:PIII}--\eqref{eq:Thetas-m-n} has at least the two distinct rational (equilibrium) solutions $u(x)=\pm 1$.  It is easy to see that there are no other rational solutions in this case. 
Indeed, if we consider the rational solutions that tend to $\pm 1$ as $x\to\infty$ and take $n=0$ in \eqref{eq:PIII}--\eqref{eq:Thetas-m-n}, a simple dominant balance argument shows that these solutions satisfy $u(x)=\pm 1+O(x^{-p})$ as $x\to\infty$ for every positive integer $p$ and hence as $u(x)$ is rational the error terms vanish identically so the exact solutions $u(x)=\pm 1$ are the only ones recovered.  On the other hand, if we consider the rational solutions that tend to $\pm \ii$ as $x\to\infty$ and take $n=0$ in \eqref{eq:Thetas-integral} we find that for some $k\in\mathbb{Z}$ we have
$m=k$ if $x=0$ is a pole of $u$ and $m=-k$ otherwise, both of which contradict the assumption that $m\not\in\mathbb{Z}$.
Similarly if $m=0$ and $n\not\in\mathbb{Z}$, then \eqref{eq:PIII}--\eqref{eq:Thetas-m-n} has  the pair $u(x)=\pm\ii$ as its only rational solutions (this also follows directly using the rotation symmetry generator $R$).  Finally if $m=n=0$ there are precisely four rational solutions:  $u(x)=\pm 1$ and $u(x)=\pm\ii$.  In Section~\ref{sec:Backlund} we use these facts to determine the precise number of rational solutions of \eqref{eq:PIII} for non-integral $m$.\bigskip  

The rational solutions of \eqref{eq:PIII} have been known at least since the paper of Gromak \cite{Gromak73}.  The paper \cite{MilneCB97} is an exhaustive survey of special solutions of the Painlev\'e-III equation that describes the effect of iterating transformations such as \eqref{eq:Backlund-n}, including cataloguing the exact numbers of poles and zeros of the iterates.  This paper also includes complete references on applications of the Painlev\'e-III equation accurate to the date of publication.  Since rational functions are naturally presented as ratios of polynomials, it is compelling to ask whether the polynomials themselves have a simple recurrence formula like \eqref{eq:Backlund-n}.  Such a result was first found for the Painlev\'e-II equation by Yablonskii \cite{Yablonskii59} and Vorob'ev \cite{Vorobev65}, and since then many algebraic representations of these polynomials have been discovered.  For the Painlev\'e-III equation, a representation of rational solutions in terms of special polynomials was first obtained by Umemura \cite[Section 9]{Umemura99}. Clarkson further developed Umemura's scheme; in \cite{Clarkson03} a sequence of functions is defined by setting 
\begin{equation}
s_{-1}(x;m)\equiv s_0(x;m)\equiv 1
\label{eq:s-IC}
\end{equation}
and then using the recurrence relation
\begin{equation}
s_{n+1}(x;m):=\frac{\left(4 x+2m+1\right) s_{n}(x;m)^2- s_{n}(x;m)
   s_{n}'(x;m)- x \left( s_{n}(x;m) s_{n}''(x;m)-
   s_{n}'(x;m)^2\right)}{2s_{n-1}(x;m)},\quad n\in\mathbb{Z}_{\geq 0}.
   \label{eq:s-recurrence}
\end{equation}
It turns out that the denominator is always a factor of the numerator, so the functions $\{s_n(x;m)\}_{n=0}^\infty$ are all \emph{polynomials} in $x$.  Note that comparing with the notation of \cite{Clarkson03,ClarksonLL16}, we have $\mu=m+\tfrac{1}{2}$, $z=2x$, $\beta=2(1-\Theta_\infty)$, and $\alpha=2\Theta_0$.  The result of the scheme is the following.
\begin{prop}[Umemura \cite{Umemura99}, Clarkson \cite{Clarkson03}, Clarkson, Law, and Lin \cite{ClarksonLL16}]
The result of applying the B\"acklund transformation \eqref{eq:Backlund-n} $n$ times to the seed solution $u(x)\equiv 1$ is the function 
\begin{equation}
u(x)=u_n(x;m):=\frac{s_n(x;m-1)s_{n-1}(x;m)}{s_n(x;m)s_{n-1}(x;m-1)},\quad n\in\mathbb{Z}_{\geq 0},%=0,1,2,3,\dots,
\label{eq:un-fraction}
\end{equation}
defined in terms of polynomials $\{s_n(x;m)\}_{n=0}^\infty$ determined by \eqref{eq:s-IC}--\eqref{eq:s-recurrence}.
Furthermore, $u_n(x;m)$ is the unique rational solution of \eqref{eq:PIII} for parameters \eqref{eq:Thetas-m-n} for which $u_n(x;m)\to 1$ as $x\to\infty$.
\label{prop:polynomial-recurrence}
\end{prop}  
The family of rational solutions $u_n(x;m)$ can be extended to negative integral values of $n$ through the inversion symmetry $I$:
\begin{equation}\label{negn}
u_{-n}(x;m):=Iu_n(x;m)=\frac{1}{u_n(x;m)},\quad n\in\mathbb{Z}_{\geq 0}.%=0,1,2,3,\dots.
\end{equation}
It obviously holds that $u_{-n}(x;m)\to 1$ as $x\to\infty$, so the family captures every rational solution of the Painlev\'e-III equation \eqref{eq:PIII} that tends to $1$ as $x\to\infty$.  It is clearly sufficient to study the family for integers $n\ge 0$.  Without loss of generality we may also restrict attention to values of $m$ in the closed right half-plane:  $\mathrm{Re}(m)\ge 0$; indeed, composing inversion $I$ with two rotations,
\begin{equation}\label{negm}
u_n(x;-m)=R\circ I\circ R u_n(x;m)=\frac{1}{u_n(-x;m)}.
\end{equation}
Moreover, unless $m\in\mathbb{Z}$, studying
the family $\{u_n(x;m)\}$ of rational solutions tending to $1$ as $x\to\infty$ captures \emph{all} rational solutions of \eqref{eq:PIII} because $R^2u_n(x;m)=-u_n(-x;m)$ is the rational solution of exactly the same Painlev\'e-III equation \eqref{eq:PIII} tending to $-1$ as $x\to\infty$.  If both $n$ and $m$ are integers, we may access the rotation symmetry generator $R$ to finally exhaust all rational solutions of \eqref{eq:PIII}.
\begin{rem}
It has been proven by Clarkson, Law, and Lin \cite[Theorem 4.6]{ClarksonLL16} that if $m+\tfrac{1}{2}\in\mathbb{Z}$, then
for $n>|m+\tfrac{1}{2}|$, $s_n$ has $\tfrac{1}{2}n(n+1)$ roots, $s_n$ vanishes to order $\tfrac{1}{2}(n-|m+\tfrac{1}{2}|)(n-|m+\tfrac{1}{2}|+1)$ at the origin, and all remaining roots are simple and nonzero.  This shows that when $m$ is a half-integer and $n$ is large, $s_n$ has a root of order $O(n^2)$ at the origin and merely $O(n)$ simple nonzero roots.  This result implies that when $m=\tfrac{1}{2},\tfrac{3}{2}, \tfrac{5}{2},\dots$, $u_n(x;m)$ has a simple zero at the origin, while when $m=-\tfrac{1}{2},-\tfrac{3}{2},-\tfrac{5}{2},\dots$, $u_n(x;m)$ has a simple pole at the origin.
\end{rem}

\subsection{Riemann-Hilbert problem formulation and main result}
\label{sec:main}
The purpose of this paper is to take the first steps toward understanding the family $\{u_n(x;m)\}_{n=0}^\infty$ of rational solutions of the Painlev\'e-III equation \eqref{eq:PIII} from the perspective of mathematical analysis, a goal which essentially begs the question of how $u_n(x;m)$ behaves when $n$ is large and how the result depends on $(x,m)\in\mathbb{C}^2$.  In Section~\ref{sec:numerics-and-scalings} we present the results of several plots of poles and zeros of $u_n(x;m)$ set in the context of a formal scaling analysis of the Painlev\'e-III equation in the limit of large (integral) $n$.  These results suggest numerous remarkable phenomena that can occur in this limit, but whose proofs would require other methods.  The issue at hand is that the methods described above for constructing the rational function $u_n(x;m)$ all involve some sort of iteration, producing formul\ae\ that generally become more complicated as $n$ increases.  The recurrence \eqref{eq:s-recurrence} is preferable to iteration of the B\"acklund transformation \eqref{eq:Backlund-n} in the sense that it takes advantage of explicit factorization of the numerator and denominator polynomials in the rational function $u_n(x;m)$, but it is a recurrence nonetheless.  Kajiwara and Masuda \cite{KajiwaraM99} found a way to express (essentially) the polynomial $s_n(x;m)$ in closed form via Wronskian determinants of polynomials obtained from an elementary generating function.  However, unlike certain determinantal representations of Hankel type appearing in the theory of the rational solutions of the Painlev\'e-II \cite{BertolaB15} and (for the ``generalized Hermite'' rational solutions) Painlev\'e-IV \cite{Buckingham17} equations, the determinants of Kajiwara and Masuda do not appear to be amenable to asymptotic analysis in the limit of large $n$ (in which the size of the determinant grows without bound).   The lack of an analytically tractable formula for $u_n(x;m)$ is the main problem that we address and solve in this paper.  After a review of the isomonodromy theory of the Painlev\'e-III equation in Section~\ref{sec:isomonodromy-review}, in Sections~\ref{sec:direct-monodromy} and \ref{sec:Schlesinger} we construct a Riemann-Hilbert representation of the function $u_n(x;m)$ that can be used \cite{BothnerM18} to successfully analyze the rational solution for large $n$.  To formulate this problem here in the introduction, given a nonzero $x\in\mathbb{C}$ with $-\pi<\mathrm{Arg}(x)<\pi$, let $L=\LInftyRed\cup\LZeroRed\cup\LInftyBlue\cup\LZeroBlue$ be a contour in the complex $\lambda$-plane consisting of four arcs with the following properties.  There is an intersection point $\lNaught$ such that:
\begin{itemize}
\item $\LInftyRed$ originates from $\lambda=\infty$ in such a direction that $\ii x\lambda$ is negative real and terminates at $\lambda=\lNaught$, 
$\LZeroRed$ begins at $\lambda=\lNaught$ and terminates at $\lambda=0$ in a direction such that $-\ii x\lambda^{-1}$ is negative real, and the net increment of the argument of $\lambda$ along $\LInftyRed\cup\LZeroRed$ is
\begin{equation}
\Delta\arg(\squareurblack)=2\mathrm{Arg}(x)-2\pi\mathrm{sgn}(\mathrm{Im}(x)).
\label{eq:increment-argument-red}
\end{equation}
\item $\LInftyBlue$ originates from $\lambda=\infty$ in such a direction that $-\ii x\lambda$ is negative real and terminates at $\lambda=\lNaught$,
$\LZeroBlue$ begins at $\lambda=\lNaught$ and terminates at $\lambda=0$ in a direction such that $\ii x\lambda^{-1}$ is negative real, and the net increment of the argument of $\lambda$ along $\LInftyBlue\cup\LZeroBlue$ is
\begin{equation}
\Delta\arg(\squarellblack)=2\mathrm{Arg}(x).
\label{eq:increment-argument-blue}
\end{equation}
\item The arcs $\LInftyRed$, $\LZeroRed$, $\LInftyBlue$, and $\LZeroBlue$ do not otherwise intersect.
\end{itemize}
See Figure~\ref{fig:jump-contour} below for an illustration.  Consider now the following problem.
\begin{rhp}  Given parameters $m\in\mathbb{C}$ and $n\in\mathbb{Z}$ as well as $x\in\mathbb{C}\setminus\{0\}$ with $-\pi<\mathrm{Arg}(x)<\pi$, let $L$ denote an $x$-dependent contour as above, and seek a $2\times 2$ matrix function $\mathbf{Y}(\lambda)=\mathbf{Y}^{(n)}(\lambda;x,m)$ with the following properties:  
\begin{itemize}
\item[(1)]\textbf{Analyticity:}  $\mathbf{Y}(\lambda)$ is analytic in $\lambda$ in the domain $\lambda\in\mathbb{C}\setminus L$.  It takes continuous boundary values on $L\setminus\{0\}$ from each maximal domain of analyticity.  
\item[(2)]\textbf{Jump conditions:}  The boundary values $\mathbf{Y}_\pm(\lambda)$ are related on each arc of $L$ by the following formul\ae:
\begin{equation}
%\begin{split}
\mathbf{Y}_+(\lambda)
%&=\mathbf{Y}_-(\lambda)\begin{bmatrix}
%1 & \displaystyle\frac{-2\pi\ii \ee^{\ii\pi m/2}\OurPower{\lambda}{-(m+1)}}{\Gamma(-\tfrac{1}{2}m-\tfrac{1}{4})\Gamma(-\tfrac{1}{2}m+\tfrac{1}{4})}\lambda^n\ee^{\ii x(\lambda-\lambda^{-1})} \\0 & 1\end{bmatrix},\quad \lambda\in \LZeroRed\\
%&
=\mathbf{Y}_-(\lambda)\begin{bmatrix}
1 & \displaystyle -\frac{\sqrt{2\pi}\OurPower{\lambda}{-(m+1)}}{\Gamma(\tfrac{1}{2}-m)}\lambda^n\ee^{\ii x(\lambda-\lambda^{-1})}\\
0 & 1
\end{bmatrix},\quad \lambda\in \LZeroRed
%\end{split}
\label{eq:Yjump-1}
\end{equation}
\begin{equation}
%\begin{split}
\mathbf{Y}_+(\lambda)
%&=\mathbf{Y}_-(\lambda)\begin{bmatrix}
%1 & \displaystyle\frac{2\pi\ii \ee^{\ii\pi m/2}\OurPower{\lambda}{-(m+1)}}{\Gamma(-\tfrac{1}{2}m-\tfrac{1}{4})\Gamma(-\tfrac{1}{2}m+\tfrac{1}{4})}\lambda^n\ee^{\ii x(\lambda-\lambda^{-1})} \\0 & 1\end{bmatrix},\quad \lambda\in \LInftyRed\\
%&
=\mathbf{Y}_-(\lambda)\begin{bmatrix}
1 & \displaystyle \frac{\sqrt{2\pi}\OurPower{\lambda}{-(m+1)}}{\Gamma(\tfrac{1}{2}-m)}\lambda^n\ee^{\ii x(\lambda-\lambda^{-1})}\\
0 & 1
\end{bmatrix},\quad \lambda\in \LInftyRed
%\end{split}
\label{eq:Yjump-2}
\end{equation}
\begin{equation}
%\begin{split}
\mathbf{Y}_+(\lambda)
%&=\mathbf{Y}_-(\lambda)\begin{bmatrix}
%1 & 0 \\\displaystyle\frac{2\pi\ii \ee^{-\ii\pi m/2}(\OurPower{\lambda}{(m+1)/2})_+(\OurPower{\lambda}{(m+1)/2})_-}{\Gamma(\tfrac{1}{2}m+\tfrac{3}{4})\Gamma(\tfrac{1}{2}m+\tfrac{5}{4})}\lambda^{-n}\ee^{-\ii x(\lambda-\lambda^{-1})} & 1\end{bmatrix},\quad \lambda\in \LInftyBlue\\
%&
=\mathbf{Y}_-(\lambda)\begin{bmatrix}
1 & 0\\
\displaystyle \frac{\sqrt{2\pi}(\OurPower{\lambda}{(m+1)/2})_+(\OurPower{\lambda}{(m+1)/2})_-}{\Gamma(\tfrac{1}{2}+m)}\lambda^{-n}\ee^{-\ii x(\lambda-\lambda^{-1})} & 1
\end{bmatrix},\quad \lambda\in \LInftyBlue
%\end{split}
\label{eq:Yjump-3}
\end{equation}
\begin{equation}
%\begin{split}
\mathbf{Y}_+(\lambda)
%&=\mathbf{Y}_-(\lambda)\begin{bmatrix}
%-\ee^{2\pi\ii m}& 0 \\\displaystyle\frac{2\pi\ii \ee^{-\ii\pi m/2}(\OurPower{\lambda}{(m+1)/2})_+(\OurPower{\lambda}{(m+1)/2})_-}{\Gamma(\tfrac{1}{2}m+\tfrac{3}{4})\Gamma(\tfrac{1}{2}m+\tfrac{5}{4})}\lambda^{-n}\ee^{-\ii x(\lambda-\lambda^{-1})} & -\ee^{-2\pi\ii m}\end{bmatrix},\quad \lambda\in \LZeroBlue\\
%&
=\mathbf{Y}_-(\lambda)\begin{bmatrix}
-\ee^{2\pi\ii m} & 0\\
\displaystyle \frac{\sqrt{2\pi}(\OurPower{\lambda}{(m+1)/2})_+(\OurPower{\lambda}{(m+1)/2})_-}{\Gamma(\tfrac{1}{2}+m)}\lambda^{-n}\ee^{-\ii x(\lambda-\lambda^{-1})} & -\ee^{-2\pi \ii m}\end{bmatrix}, \quad \lambda\in \LZeroBlue.
%\end{split}
\label{eq:Yjump-4}
\end{equation}
\item[(3)]\textbf{Asymptotics:}  $\mathbf{Y}(\lambda)\to\mathbb{I}$ as $\lambda\to\infty$.  Also,
the matrix function $\mathbf{Y}(\lambda)\OurPower{\lambda}{-(\Theta_0+\Theta_\infty)\sigma_3/2}=\mathbf{Y}(\lambda)\OurPower{\lambda}{-(m+\tfrac{1}{2})\sigma_3}$ has a well-defined limit as $\lambda\to 0$ (the same limit from each side of $L$).
\end{itemize}
\label{rhp:renormalized}
\end{rhp}
Here, $\OurPower{\lambda}{p}$ is notation for a certain well-defined (see Section \ref{sec:normalized-solutions} below) branch of the power function with its branch cut on the contour $\LZeroBlue\cup\LInftyBlue$, $\sigma_3:=\mathrm{diag}[1,-1]$ denotes a standard Pauli spin matrix, and subscripts $+$/$-$ refer to boundary values taken on the indicated contour from the left/right.
We introduce the expansions 
\begin{equation}
\mathbf{Y}(\lambda)=\mathbb{I}+\mathbf{Y}_1^\infty(x)\lambda^{-1}+O(\lambda^{-2}),\quad\lambda\to\infty;\ \ \ \ \ \mathbf{Y}_1^{\infty}(x)=\big[Y_{1,jk}^{\infty}(x)\big]_{j,k=1}^2
\label{eq:Y-expand-infty}
\end{equation}
and 
\begin{equation}
\mathbf{Y}(\lambda)\OurPower{\lambda}{-(m+\tfrac{1}{2})\sigma_3}=\mathbf{Y}^0_0(x)+O(\lambda),\quad\lambda\to 0;\ \ \ \ \ \mathbf{Y}_0^{0}(x)=\big[Y_{0,jk}^{0}(x)\big]_{j,k=1}^2.
\label{eq:Y-expand-zero}
\end{equation}
Note that the matrix coefficients $\mathbf{Y}_1^\infty(x)$ and $\mathbf{Y}_0^0(x)$ depend parametrically on both $n$ and $m$, as well as $x$.  Then we have the following result.
\begin{theorem}
\label{thm:RH-representation}
The rational solution $u_n(x;m)$ of the Painlev\'e-III equation \eqref{eq:PIII} with parameters $m$ and $n\in\mathbb{Z}$ defined in Proposition~\ref{prop:polynomial-recurrence} and extended to negative integral $n$ by inversion $I$ is given equivalently in terms of the solution $\mathbf{Y}^{(n)}(\lambda;x,m)$ of Riemann-Hilbert Problem~\ref{rhp:renormalized} by
\begin{equation}
u_n(x;m)=\frac{-\ii Y^\infty_{1,12}(x)}{Y^0_{0,11}(x)Y^0_{0,12}(x)}
\label{eq:u-n-from-Y-formula}
\end{equation}
where we have suppressed the parametric dependence on $n\in\mathbb{Z}$ and $m\in\mathbb{C}$ on the right-hand side.
\end{theorem}
 
%In each case the second line of \eqref{eq:Yjump-1}--\eqref{eq:Yjump-4} is obtained from the first with the help of the duplication formula \cite[Eq.~5.5.5]{DLMF} 
%$\Gamma(2z)=\pi^{-1/2}2^{2z-1}\Gamma(z)\Gamma(z+1/2)$.

The proof of this theorem will be completed at the end of Section~\ref{sec:Schlesinger}.  Finally, in Section~\ref{sec:m-half-integer} we study how the Riemann-Hilbert representation degenerates when $m\in\mathbb{Z}+\tfrac{1}{2}$.

%\subsection*{Acknowledgements}
%The work of PDM was supported by the National Science Foundation under grant DMS-1513054.  The authors are grateful to P. Clarkson and W. Van Assche for useful conversations.

\section{Numerical Observations and Formal Scaling Theory}
\label{sec:numerics-and-scalings}
%\textcolor{magenta}{Weave the numerics in together with the scaling formalism discussion, and put it into a proper section after the introduction.  Talk about how the large $n$ picture appears to vary with bounded $m$, in particular what seems to happen when $m$ is near a half-integer.}  
%
\subsection{Scaling analysis}
Eliminating $\Theta_0$ and $\Theta_\infty$ in favor of $m$ and $n$ by \eqref{eq:Thetas-m-n}, the Painlev\'e-III equation \eqref{eq:PIII} becomes
\begin{equation}
\frac{\dd^2u}{\dd x^2}=\frac{1}{u}\left(\frac{\dd u}{\dd x}\right)^2-\frac{1}{x}\frac{\dd u}{\dd x} + 
\frac{4(n+m) u^2 + 4(n-m)}{x} + 4u^3-\frac{4}{u}.
\label{eq:PIII-again}
\end{equation}
Considering $m$ fixed and $n$ large, we introduce a new independent variable by the scaling $x=ny$, and then to further zoom in on the neighborhood of a particular point $y_0$ we set $y=y_0+w/n$.  A simple calculation then shows that if we set $\lNaught(w):=-\ii u(x)=-\ii u(ny_0+w)$, %$V(w):=u(x)=u(ny_0+w)$, 
\eqref{eq:PIII-again} becomes
\begin{equation*}
%\frac{\dd^2V}{\dd w^2}=\frac{1}{V}\left(\frac{\dd V}{\dd w}\right)^2 +\frac{4}{y_0}(V^2+1)+4V^3-\frac{4}{V} + O(n^{-1})
\frac{\dd^2\lNaught}{\dd w^2}=\frac{1}{\lNaught}\left(\frac{\dd \lNaught}{\dd w}\right)^2+\frac{4\ii}{y_0}(\lNaught^2-1)-4\lNaught^3+\frac{4}{\lNaught} + O(n^{-1})
\end{equation*}
where the final term combines several others all of which are proportional to $n^{-1}$.  Neglecting this formally small term and replacing $\lNaught$ with the symbol $\dot{\lNaught}$ indicating a formal approximation yields an autonomous nonlinear equation parametrized by $y_0\in\mathbb{C}\setminus\{0\}$:
\begin{equation}
%\frac{\dd^2\dot{V}}{\dd w^2}=\frac{1}{\dot{V}}\left(\frac{\dd \dot{V}}{\dd w}\right)^2 +\frac{4}{y_0}(\dot{V}^2+1)+4\dot{V}^3-\frac{4}{\dot{V}}.
\frac{\dd^2\dot{\lNaught}}{\dd w^2}=\frac{1}{\dot{\lNaught}}\left(\frac{\dd \dot{\lNaught}}{\dd w}\right)^2+\frac{4\ii}{y_0}(\dot{\lNaught}^2-1)-4\dot{\lNaught}^3+\frac{4}{\dot{\lNaught}}.
\label{eq:Boutroux-model}
\end{equation}
This model equation admits a first integral:  multiply \eqref{eq:Boutroux-model} through by $\dot{\lNaught}'/\dot{\lNaught}^2$ ($\prime = \dd/\dd w$) and rearrange to obtain
\begin{equation*}
%\frac{\dot{V}'\dot{V}''}{\dot{V}^2}-\frac{(\dot{V}')^3}{\dot{V}^3}=4\left[\frac{1}{y_0}(1+\dot{V}^{-2}) + \dot{V}-\dot{V}^{-3}\right]\dot{V}'
\frac{\dot{\lNaught}'\dot{\lNaught}''}{\dot{\lNaught}^2}-\frac{(\dot{\lNaught}')^3}{\dot{\lNaught}^3}=4\left[\frac{\ii}{y_0}(1-\dot{\lNaught}^{-2})-\dot{\lNaught}+\dot{\lNaught}^{-3}\right]\dot{\lNaught}'
\end{equation*}
which is easily integrated to yield 
\begin{equation*}
%\frac{(\dot{V}')^2}{2\dot{V}^2}=4\left[\frac{1}{y_0}(\dot{V}-\dot{V}^{-1})+\frac{1}{2}\dot{V}^2+\frac{1}{2}\dot{V}^{-2}\right]+C,
\frac{(\dot{\lNaught}')^2}{2\dot{\lNaught}^2}=4\left[\frac{\ii}{y_0}(\dot{\lNaught}+\dot{\lNaught}^{-1})-\frac{1}{2}\dot{\lNaught}^2-\frac{1}{2}\dot{\lNaught}^{-2}\right]+\frac{8C}{y_0^2},
\end{equation*}
where $C$ is a constant of integration.  Therefore,
\begin{equation}
%\left(\frac{\dd \dot{V}}{\dd w}\right)^2=P_4(\dot{V};y_0,C):=4\dot{V}^4+\frac{8}{y_0}\dot{V}^3+2C\dot{V}^2 -\frac{8}{y_0}\dot{V}+4.
\left(\frac{\dd\dot{\lNaught}}{\dd w}\right)^2 = \frac{16}{y_0^2}P(\dot{\lNaught};y_0,C),\quad P(\dot{\lNaught};y_0,C):=-\frac{y_0^2}{4}\dot{\lNaught}^4+\frac{\ii y_0}{2}\dot{\lNaught}^3 + C\dot{\lNaught}^2 + \frac{\ii y_0}{2}\dot{\lNaught}-\frac{y_0^2}{4}.
\label{eq:dotV-ODE}
\end{equation}

Suppose that $y_0$ and $C$ are such that the quartic %$P_4(\dot{V};y_0,C)$ 
$P(\dot{\lNaught};y_0,C)$ has a double root $\dot{\lNaught}=\lNaught_0$; eliminating $C$ between the equations %$P_4(\dot{V};y_0,C)=0$ and $P_4'(\dot{V};y_0,C)=0$ 
$P(\lNaught_0;y_0,C)=0$ and $P'(\lNaught_0;y_0,C)=0$ shows that $\lNaught_0$ is a solution of the quartic equation
\begin{equation}
%y_0V_0^4+V_0^3+V_0-y_0=0.
y_0\lNaught_0^4-\ii \lNaught_0^3+\ii \lNaught_0-y_0=0.
\label{eq:V0-quartic}
\end{equation}
Obviously, %$V_0^2+1$ 
$\lNaught_0^2-1$ is a factor of the left-hand side:  
%$y_0V_0^4+V_0^3+V_0-y_0=(V_0^2+1)(y_0(V_0^2-1)+V_0)$, 
$y_0\lNaught_0^4-\ii \lNaught_0^3+\ii \lNaught_0-y_0=(\lNaught_0^2-1)(y_0(\lNaught_0^2+1)-\ii \lNaught_0)$, so there are four possibilities for double roots of %$P_4(\dot{V};y_0,C)$ 
$P(\dot{\lNaught};y_0,C)$, namely:
\begin{equation}
%V_0=\ii, \quad V_0=-\ii,\quad V_0=-\frac{1}{2y_0}+\sqrt{\frac{1}{4y_0^2}+1},\quad V_0=-\frac{1}{2y_0}-\sqrt{\frac{1}{4y_0^2}+1}.
\lNaught_0= 1,\quad \lNaught_0=-1,\quad \lNaught_0=\lNaught^+_0(y_0):=\frac{\ii}{2 y_0}-\ii\sqrt{\frac{1}{4y_0^2}+1},\quad \lNaught_0=\lNaught^-_0(y_0):=\frac{\ii}{2y_0}+\ii\sqrt{\frac{1}{4y_0^2}+1}.
\label{eq:V0-four-answers}
\end{equation}
Note that since the quartic equation \eqref{eq:V0-quartic} is the same equation as arises upon setting $\dot{\lNaught}=\lNaught_0$ and neglecting derivatives of $\dot{\lNaught}$ in 
\eqref{eq:Boutroux-model},
the four values \eqref{eq:V0-four-answers} are precisely the equilibrium solutions of the differential equation \eqref{eq:Boutroux-model}.
The corresponding values of $C$ are then obtained explicitly from the equation %$P_4'(V_0;y_0,C)=0$
$P'(\lNaught_0;y_0,C)=0$, which is linear in $C$ (and the coefficient of $C$ is nonzero in each case):
\begin{equation}
%C=\frac{2}{y_0V_0}-\frac{6V_0}{y_0}-4V_0^2.
C=-\frac{\ii y_0}{4\lNaught_0}-\frac{3\ii y_0}{4}\lNaught_0+\frac{y_0^2}{2}\lNaught_0^2.
\label{eq:C-formula}
\end{equation}
Thus, whenever $C$ is given by \eqref{eq:C-formula} and $\lNaught_0$ is a root of the quartic equation \eqref{eq:V0-quartic} (equivalently, an equilibrium solution of \eqref{eq:Boutroux-model}), 
\begin{equation*}
%P_4(\dot{V};y_0,C)=4(\dot{V}-V_0)^2(\dot{V}^2+b\dot{V}+c),\quad\text{where}\quad
%b:=2V_0+\frac{2}{y_0},\quad c:=\frac{1}{V_0^2}.
P(\dot{\lNaught};y_0,C)=-\frac{y_0^2}{4}(\dot{\lNaught}-\lNaught_0)^2(\dot{\lNaught}^2+b\dot{\lNaught}+c),\quad\text{where}\quad
b:=2\lNaught_0-\frac{2\ii}{y_0},\quad c:=\frac{1}{\lNaught_0^2}.
\end{equation*}
%In particular, $\dot{V}=V_0$ is an equilibrium solution of \eqref{eq:dotV-ODE} whenever $V_0$ is one of the four values indicated in \eqref{eq:V0-four-answers}.  
%\textcolor{magenta}{The following isn't quite right, I think.  If we want to have a pair of double roots and two simple roots, it looks like it is necessary to have the double root at $\pm \ii$.  Otherwise all roots are (at least) double.  Only the double roots are equilibrium solutions of \eqref{eq:Boutroux-model} that could model the behavior outside of the domain $D$.  Based on $g$-function analysis it seems like what we want is two distinct double roots, i.e., $P_4(\dot{V};y_0,C)$ is a perfect square.  Revisit this later after the $g$-function analysis is worked out in the exterior domain.  
%}
For each fixed $(y_0,C)$ pair, the root locus of %$P_4(\dot{V};y_0,C)$ 
$P(\dot{\lNaught};y_0,C)$ is invariant under %$\dot{V}\mapsto -1/\dot{V}$ 
$\dot{\lNaught}\mapsto 1/\dot{\lNaught}$.  Since %$\pm\ii$ 
$\pm 1$ are individually fixed by this involution while the other two possible double roots listed in \eqref{eq:V0-four-answers} are permuted by this involution, we see that if there exists a double root distinct from %$\ii$ or $-\ii$ 
$1$ or $-1$, then there are two distinct double roots and hence %$P_4(\dot{V};y_0,C)$ 
$P(\dot{\lNaught};y_0,C)$ factors as a perfect square of a quadratic with distinct roots.  If one of the points %$\pm\ii$ 
$\pm 1$ is a double root, then either all four roots coincide, the two remaining roots coalesce at 
%$\mp\ii$ 
$\mp 1$, or the two remaining roots are distinct simple roots that are permuted by the involution. 

\subsection{Experiments and conjectures}
To begin to assess the validity of predictions following from the above formal large-$n$ scaling arguments, we may try to examine a finite number of the functions $u_n(x;m)$, say for $n=0,1,2,\dots,N$, and plot their poles and zeros in $x$.  Since according to Proposition~\ref{prop:polynomial-recurrence}, $u_n(x;m)\to 1$ as $x\to\infty$ and $u_n(x;m)$ is rational in $x$ with simple poles and zeros only, such plots actually convey complete information.  In practice, it is substantially more efficient for large $n$ to implement the polynomial recurrence scheme of Umemura/Clarkson than to directly iterate the B\"acklund transformation \eqref{eq:Backlund-n}.  Therefore, we symbolically compute
a sufficient number of the polynomials $s_n$, which have coefficients rational in $m$.  Then by using rational values\footnote{We observed that if the real or imaginary part of $m$ is irrational then \texttt{NSolve} performs poorly for moderately large $n$.} for the real and imaginary parts of $m$, we may apply the \textit{Mathematica}\footnote{We used \textit{Mathematica} version 11.} routine \texttt{NSolve} with the option \texttt{WorkingPrecision->30} to obtain accurate approximations of the roots.  We then plot separately the roots of the four polynomial factors in the representation \eqref{eq:un-fraction}.  As long as the roots of the factors are simple and distinct, no information is lost in making such a plot; this is known to be the case \cite{Clarkson03,ClarksonLL16} unless $m\in\mathbb{Z}+\tfrac{1}{2}$, in which case for large enough $n$ there is a common root of high order at the origin in all four factors, leading to a high degree of cancellation.  We restrict our numerical calculations of poles and zeros to nonnegative values of $n$ and to $\mathrm{Re}(m)\ge 0$ without loss of generality, compare \eqref{negn} and \eqref{negm}.\smallskip

Since the scaling formalism is based at first on the scaling $x=ny$, it is useful to initially view the plots of poles/zeros of $u_n(x;m)$ in the $y$-plane.  Figures~\ref{fig:1}--\ref{fig:4} study the convergence properties of the pole/zero patterns in the $y$-plane as $n$ increases for several values of $m\in\mathbb{C}$.
\begin{figure}[h]
\begin{center}
\includegraphics[width=0.3\linewidth]{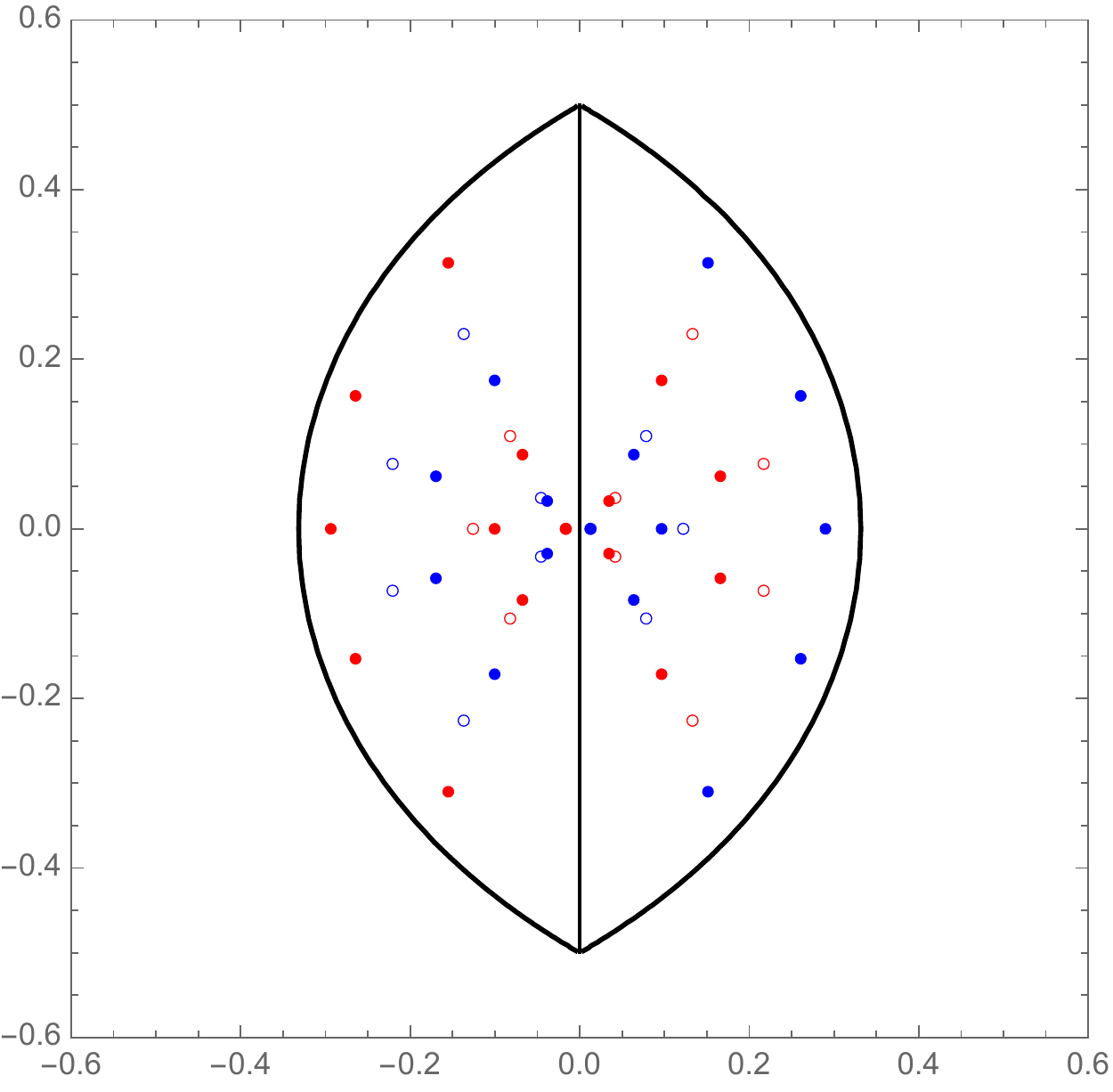}\hspace{0.03\linewidth}%
\includegraphics[width=0.3\linewidth]{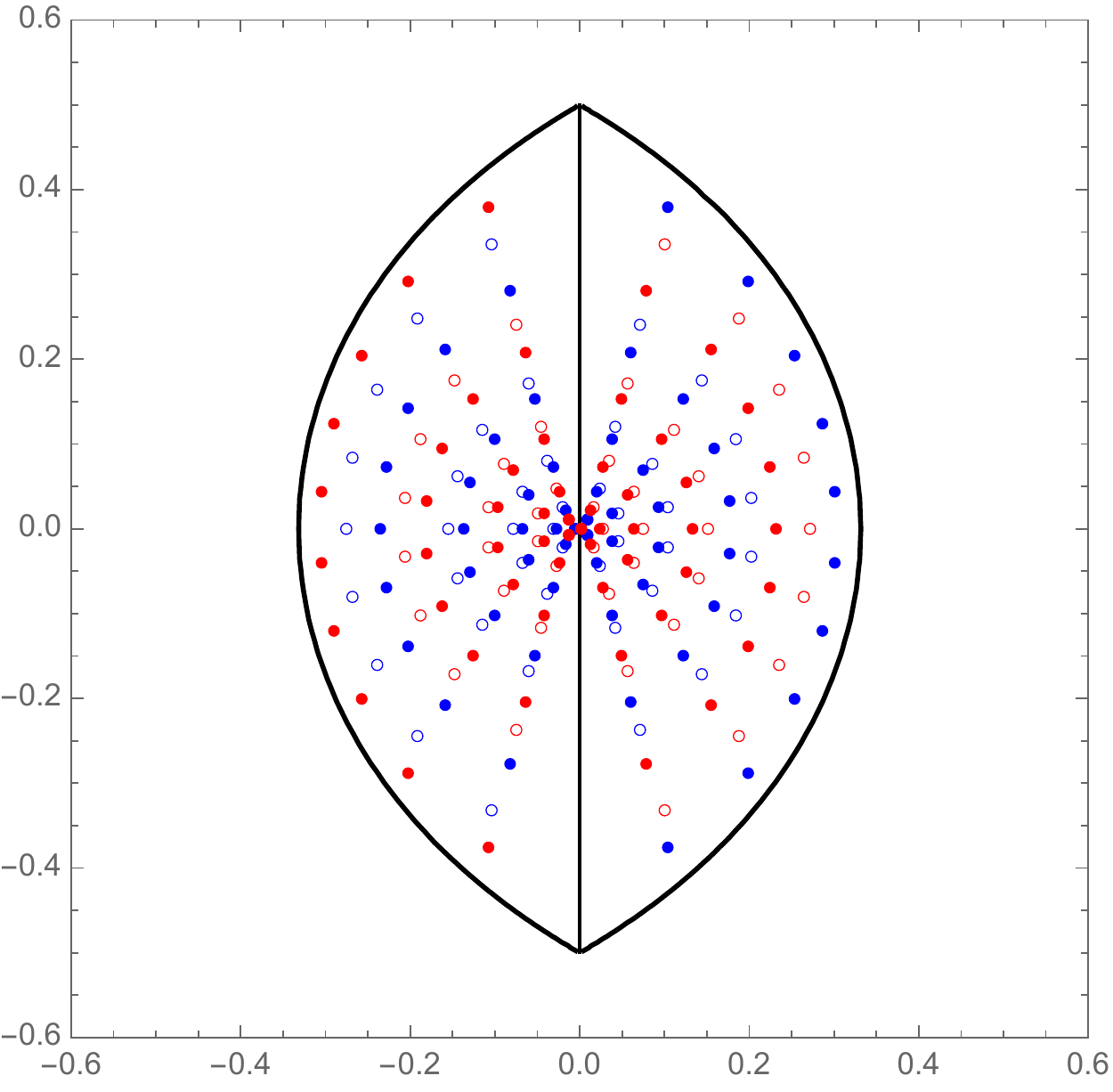}\hspace{0.03\linewidth}%
\includegraphics[width=0.3\linewidth]{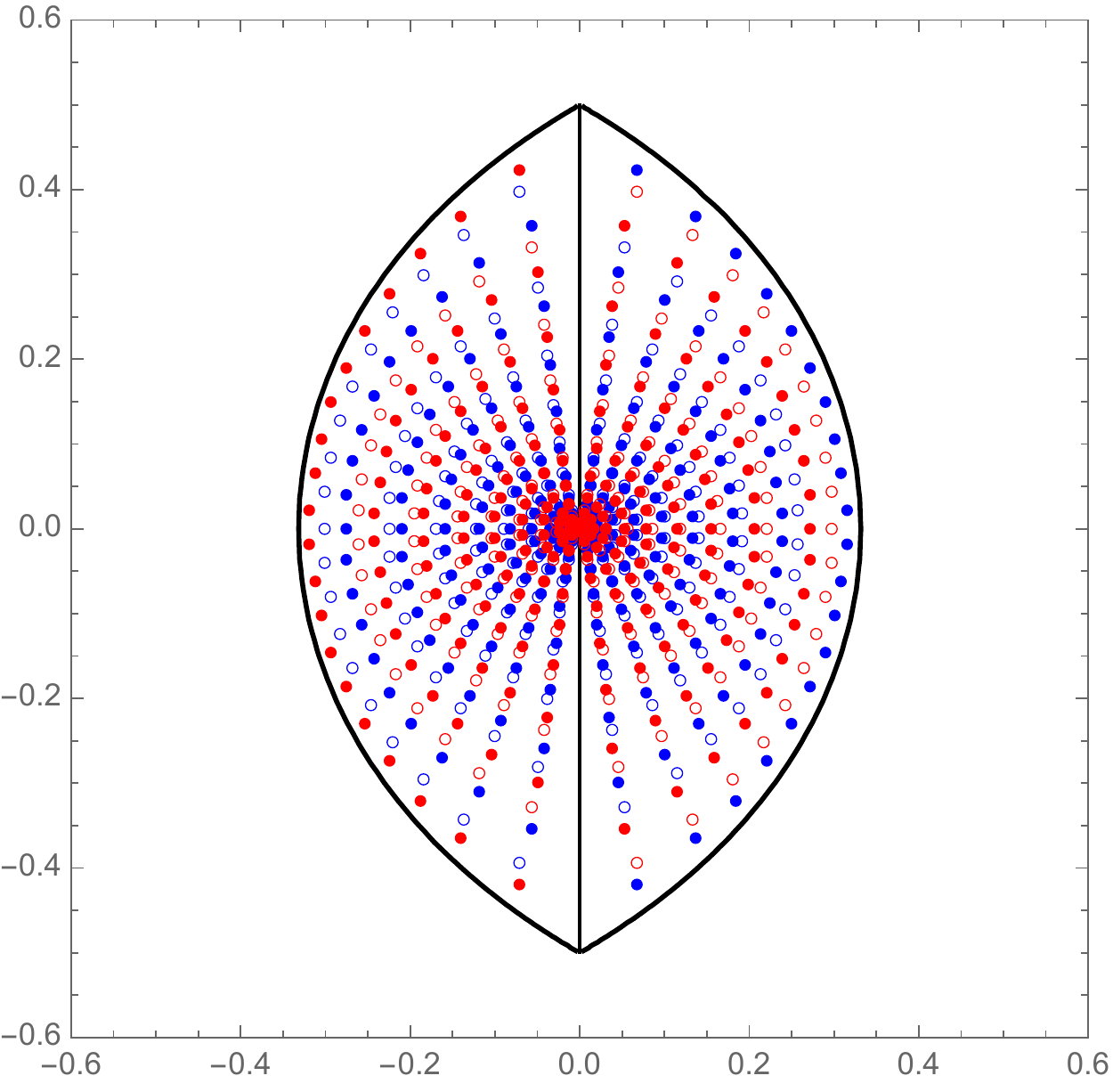}
\end{center}
\caption{Poles of $u_n(x;m)$ (red dots, filled for the roots of $s_n(x;m)$ and unfilled for the roots of $s_{n-1}(x;m-1)$) and zeros of $u_n(x;m)$ (blue dots, filled for the roots of $s_n(x;m-1)$ and unfilled for the roots of $s_{n-1}(x;m)$) rendered in the $y=x/n$-plane for $m=0$.  Left:  $n=5$, center: $n=10$, right: $n=20$.  The black curves are independent of $n$ and $m$ and form the boundaries of two half-eye-shaped regions known to contain the poles and zeros of $u_n(x;m)$ for large $n$ \cite{BothnerM18}.}
\label{fig:1}
\end{figure}
\begin{figure}[h]
\begin{center}
\includegraphics[width=0.3\linewidth]{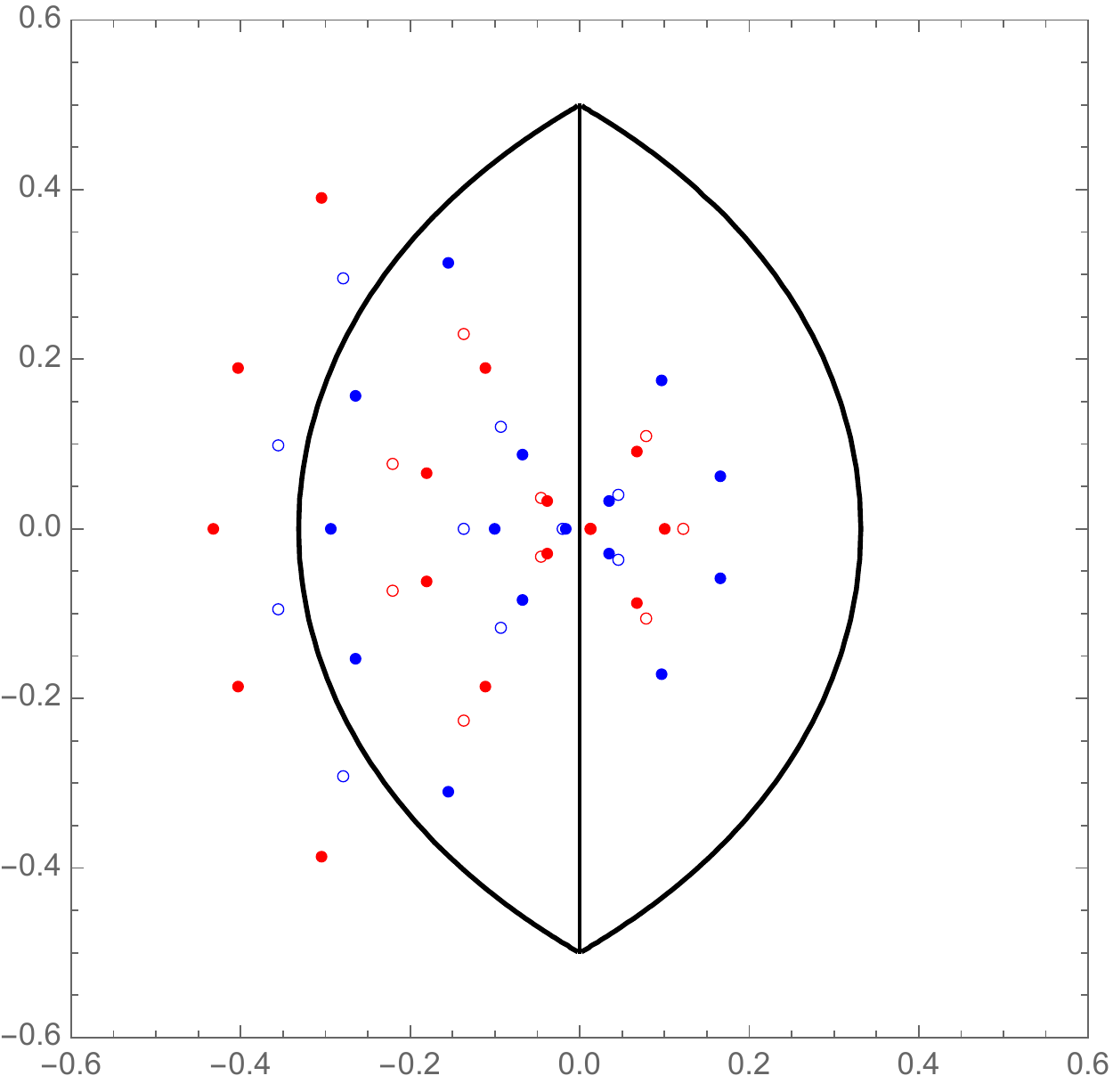}\hspace{0.03\linewidth}%
\includegraphics[width=0.3\linewidth]{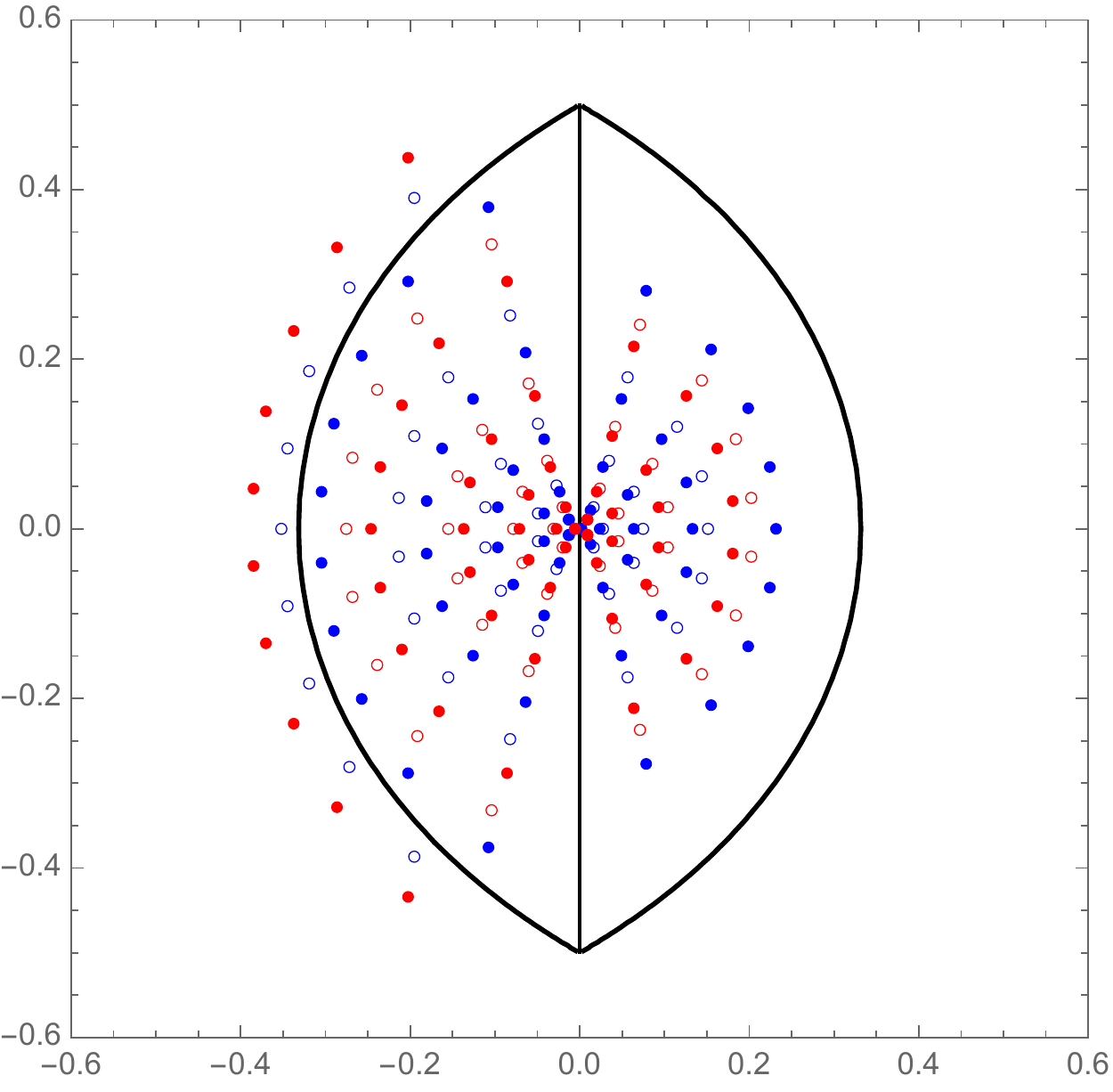}\hspace{0.03\linewidth}%
\includegraphics[width=0.3\linewidth]{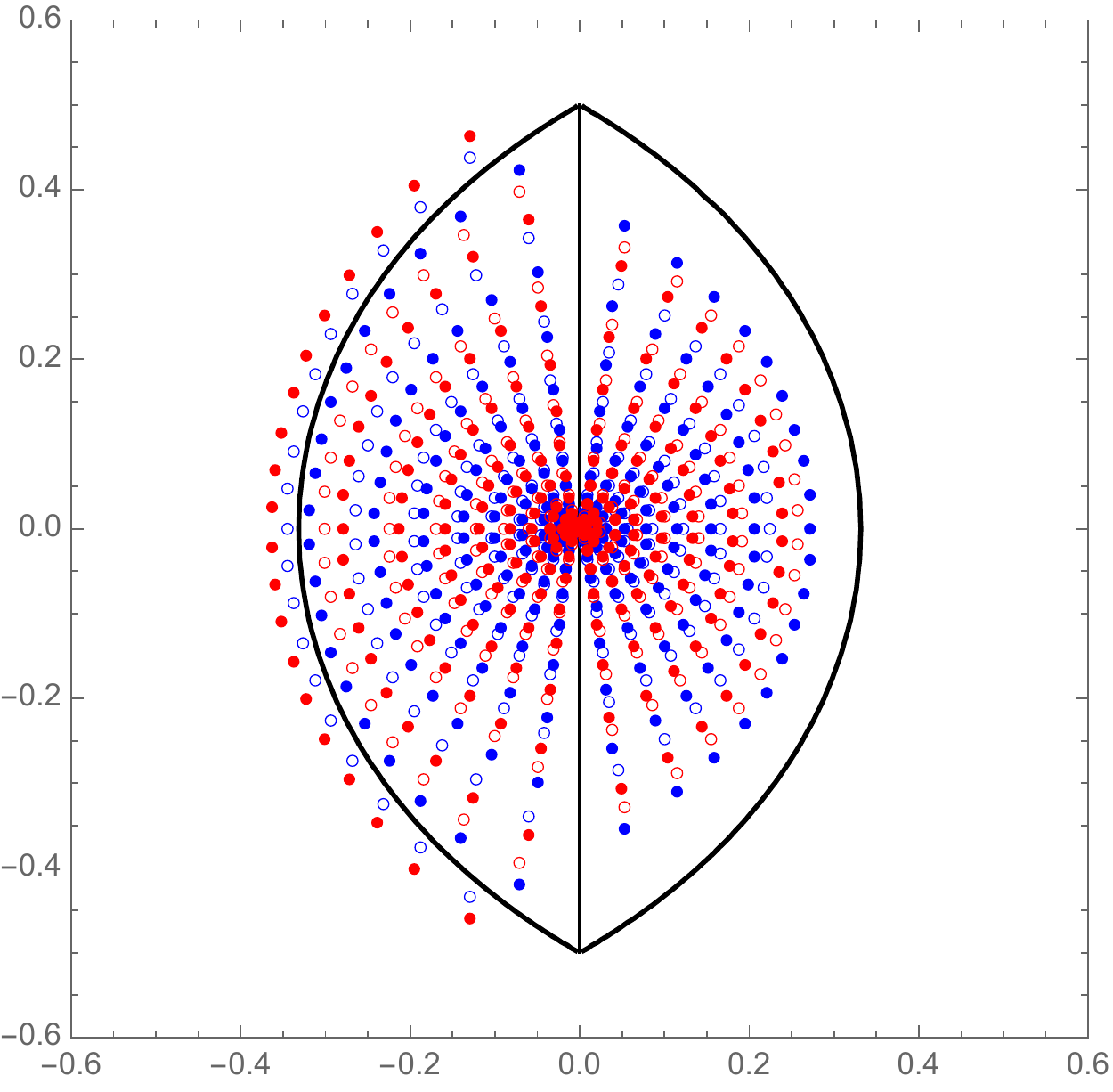}
\end{center}
\caption{As in Figure~\ref{fig:1} but for $m=1$.}
\label{fig:2}
\end{figure}
\begin{figure}[h]
\begin{center}
\includegraphics[width=0.3\linewidth]{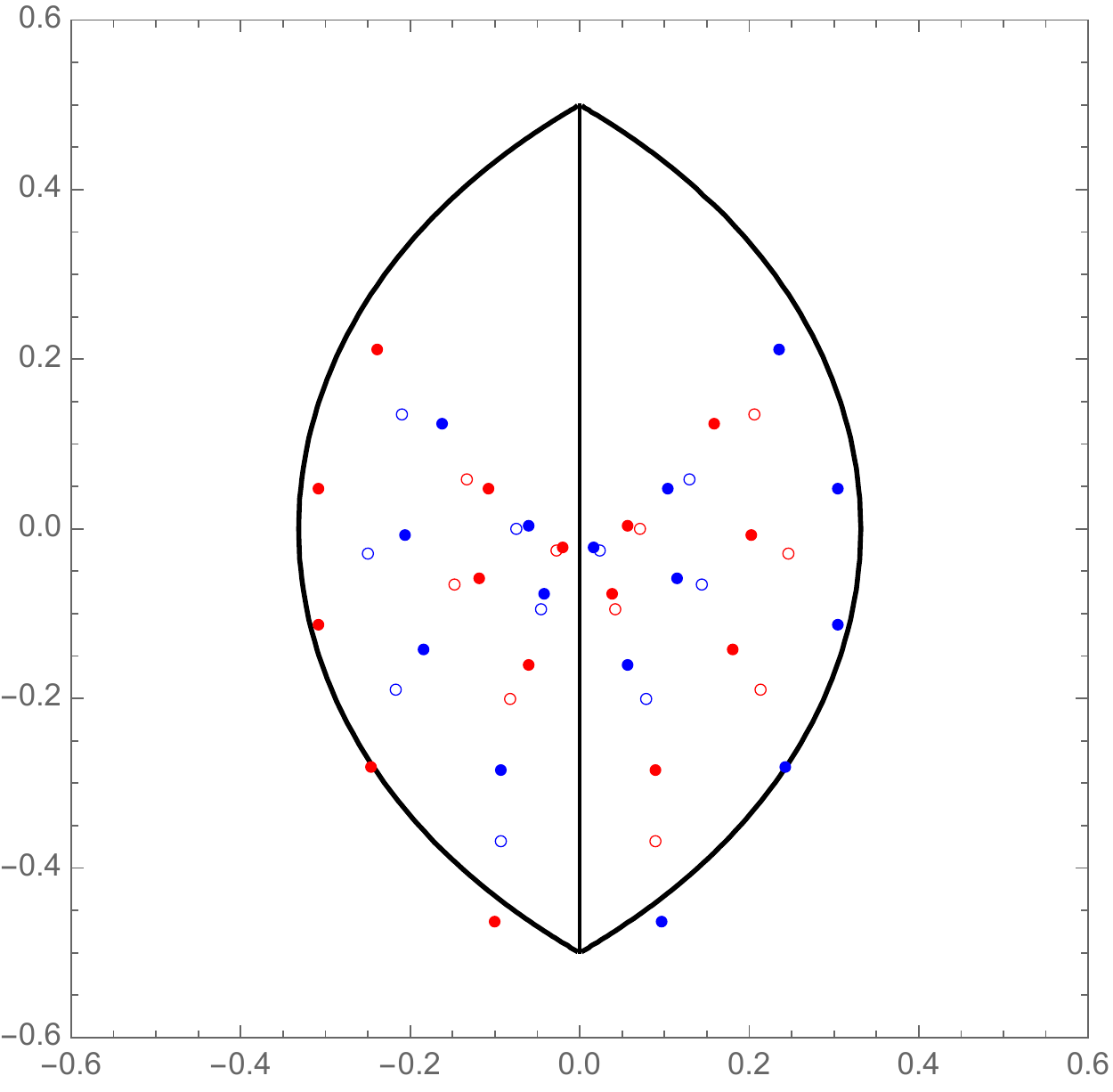}\hspace{0.03\linewidth}%
\includegraphics[width=0.3\linewidth]{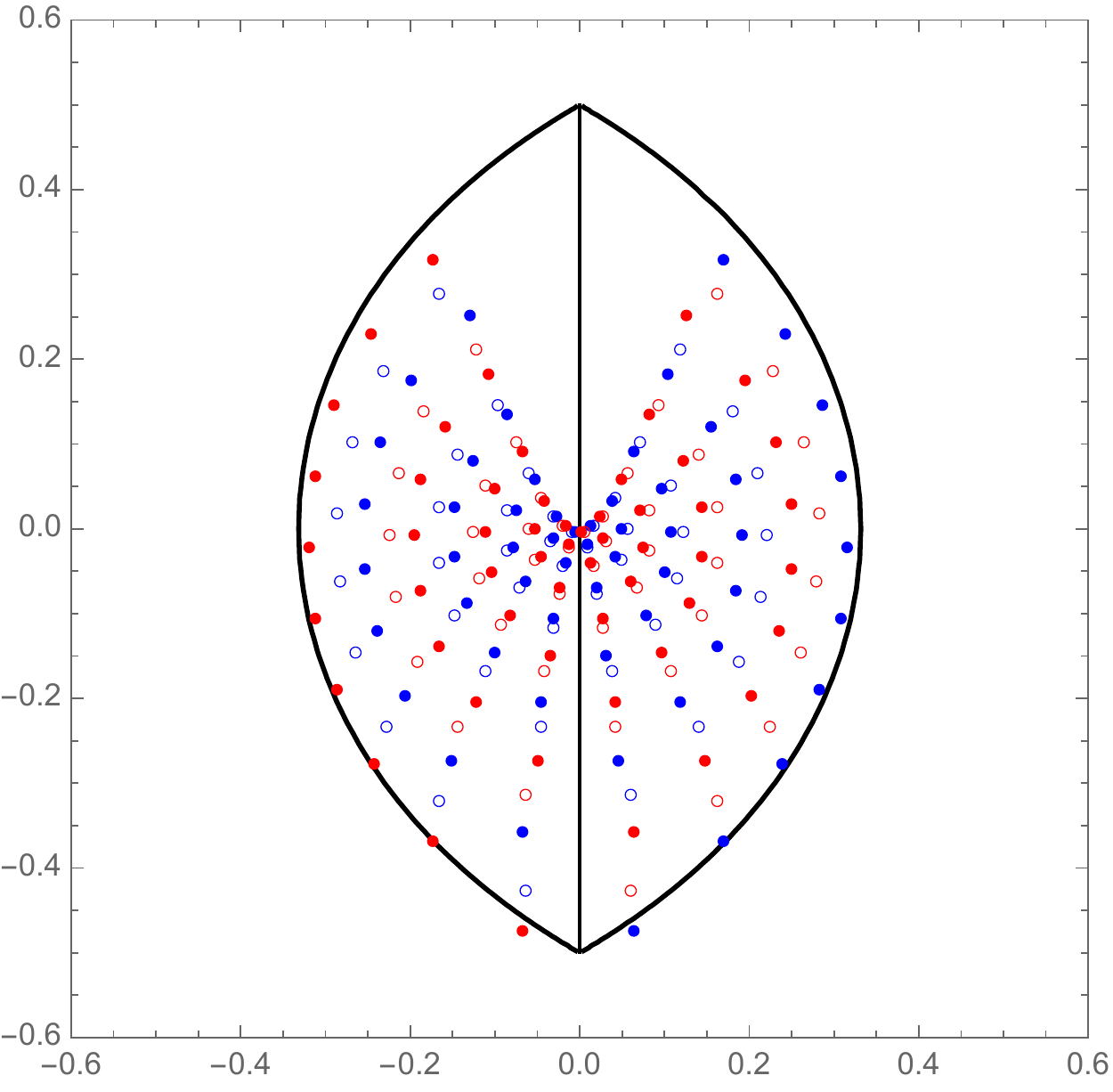}\hspace{0.03\linewidth}%
\includegraphics[width=0.3\linewidth]{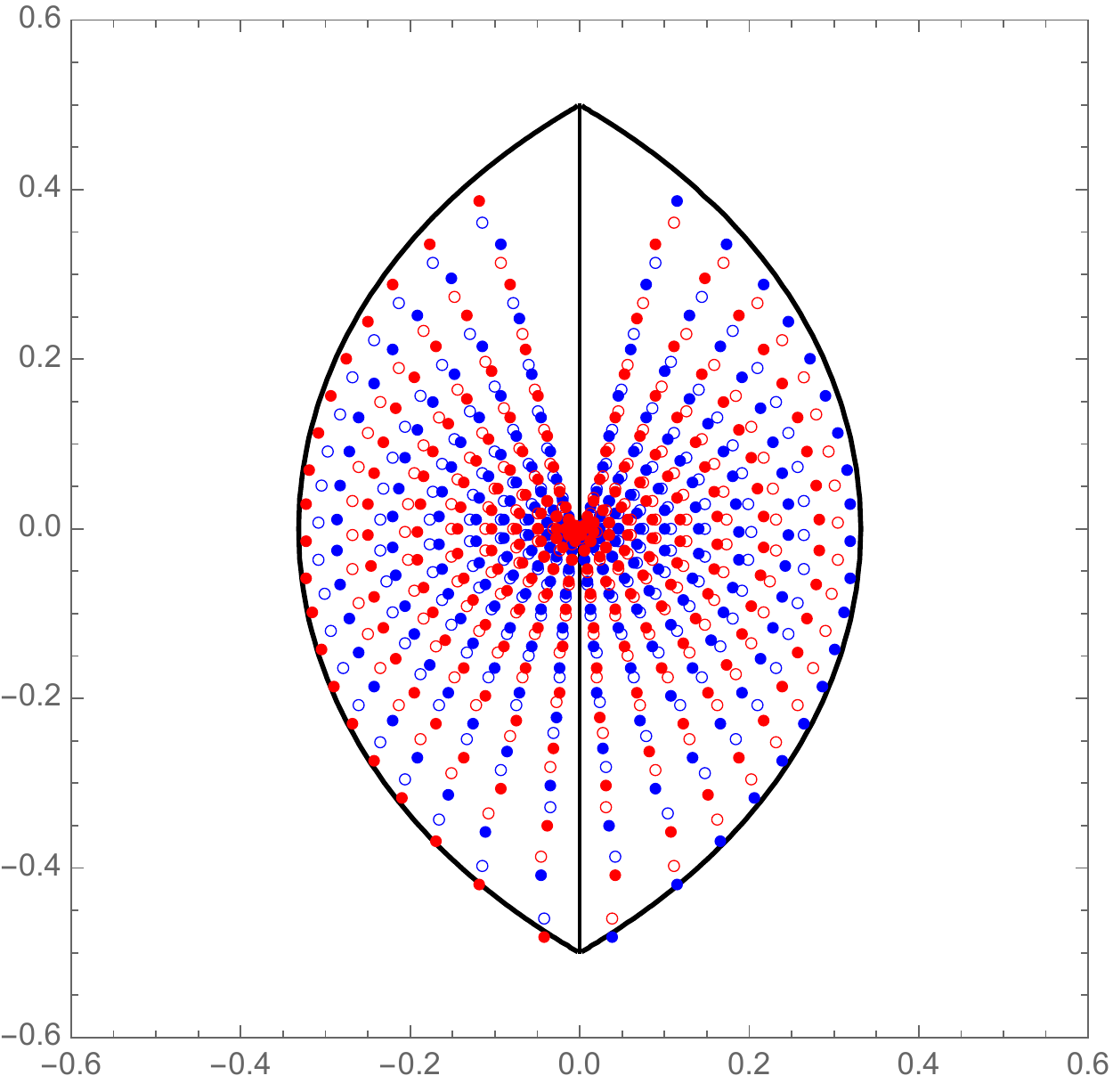}
\end{center}
\caption{As in Figure~\ref{fig:1} but for $m=\tfrac{4}{5}\ii$.}
\label{fig:3}
\end{figure}
The key feature evident in the plots of Figures~\ref{fig:1}, \ref{fig:2}, and \ref{fig:3} is that while there is some variability with the value of $m\in\mathbb{C}$, as $n$ increases the region of the $y$-plane that contains the poles and zeros of $u_n(ny;m)$ appears to stabilize to an eye-shaped domain $E$ that is independent of both $n$ and $m$.  
\begin{figure}[h]
\begin{center}
\includegraphics[width=0.3\linewidth]{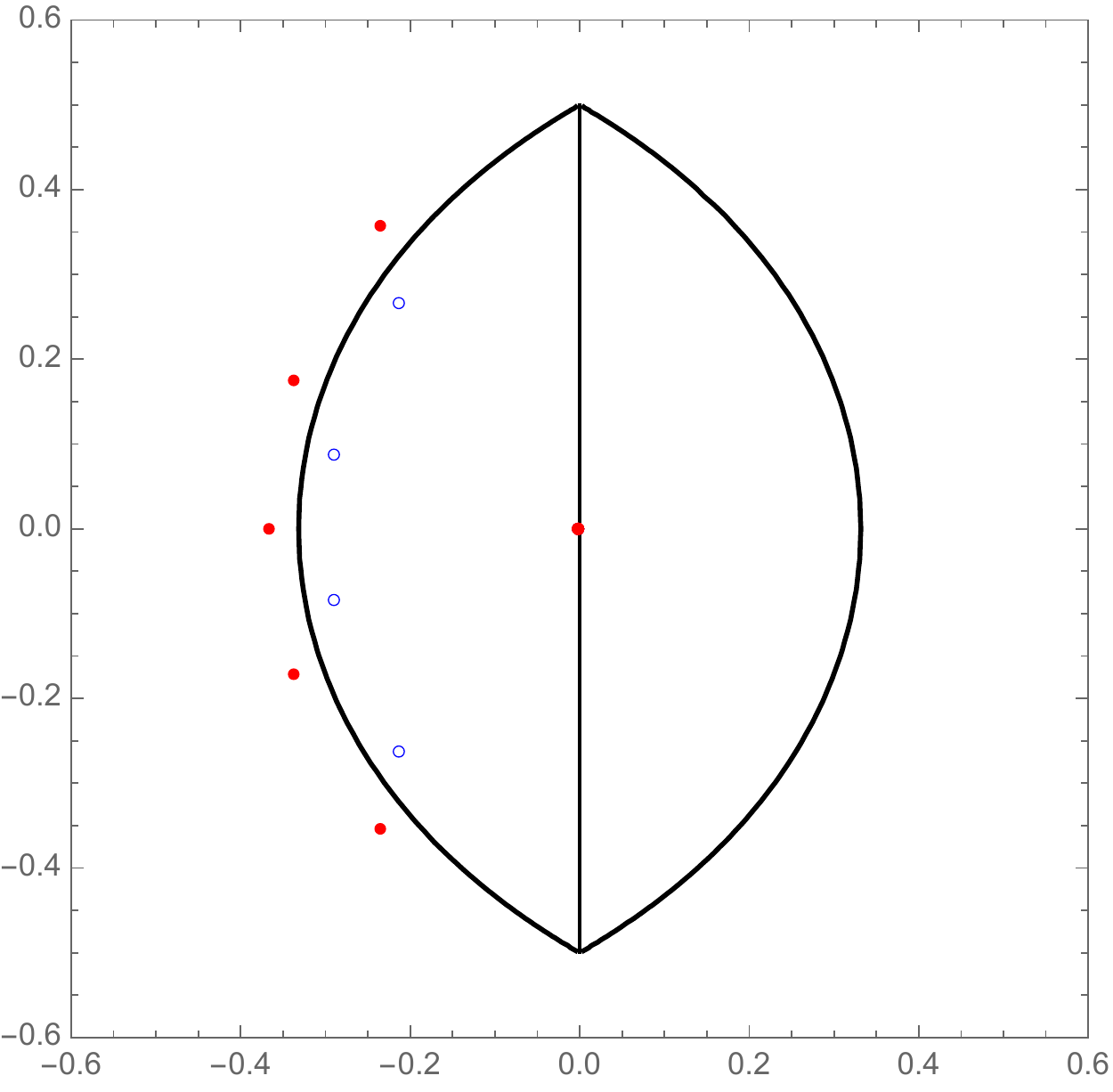}\hspace{0.03\linewidth}%
\includegraphics[width=0.3\linewidth]{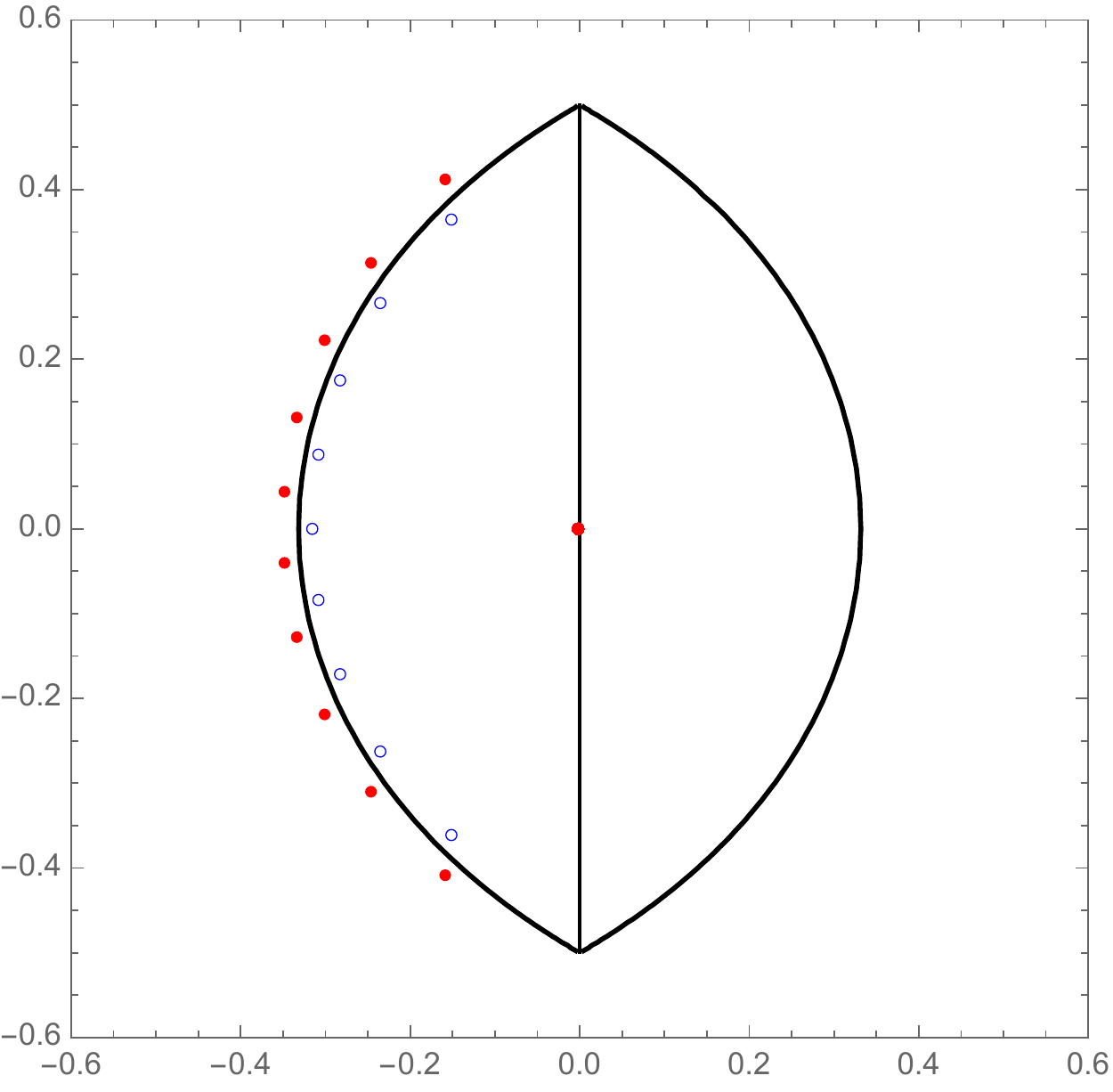}\hspace{0.03\linewidth}%
\includegraphics[width=0.3\linewidth]{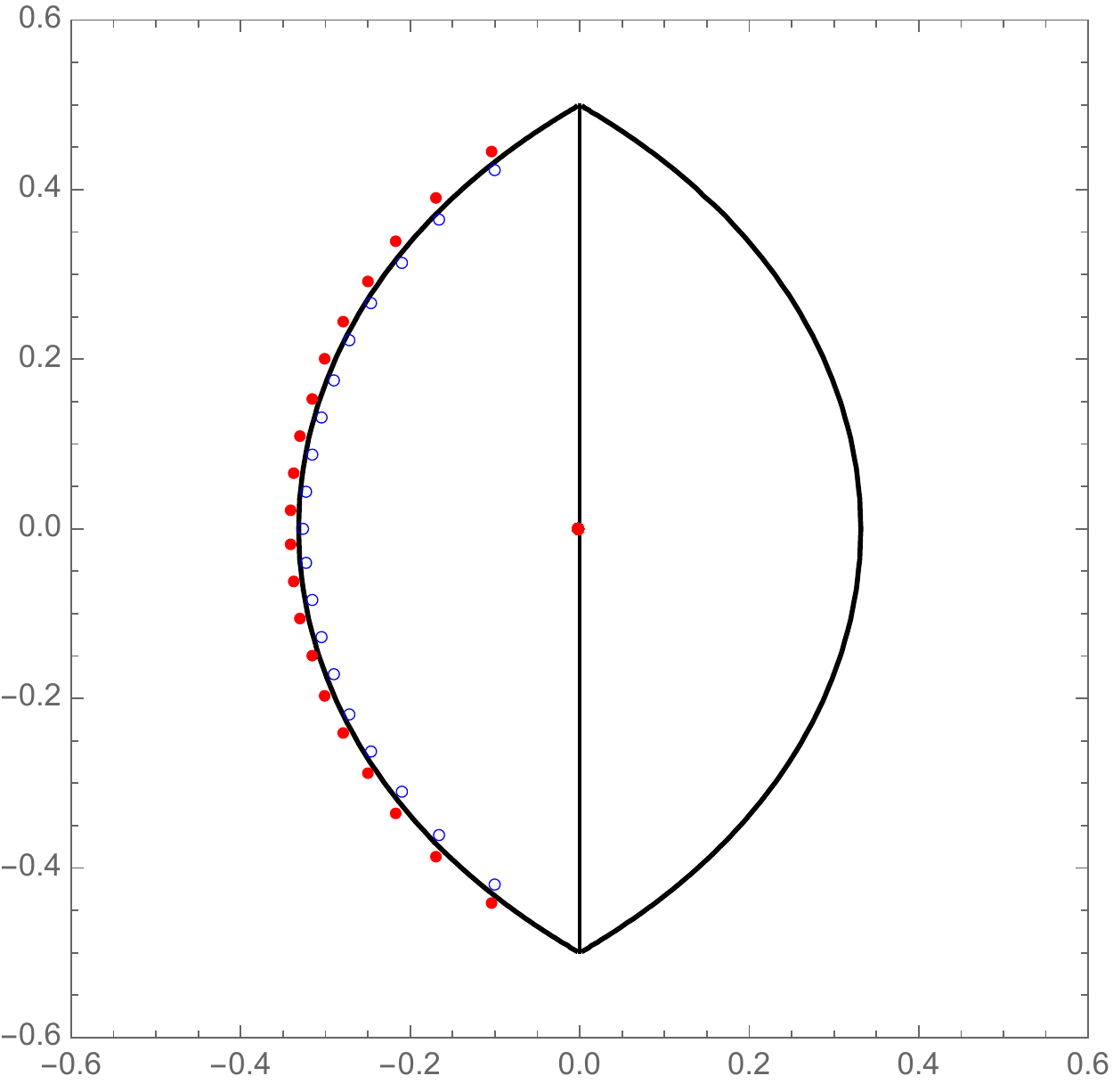}
\end{center}
\caption{As in Figure~\ref{fig:1} but for $m=\tfrac{1}{2}$.  Here we know from \cite{ClarksonLL16} that the apparent pole near the origin in the plots is an artifact of our method of plotting separately the roots of the polynomial factors in \eqref{eq:un-fraction}; in fact $u_n(x;\tfrac{1}{2})$ has a simple zero at $x=0$.}
\label{fig:4}
\end{figure}
Figure~\ref{fig:4} shows a similar convergence study, here for a half-integral value of $m$.  While the poles and zeros seem to move toward the same eye-shaped domain $E$ as $n$ increases, the distribution of poles and zeros within $E$ appears to be completely different than in Figures~\ref{fig:1}--\ref{fig:3}, with poles and zeros concentrating only along one ``eyebrow'' of the eye $E$.\smallskip

Taken together, these figures suggest that $u_n(ny;m)$ may have a well-defined limit as $n\to\infty$ as long as $y$ is restricted to the exterior of $E$.  We are led to formulate the following conjecture.
\begin{conjecture}
Assume that $y$ lies outside of a certain eye-shaped bounded domain $E\subset\mathbb{C}$.  Then 
\begin{equation}
\lim_{n\to\infty}u_n(ny;m)=\ii p_0^+(y),
\end{equation}
where $p_0^+(y)$ is defined by \eqref{eq:V0-four-answers} in which the square root refers to the principal branch.
\label{conjecture:outside}
\end{conjecture}
This conjecture asserts that for $y$ outside of $E$, the quartic $P(\dot{p};y,C)$ has a distinct pair of double roots at $\dot{p}=p_0^\pm(y)$, and that the equilibrium $\dot{p}=p_0^+(y)$ (we are identifying $y$ with the constant $y_0$) is the relevant solution of the autonomous model differential equation \eqref{eq:Boutroux-model}.  Note that $\ii p_0^+(y)$ is independent of the second parameter $m$, and $\ii p_0^+(y)\to 1$ as $y\to\infty$, which is consistent with the fact that for each fixed $n$, $u_n(x;m)\to 1$ as $x\to\infty$.  A suitably precise version of Conjecture~\ref{conjecture:outside} is proven in \cite{BothnerM18} using the Riemann-Hilbert representation of $u_n(x;m)$ presented in Theorem~\ref{thm:RH-representation} formulated in Section~\ref{sec:main}; part of the proof is to correctly specify the domain $E$.  The black curves shown in Figures~\ref{fig:1}--\ref{fig:4} are described in \cite{BothnerM18}; in particular the top and bottom corners of the domain $E$ lie at the points $y=\pm\tfrac{1}{2}\ii$.\bigskip

The asymptotic pattern of poles and zeros of $u_n(x;m)$ is qualitatively similar to that shown in Figure~\ref{fig:4} whenever $m\in\mathbb{Z}+\tfrac{1}{2}$, but different details emerge as $m$ is increased through half-integers as illustrated in Figure~\ref{fig:5}.  
\begin{figure}[h]
\begin{center}
\includegraphics[width=0.3\linewidth]{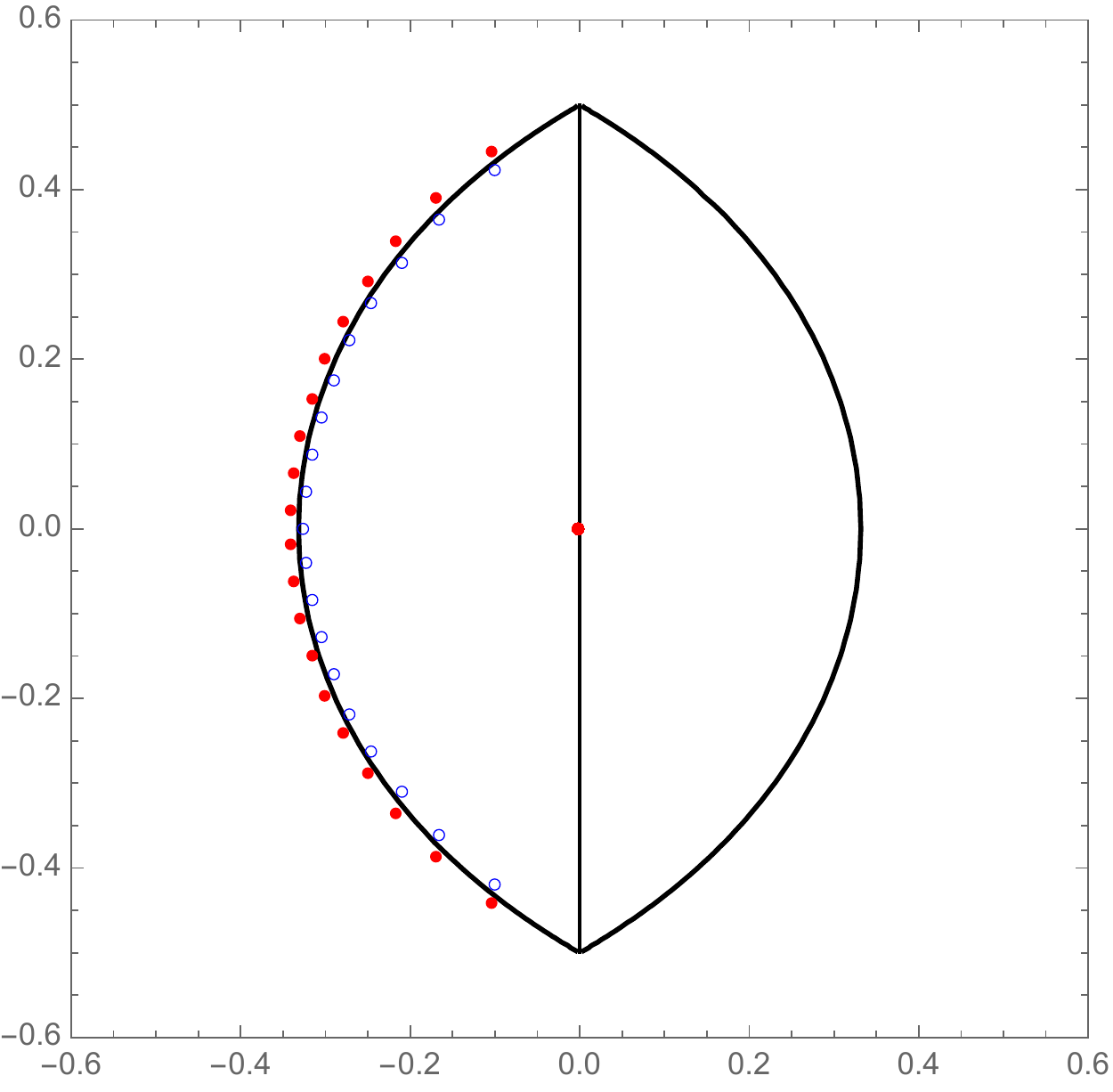}\hspace{0.03\linewidth}%
\includegraphics[width=0.3\linewidth]{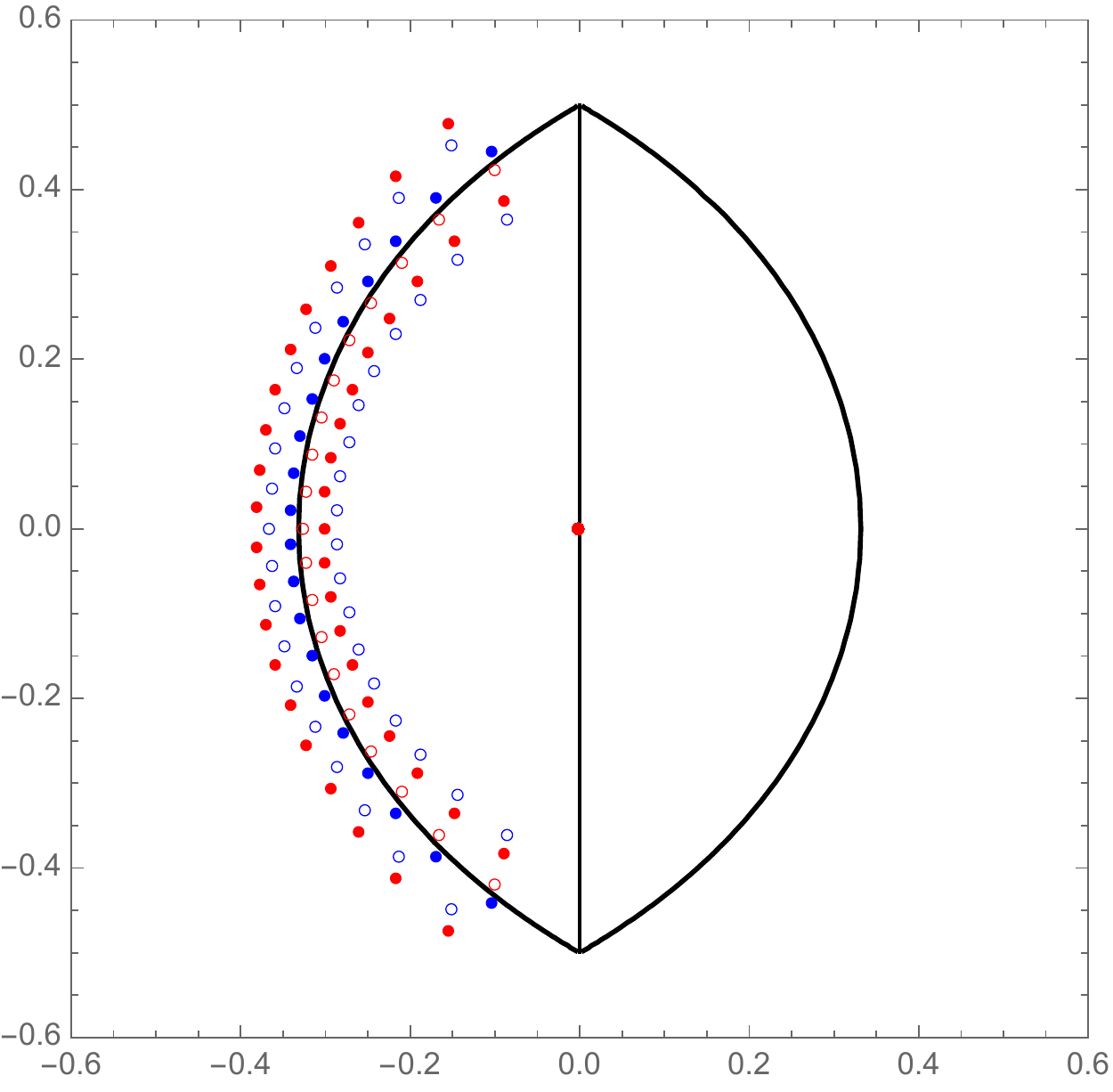}\hspace{0.03\linewidth}%
\includegraphics[width=0.3\linewidth]{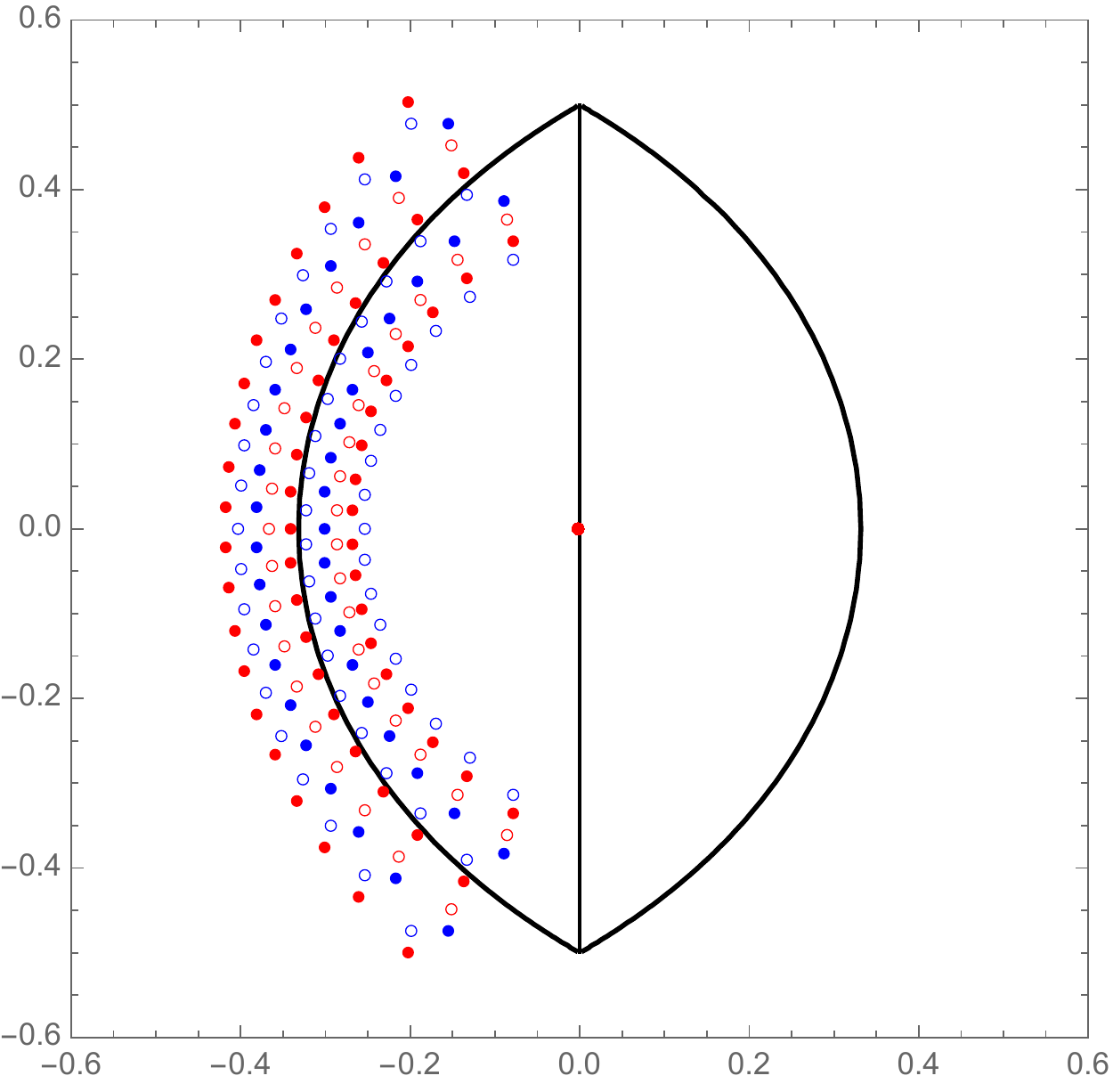}
\end{center}
\caption{As in Figure~\ref{fig:4} but for $n=20$ and $m=\tfrac{1}{2}$ (left), $m=\tfrac{3}{2}$ (center), and $m=\tfrac{5}{2}$ (right).}
\label{fig:5}
\end{figure}
From these plots we may formulate a second conjecture.
\begin{conjecture}
Suppose that $m=\tfrac{1}{2}+k$, $k\in\mathbb{Z}_{\ge 0}$.  Then as $n\to\infty$, the poles and zeros of $u_n(ny,m)$ accumulate near the left boundary arc of the domain $E$ in the $y$-plane.  In more detail, the poles and zeros are arranged along $4k+2$ non-intersecting arcs roughly parallel to and $o(1)$ distance from the left boundary arc of $E$.  The outermost curve contains $n$ poles of $u_n(ny;m)$ coming from roots of $s_n(ny;m)$ and moving inwards the next curve contains $n-1$ zeros of $u_n(ny;m)$ coming from roots of $s_{n-1}(ny;m)$.  If $k>0$ there are then $k$ families of four nested curves each; the $j^\mathrm{th}$ family lies to the outside of the $j+1^\mathrm{st}$ and consists of (in order from outside to inside, $j=1,\dots,k$):
\begin{itemize}
\item A curve containing $n-j+1$ zeros of $u_n(ny;m)$ coming from roots of $s_n(ny;m-1)$.
\item A curve containing $n-j$ poles of $u_n(ny;m)$ coming from roots of $s_{n-1}(ny;m-1)$.
\item A curve containing $n-j$ poles of $u_n(ny;m)$ coming from roots of $s_n(ny;m)$.
\item A curve containing $n-j-1$ zeros of $u_n(ny;m)$ coming from roots of $s_{n-1}(ny;m)$.
\end{itemize}
\label{conjecture:half-integers}
\end{conjecture}
A suitably precise form of Conjecture~\ref{conjecture:half-integers} is proven in \cite{BothnerM18} using classical steepest descent analysis for certain Hankel systems with Bessel function coefficients derived from Riemann-Hilbert Problem~\ref{rhp:renormalized} in Section~\ref{sec:m-half-integer} below.\bigskip

Comparing Figures~\ref{fig:1}--\ref{fig:3} with Figures~\ref{fig:4}--\ref{fig:5} makes clear that the asymptotic behavior of $u_n(x;m)$ cannot possibly be uniform with respect to $m$ in any neighborhood of a half-integral value.  It appears to therefore be compelling to investigate how $u_n(x;m)$ behaves if $n$ is large while simultaneously $m$ is close to a given half-integer.  Such an experiment is reproduced in Figure~\ref{fig:6}.
\begin{figure}[h]
\begin{center}
\includegraphics[width=0.3\linewidth]{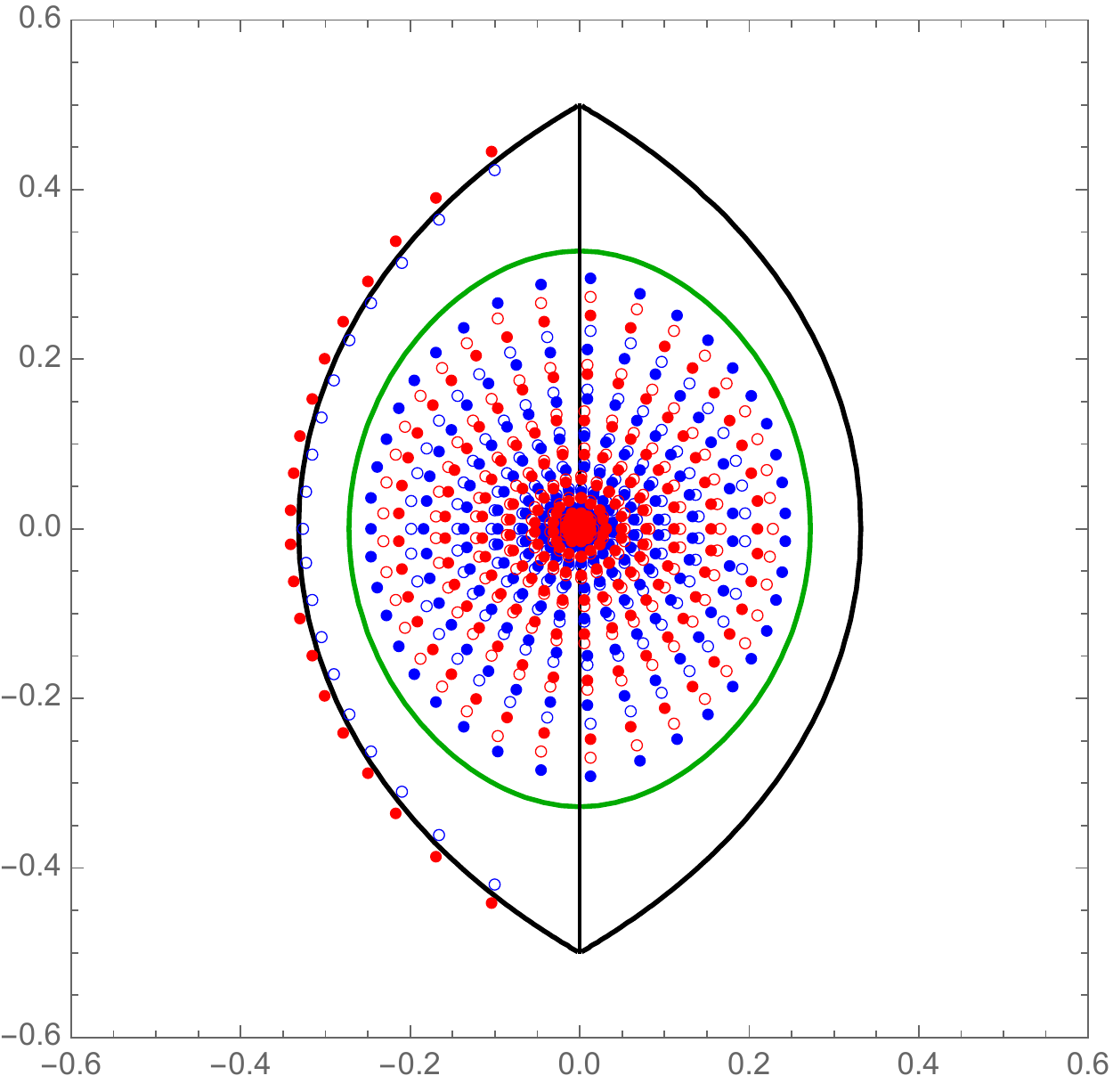}\hspace{0.03\linewidth}%
\includegraphics[width=0.3\linewidth]{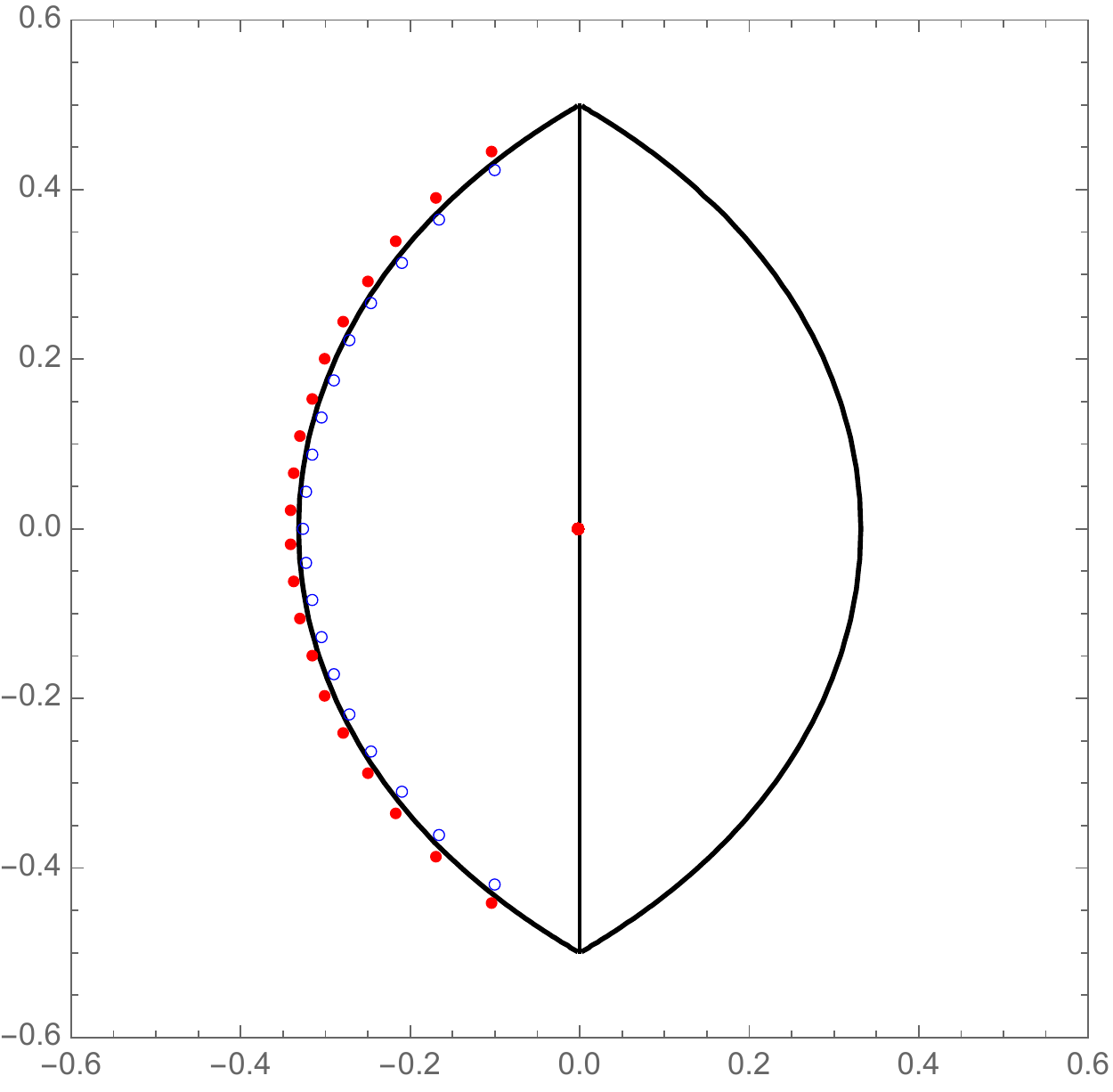}\hspace{0.03\linewidth}%
\includegraphics[width=0.3\linewidth]{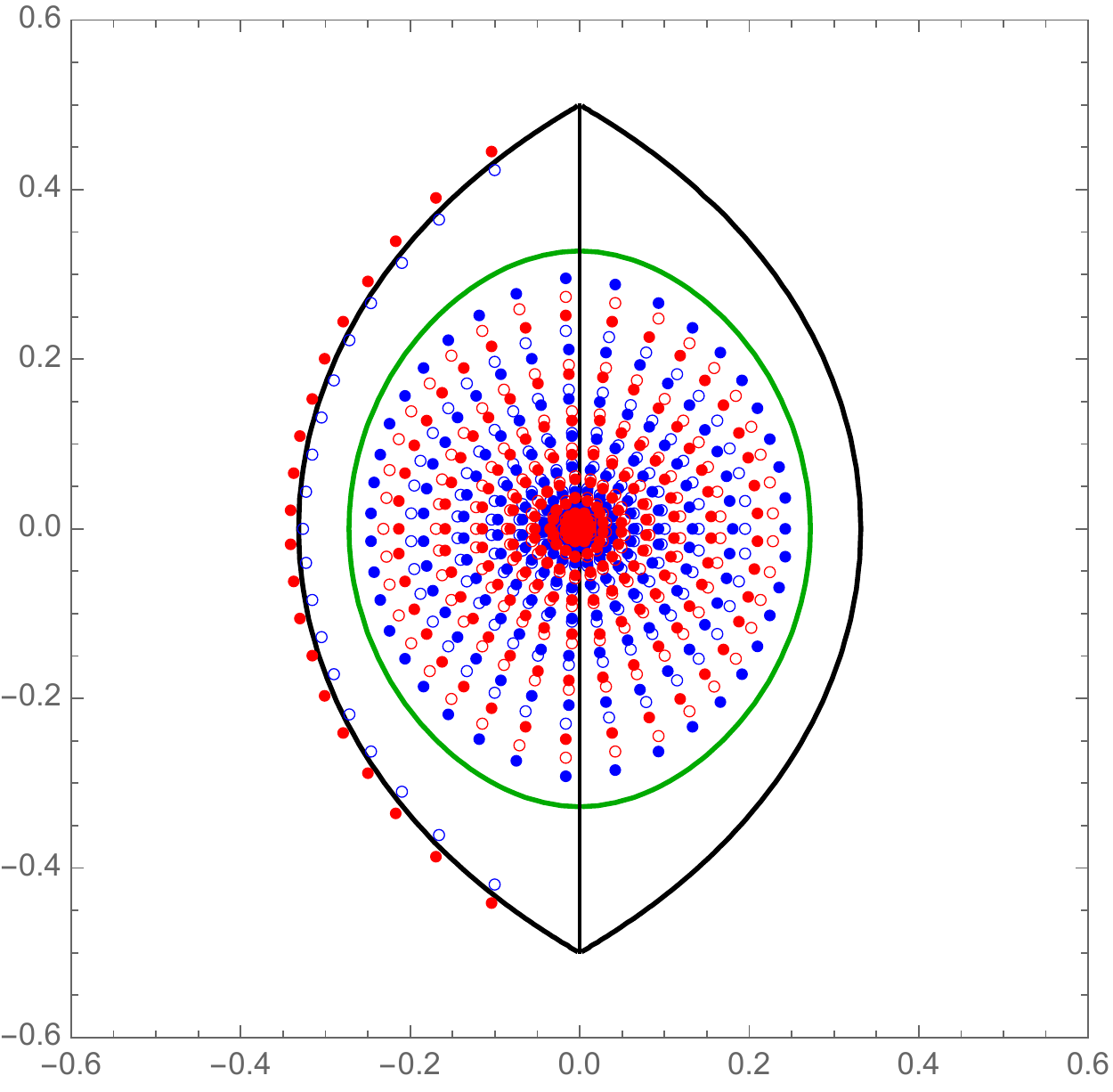}
\end{center}
\caption{As in Figure~\ref{fig:1} but for $n=20$ and $m=\tfrac{1}{2}-10^{-4}$ (left), $m=\tfrac{1}{2}$ (center), and $m=\tfrac{1}{2}+10^{-4}$ (right).  Superimposed in green is another curve that better approximates the central pole/zero region in a double-scaling limit where $n$ grows while $m$ approaches a half-integer \cite{BothnerM18}.}
\label{fig:6}
\end{figure}
This figure suggests that if $m$ is taken to be very close to a half-integer, the majority of the poles and zeros of $u_n(x;m)$ are captured in the midst of a process in which they are collapsing toward the origin, leaving just a small fraction of them near the left (for positive half-integer $m$) ``eyebrow''. 
In this situation, the domain containing the majority of the poles and zeros appears to be smaller than the full domain $E$.  This collapse process can be studied \cite{BothnerM18} with the help of Theorem~\ref{thm:RH-representation} and asymptotic analysis in a double-scaling limit in which $n$ is large and $m$ differs from a half-integer by an exponentially small amount.  The green curve plotted in Figure~\ref{fig:6} is one of the outcomes of this analysis.  The same analysis shows that the convergence claimed in Conjecture~\ref{conjecture:outside} also holds for $y$ in the annular region between the boundary of $E$ and the green curve, as well as near the right ``eyebrow'' (but something more like Conjecture~\ref{conjecture:half-integers} occurs near the left ``eyebrow'').\bigskip  

Taking now $m\not\in\mathbb{Z}+\tfrac{1}{2}$, an interesting question suggested by the scaling analysis above is whether $u_n(ny_0+w;m)$ behaves asymptotically (as a function of $w$ for fixed $y_0\in E$) like an elliptic function solving \eqref{eq:dotV-ODE} for a suitable choice of integration constant $C$ such that the quartic $P$ has four distinct roots.  To investigate this, we select a point $y_0$ in the domain $E$ and display in Figure~\ref{fig:7} the poles and zeros of $u_n(ny_0+w;m)$ in the $w$-plane.
\begin{figure}[h]
\begin{center}
\includegraphics[width=0.3\linewidth]{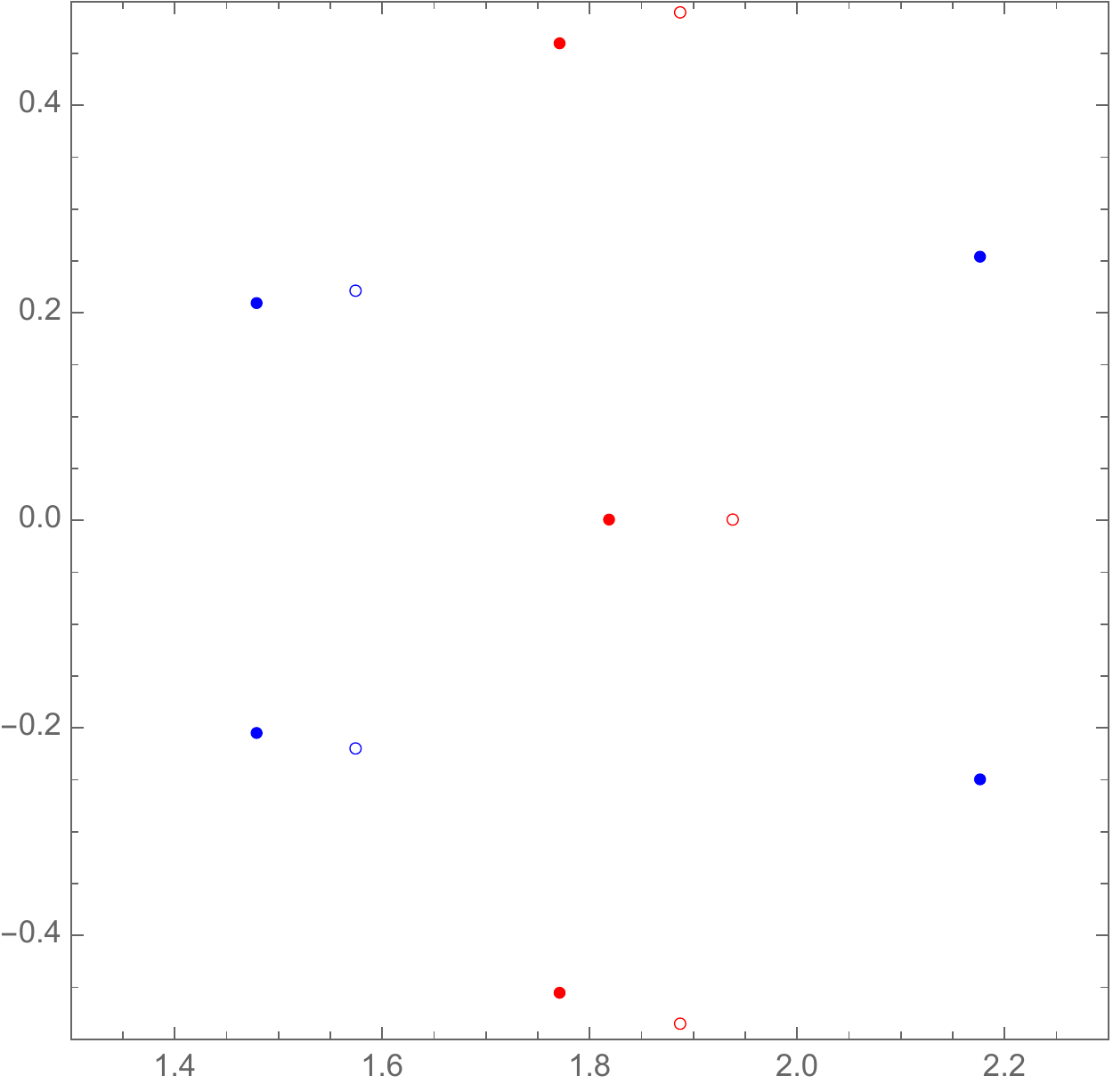}\hspace{0.03\linewidth}%
\includegraphics[width=0.3\linewidth]{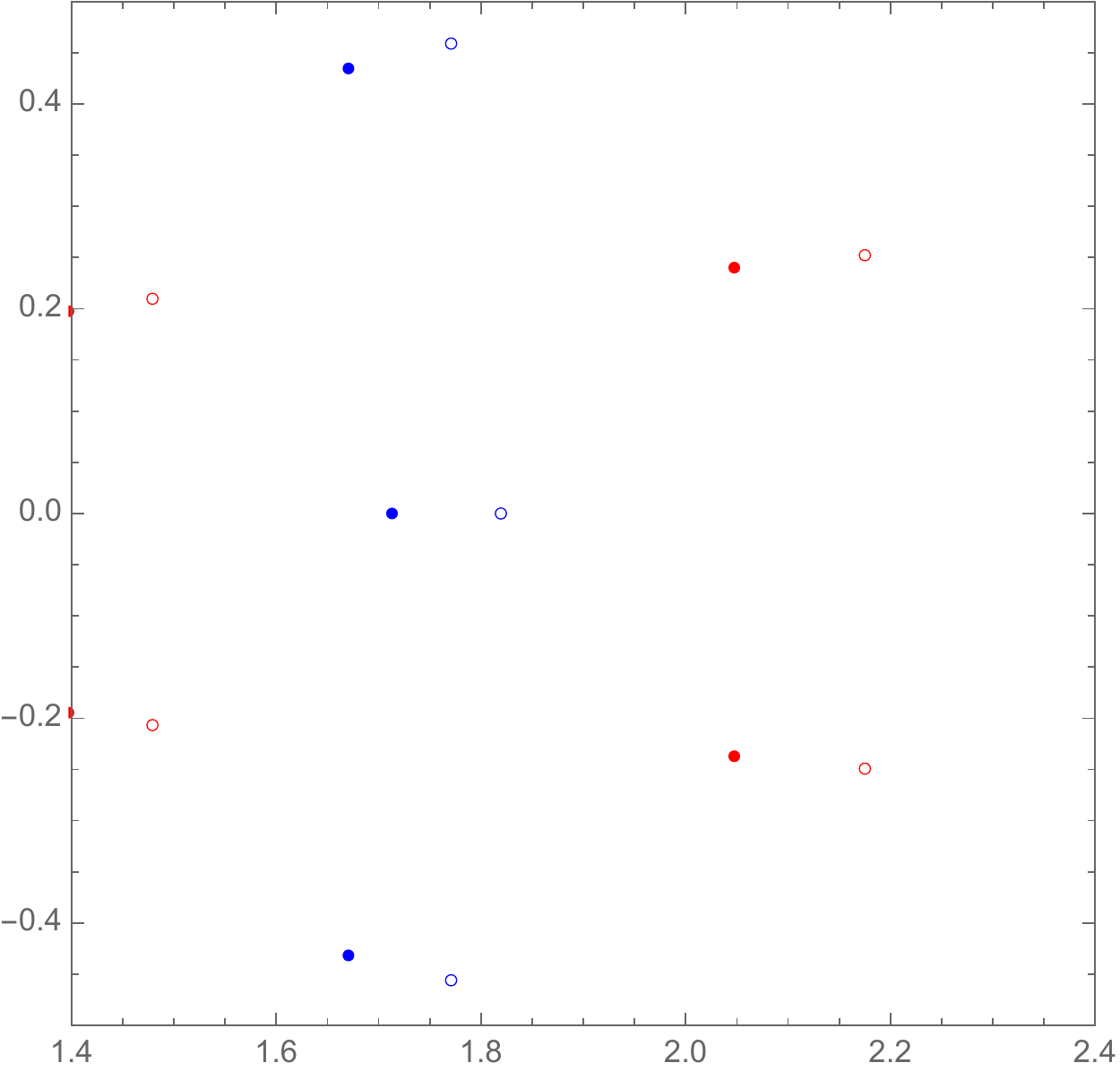}\hspace{0.03\linewidth}%
\includegraphics[width=0.3\linewidth]{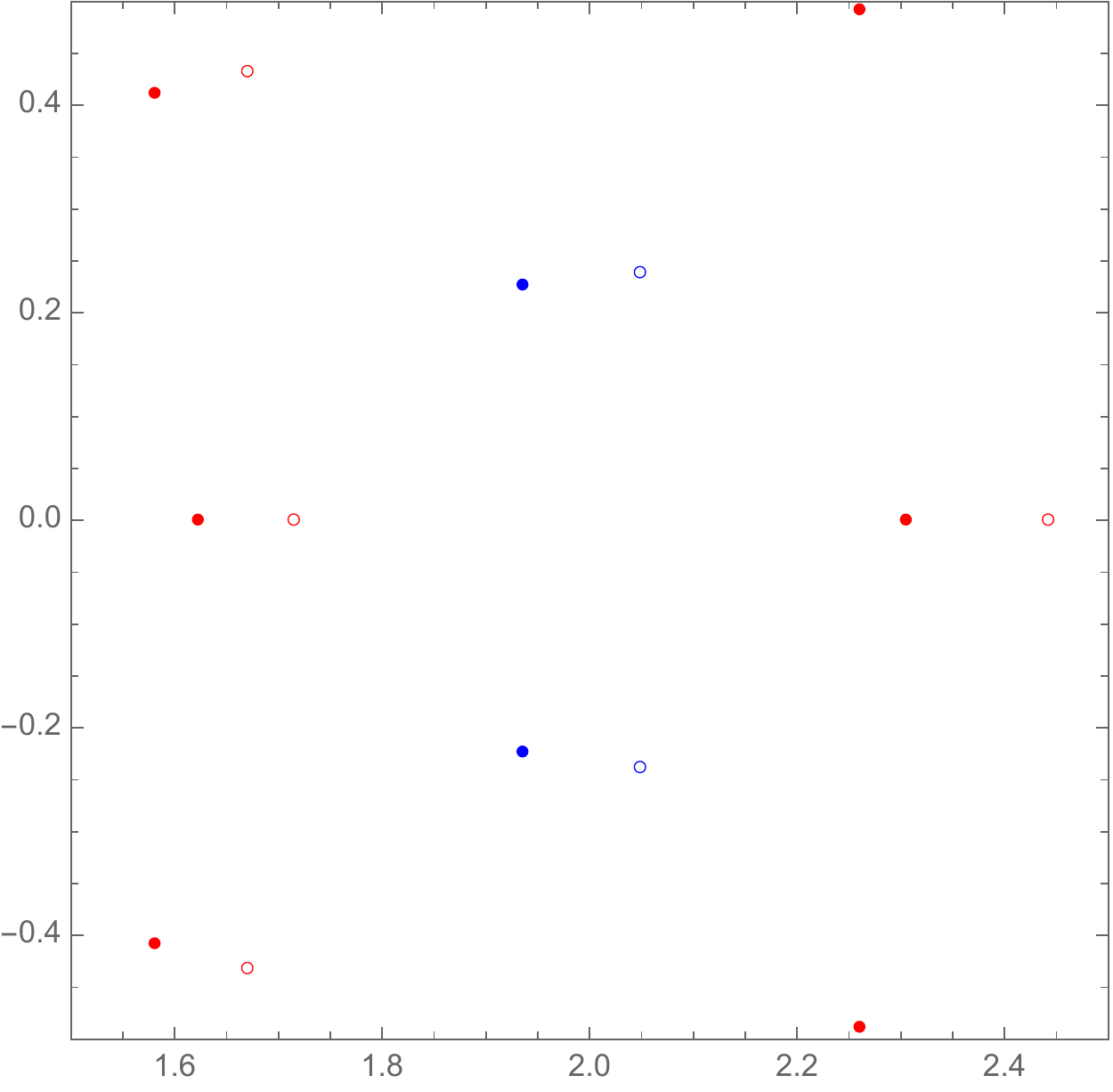}
\end{center}
\caption{As in Figure~\ref{fig:1} but plotted in the $w$-plane for $m=0$ and $y_0=0.1$ with $n=18$ (left), $n=19$ (center), and $n=20$ (right).}
\label{fig:7}
\end{figure}
This figure suggests that indeed for given large $n$, the poles and zeros are arranged roughly in a doubly-periodic lattice, with the lattice becoming more rigid as $n$ increases.  An important observation is that the lattice does not appear to become fixed as $n$ increases, although its lattice vectors do.  To the contrary, there appears to be a strong fluctuation of the offset of the lattice as $n$ is increased in integer increments.  These observations are consistent with the approximation of $u_n(ny_0+w;m)$ by a family of solutions of the autonomous elliptic function differential equation \eqref{eq:dotV-ODE} differing by an $n$-dependent shift in the argument $w$.  We formulate this as a conjecture.
\begin{conjecture}
Assume that $m\not\in\mathbb{Z}+\tfrac{1}{2}$ is fixed, and fix $y_0\in E$.  Then there is a solution $\dot{p}=\dot{p}_n(w;y_0)$ (an elliptic function of $w$) of the differential equation \eqref{eq:dotV-ODE} for suitable $C=C(y_0)$ such that the quartic $P$ has distinct roots, for which
\begin{equation}
\lim_{n\to\infty} \left(u_n(ny_0+w;m)-\ii\dot{p}_n(w;y_0)\right)=0.
\end{equation}
\label{conjecture:elliptic}
\end{conjecture}
This conjecture is proved in \cite{BothnerM18} using Theorem~\ref{thm:RH-representation}.  Part of the proof involves isolating the correct value of the integration constant $C$ given $y_0\in E$.  It is also important in the proof that $y_0$ not lie on the imaginary axis, which is excluded from $E$ as shown in Figures~\ref{fig:1}--\ref{fig:6}.  Also, $w$ should be restricted to a bounded domain that excludes arbitrarily small fixed neighborhoods of certain lattice points.\bigskip

We have already pointed out that the two ``corner points'' of the eye-shaped domain $E$ occur at the values $y=y_0=\pm\tfrac{1}{2}\ii$.  These values are the only ones for which the quartic $P$ can have only one four-fold root.  This particularly severe degeneration of the quartic suggests that the rational solution $u_n(x;m)$ may behave in a special way for large $n$ when $x\approx\pm\tfrac{1}{2}\ii n$, a notion that is reinforced by another suitable rescaling of \eqref{eq:PIII-again}.  Indeed, to localize $y=x/n$ near $y_0=\pm \tfrac{1}{2}\ii$, we set $x=\pm\ii (\tfrac{1}{2}n + (\tfrac{1}{32}n)^{1/3}\xi^\pm)$ and consider $\xi^\pm$ to be bounded.  Similarly, since $\lNaught_0^+(\pm \tfrac{1}{2}\ii n)
%=\lNaught_0^-(\pm\tfrac{1}{2}\ii n)
=\pm 1$, we wish to localize $u$ near $\pm \ii$ so we set $u=\pm\ii(1 - (\tfrac{1}{4}n)^{-1/3}W^\pm)$ and consider $W^\pm$ to be bounded.  (The exponents of $\pm \tfrac{1}{3}$ are chosen to achieve a dominant balance, and the numerical coefficients of $\tfrac{1}{32}$ and $\tfrac{1}{4}$ are chosen for convenience.)  Making these substitutions, we multiply \eqref{eq:PIII-again} through by $\mp\tfrac{1}{8}\ii xu(x)$ and obtain
\begin{equation*}
\frac{\dd^2W}{\dd\xi^2}=2W^3+\xi W +m + O(n^{-1/3}),\quad \xi=\xi^\pm,\quad W=W^\pm,
\end{equation*}
where again the final term combines several others all proportional to $n^{-1/3}$ or more negative powers of $n$.  Neglecting the error terms and relabeling $W$ as $\dot{W}$ yields as a model equation
\begin{equation}
\frac{\dd^2\dot{W}}{\dd\xi^2}=2\dot{W}^3+\xi\dot{W}+m
\label{eq:PII}
\end{equation}
which is the Painlev\'e-II equation with parameter $m$.  Based on this calculation, we may expect that when $n$ is large and $m$ is held fixed, the rational Painlev\'e-III functions behave near the points $x=\pm\tfrac{1}{2}\ii n$ like certain solutions of the Painlev\'e-II equation \eqref{eq:PII}; moreover, the dependence on the fixed parameter $m$ becomes apparent at leading order in this approximation.
To explore this possibility, we plot the poles and zeros of $u_n(x;m)$ in the $\xi^\pm$ planes for two fixed values of $m$ and for increasing $n$ in Figures~\ref{fig:8}--\ref{fig:11}.
\begin{figure}[h]
\begin{center}
\includegraphics[width=0.3\linewidth]{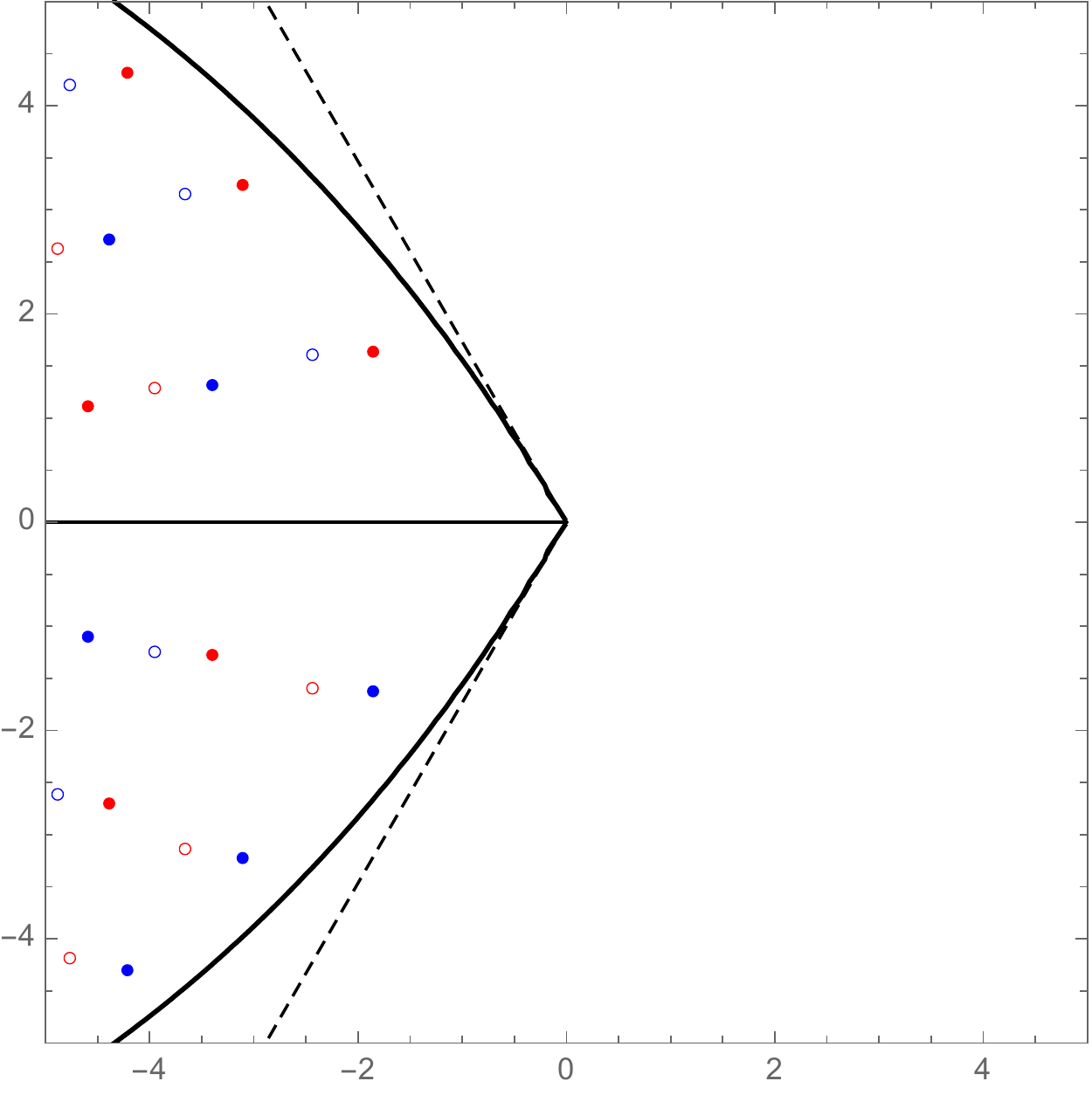}\hspace{0.03\linewidth}%
\includegraphics[width=0.3\linewidth]{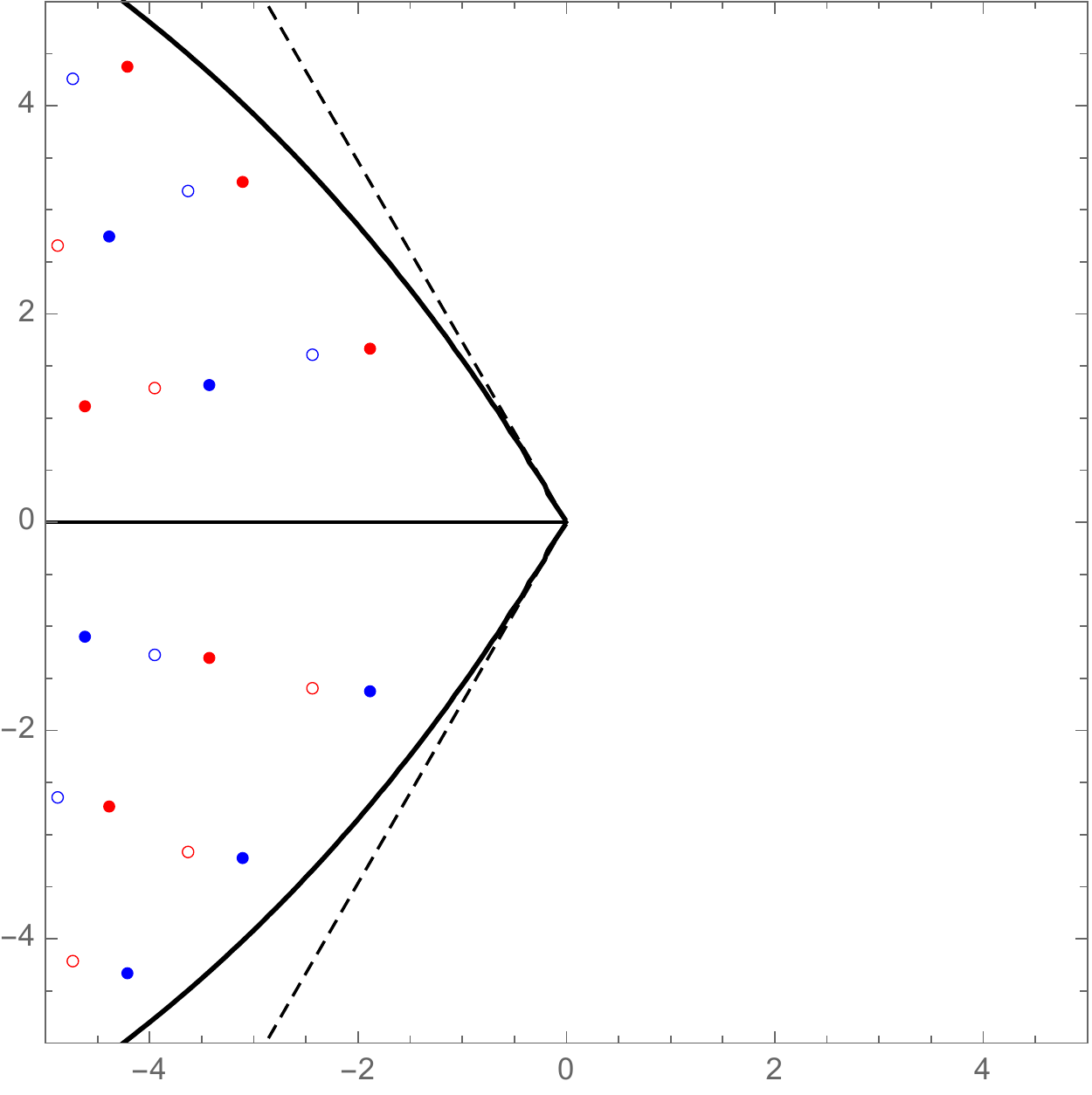}\hspace{0.03\linewidth}%
\includegraphics[width=0.3\linewidth]{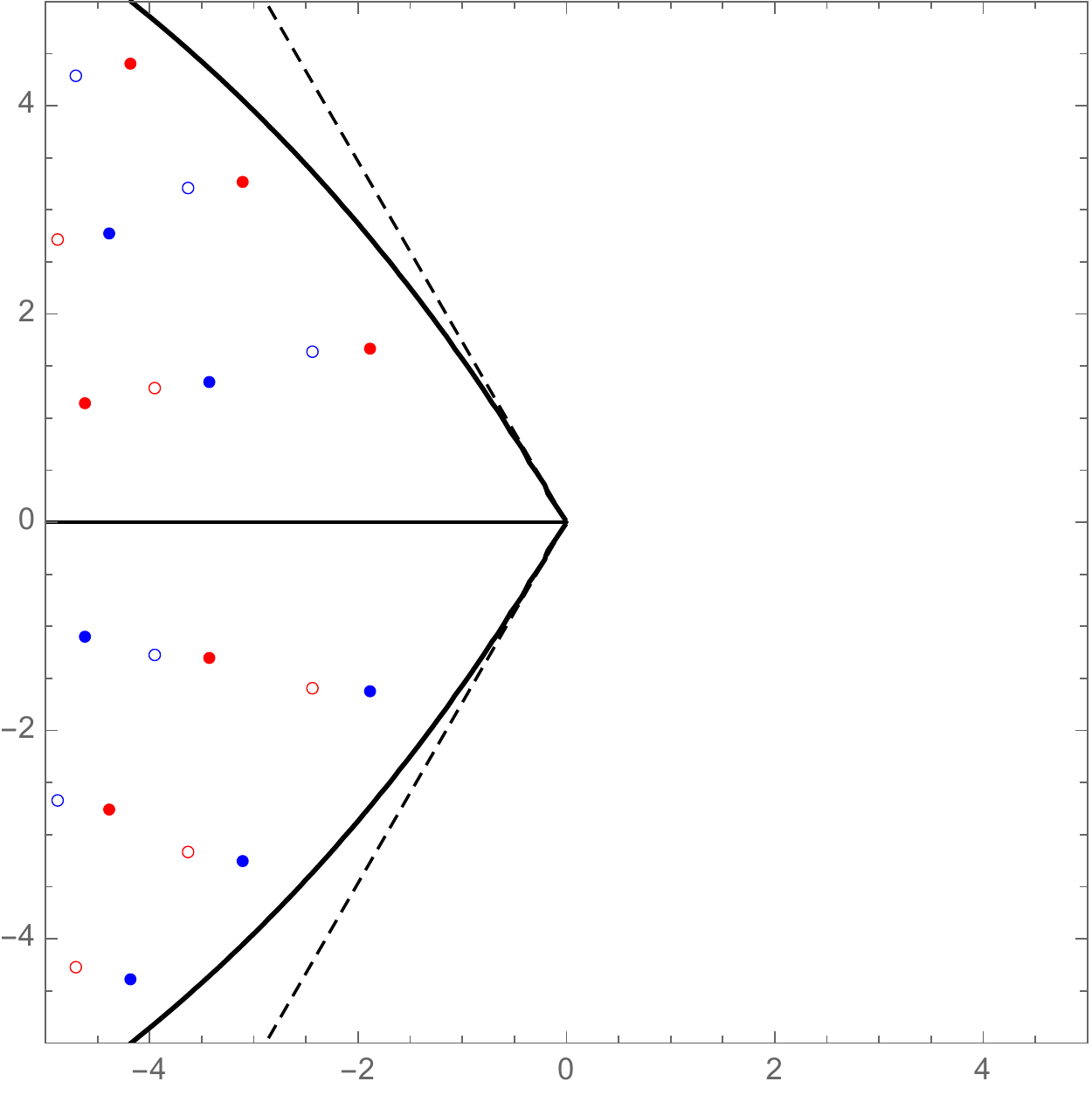}
\end{center}
\caption{As in Figure~\ref{fig:1} but plotted in the $\xi^+$-plane for $m=0$ and $n=18$ (left), $n=19$ (center), and $n=20$ (right).  Also shown with dashed lines are the rays $\mathrm{Arg}(\xi^+)=\pm \tfrac{2}{3}\pi$, which are the tangents to the boundary of $E$ at the upper corner.}
\label{fig:8}
\end{figure}
\begin{figure}[h]
\begin{center}
\includegraphics[width=0.3\linewidth]{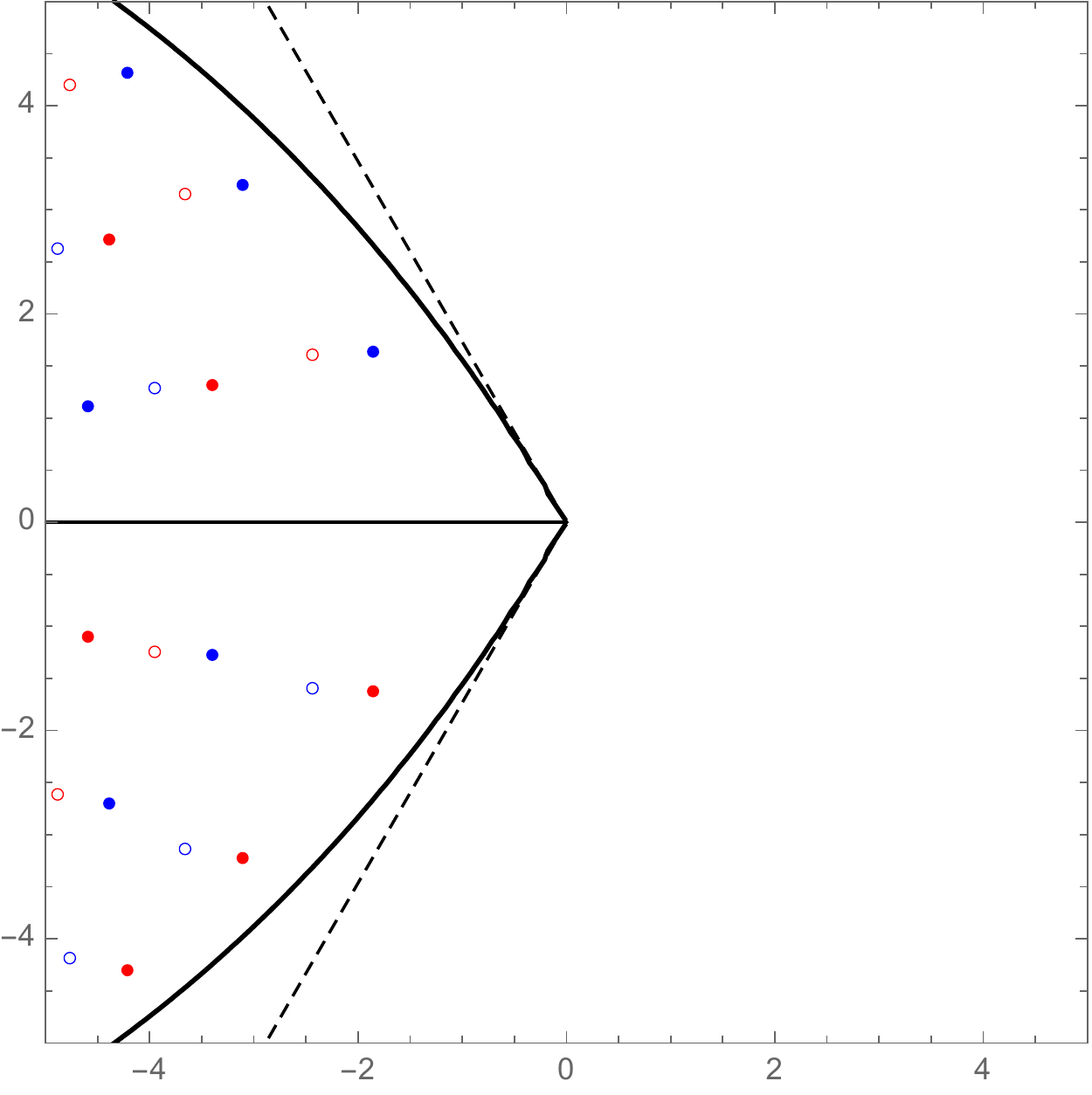}\hspace{0.03\linewidth}%
\includegraphics[width=0.3\linewidth]{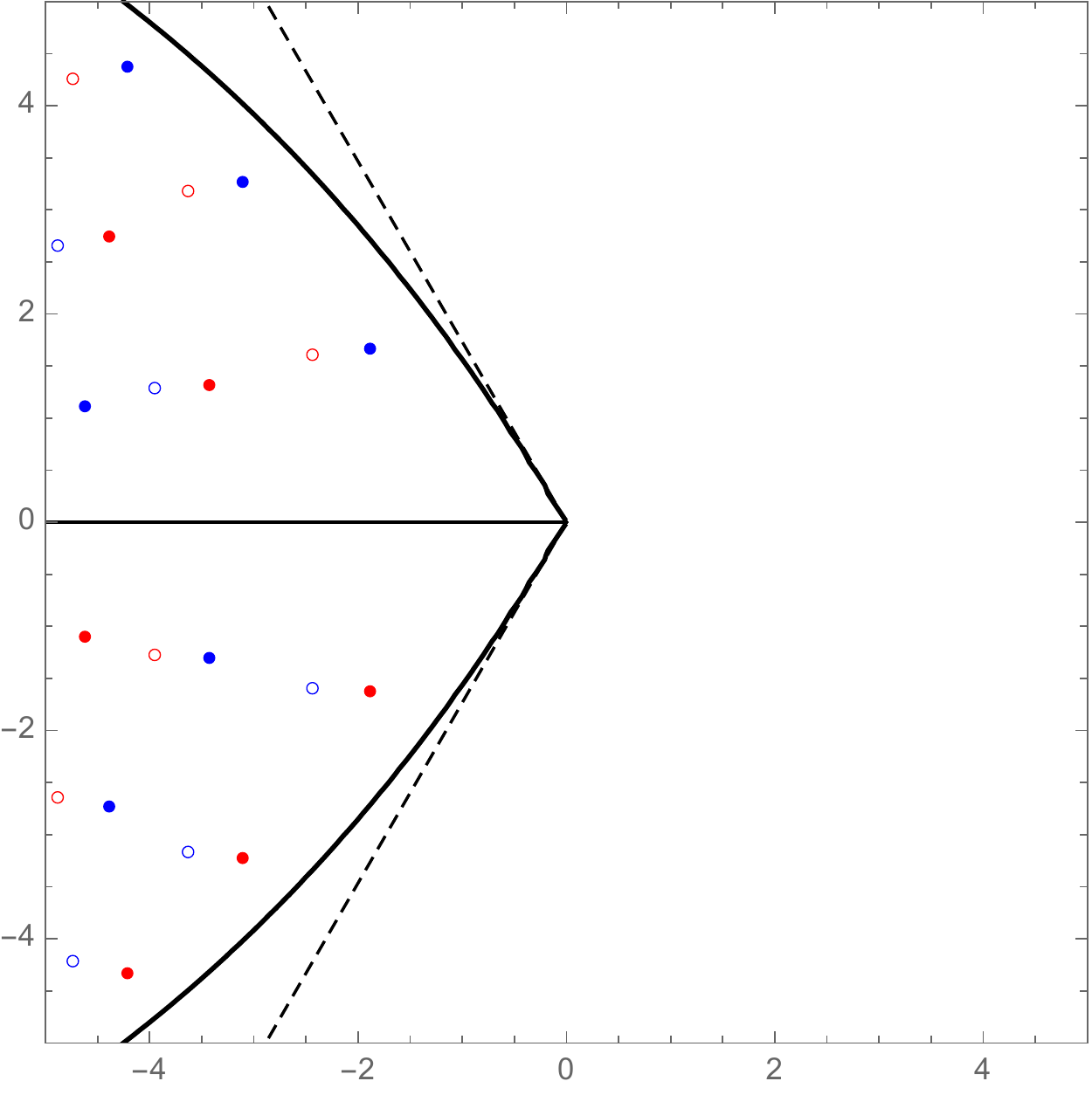}\hspace{0.03\linewidth}%
\includegraphics[width=0.3\linewidth]{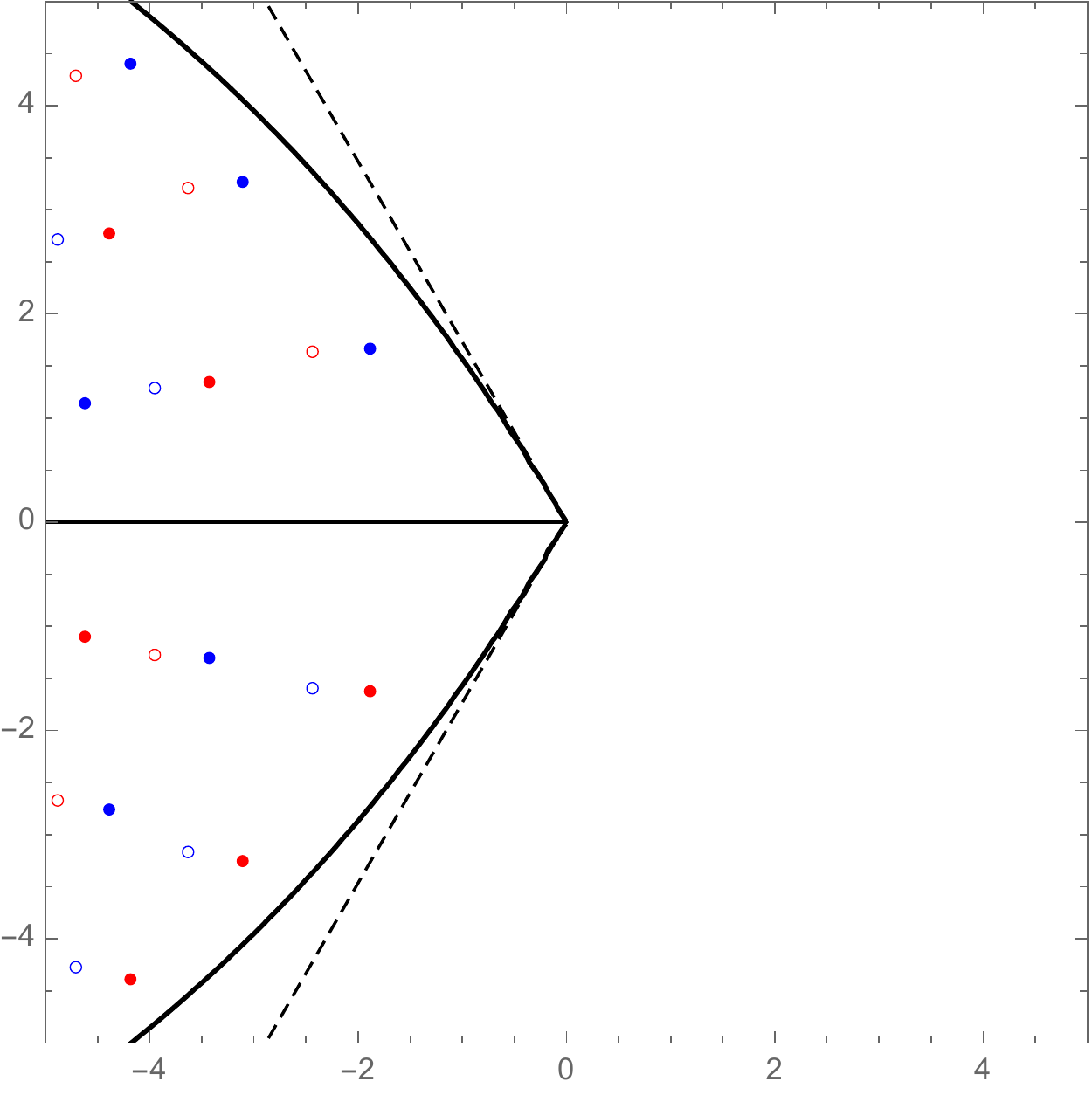}
\end{center}
\caption{As in Figure~\ref{fig:1} but plotted in the $\xi^-$-plane for $m=0$ and $n=18$ (left), $n=19$ (center), and $n=20$ (right).  Also shown with dashed lines are the rays $\mathrm{Arg}(\xi^-)=\pm \tfrac{2}{3}\pi$, which are the tangents to the boundary of $E$ at the lower corner.}
\label{fig:9}
\end{figure}
\begin{figure}[h]
\begin{center}
\includegraphics[width=0.3\linewidth]{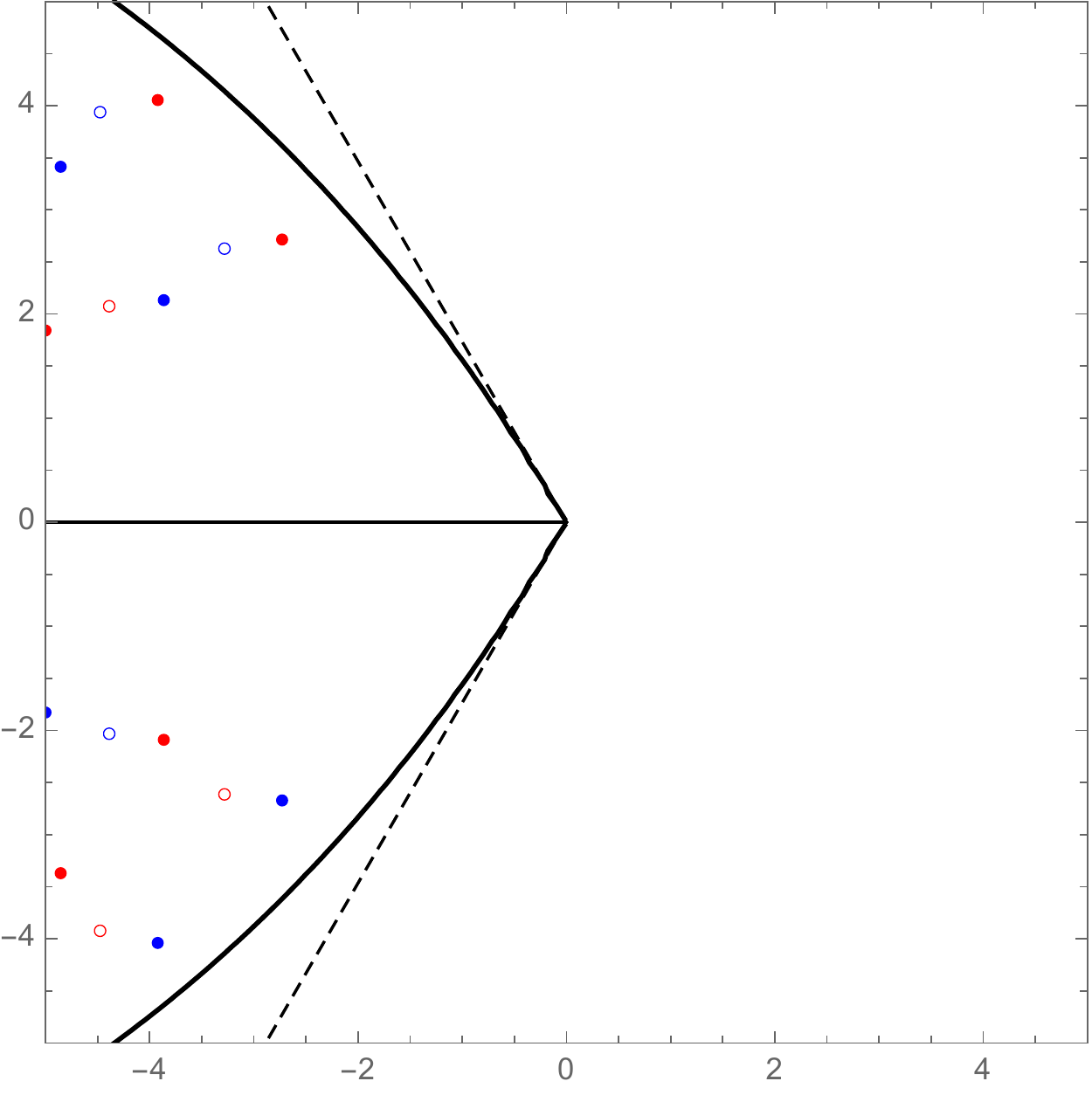}\hspace{0.03\linewidth}%
\includegraphics[width=0.3\linewidth]{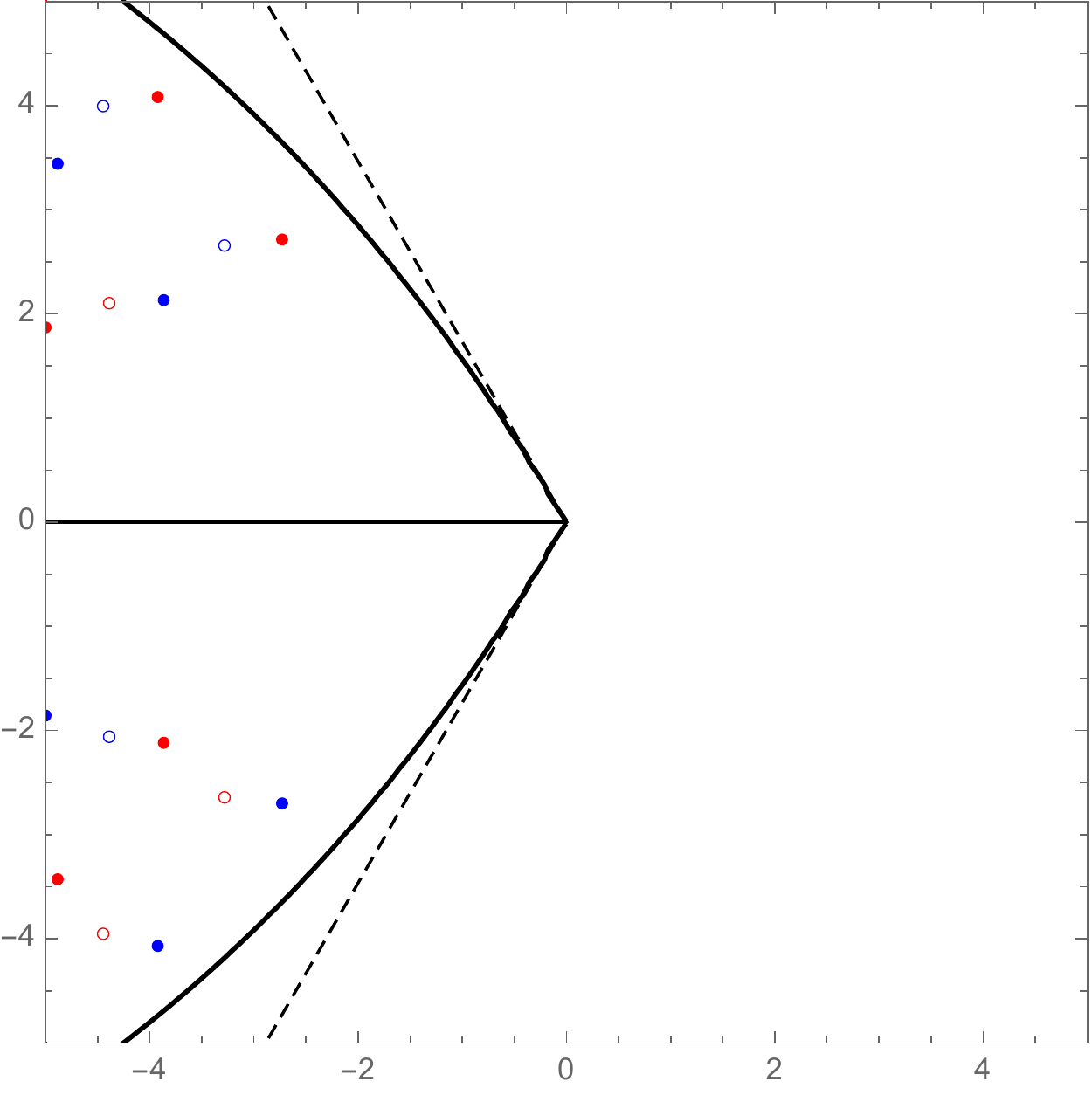}\hspace{0.03\linewidth}%
\includegraphics[width=0.3\linewidth]{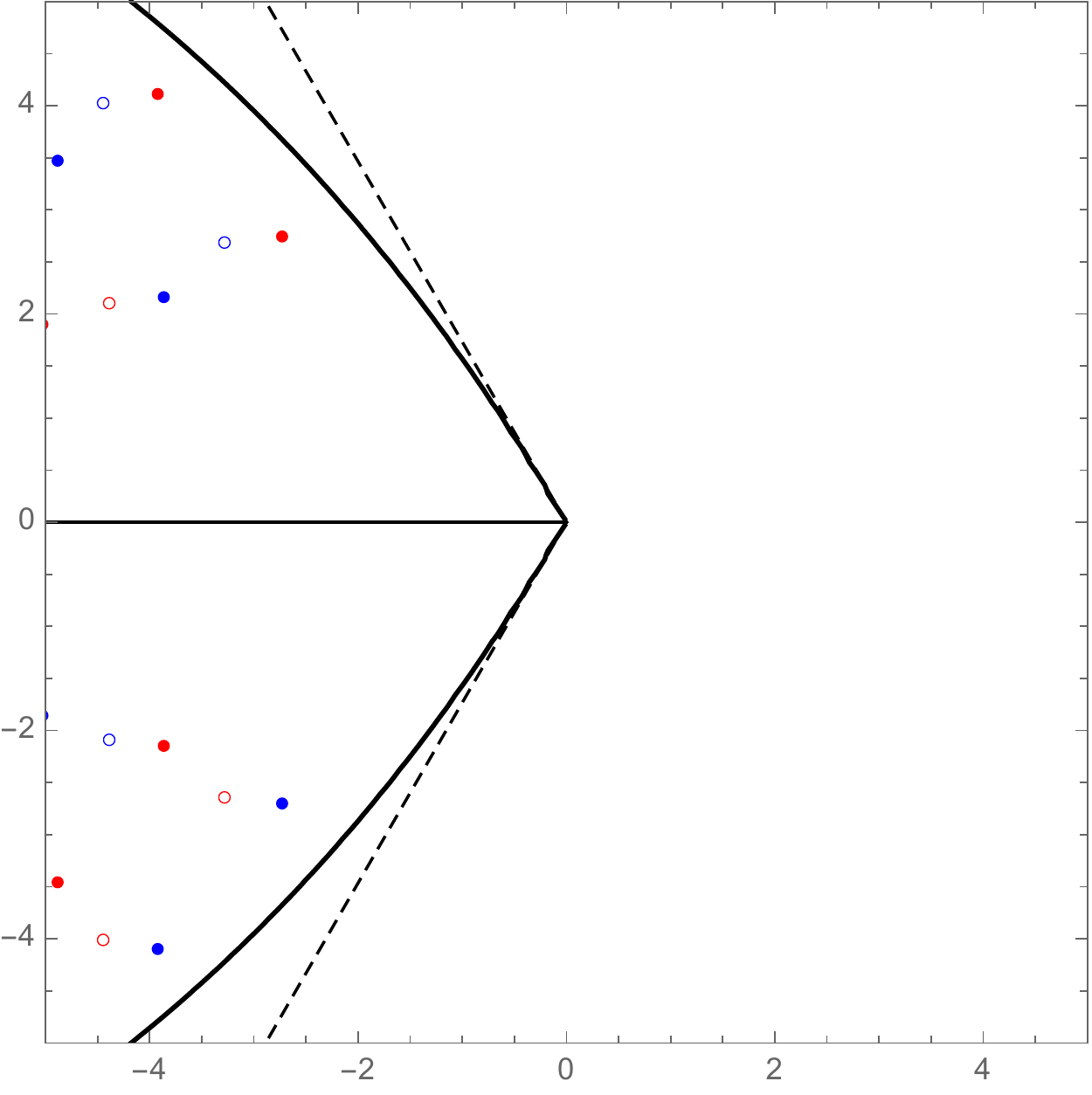}
\end{center}
\caption{As in Figure~\ref{fig:8} (zooming into the upper corner of the domain $E$) but for $m=\tfrac{4}{5}\ii$.}
\label{fig:10}
\end{figure}
\begin{figure}[h]
\begin{center}
\includegraphics[width=0.3\linewidth]{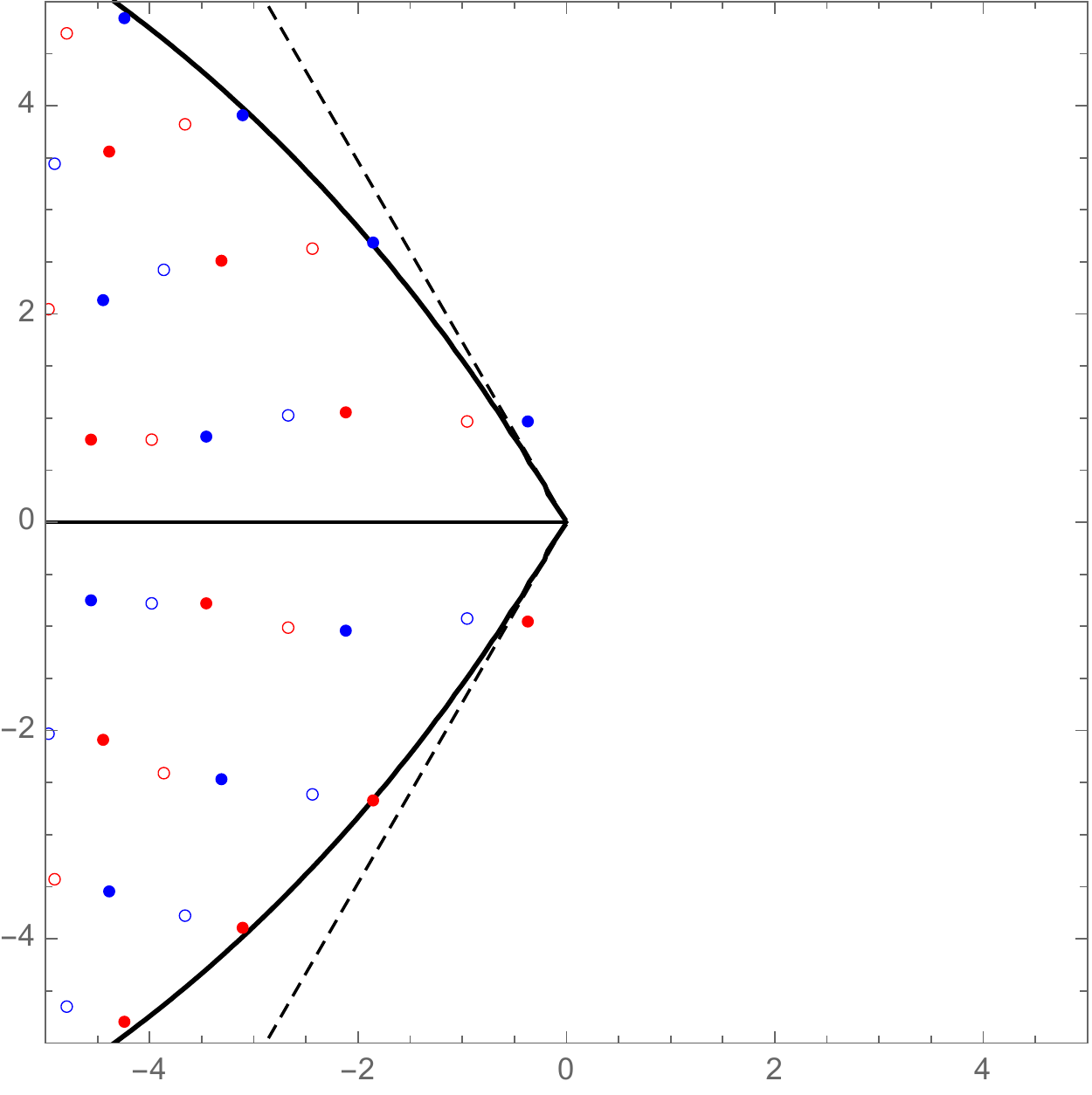}\hspace{0.03\linewidth}%
\includegraphics[width=0.3\linewidth]{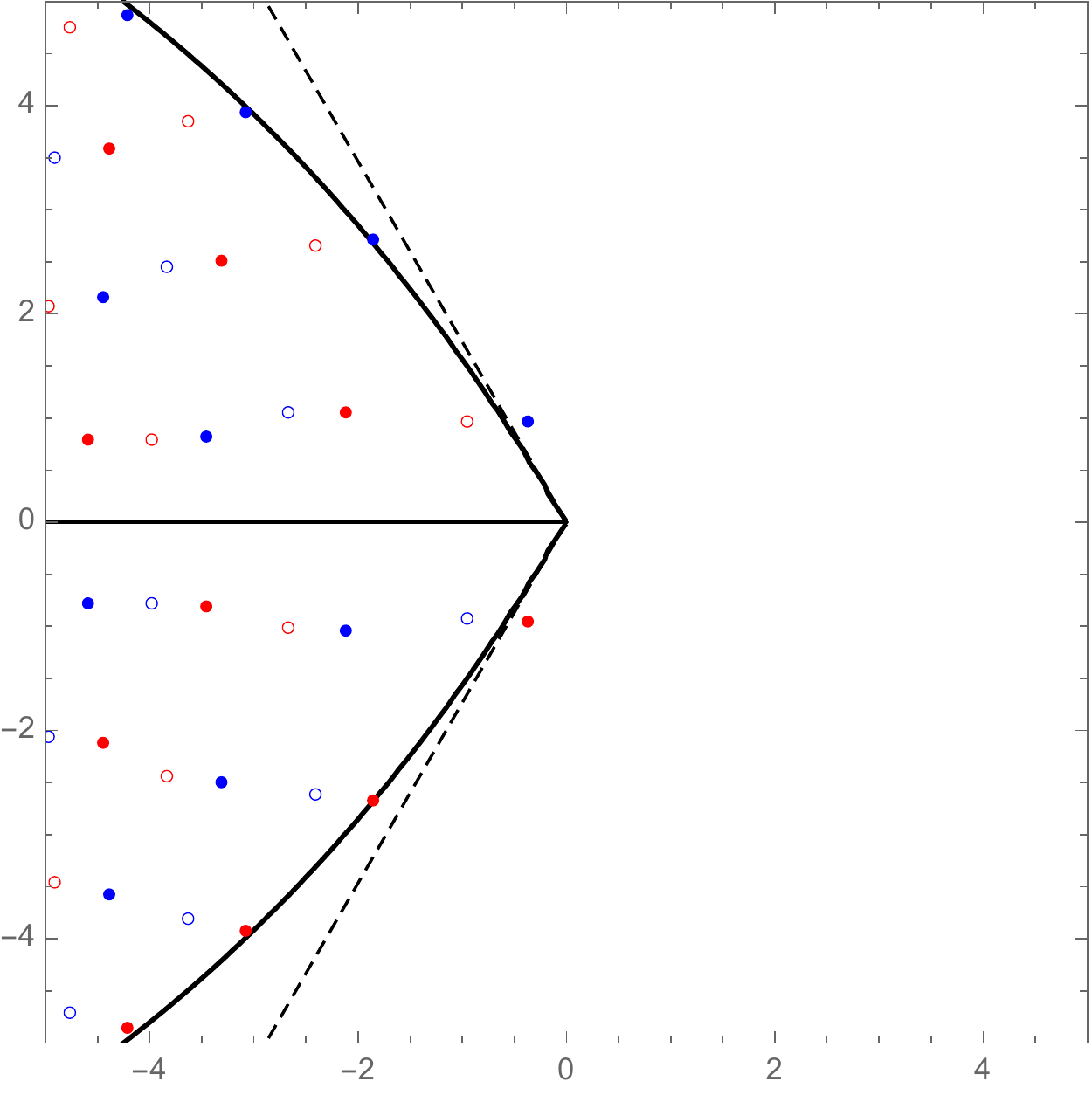}\hspace{0.03\linewidth}%
\includegraphics[width=0.3\linewidth]{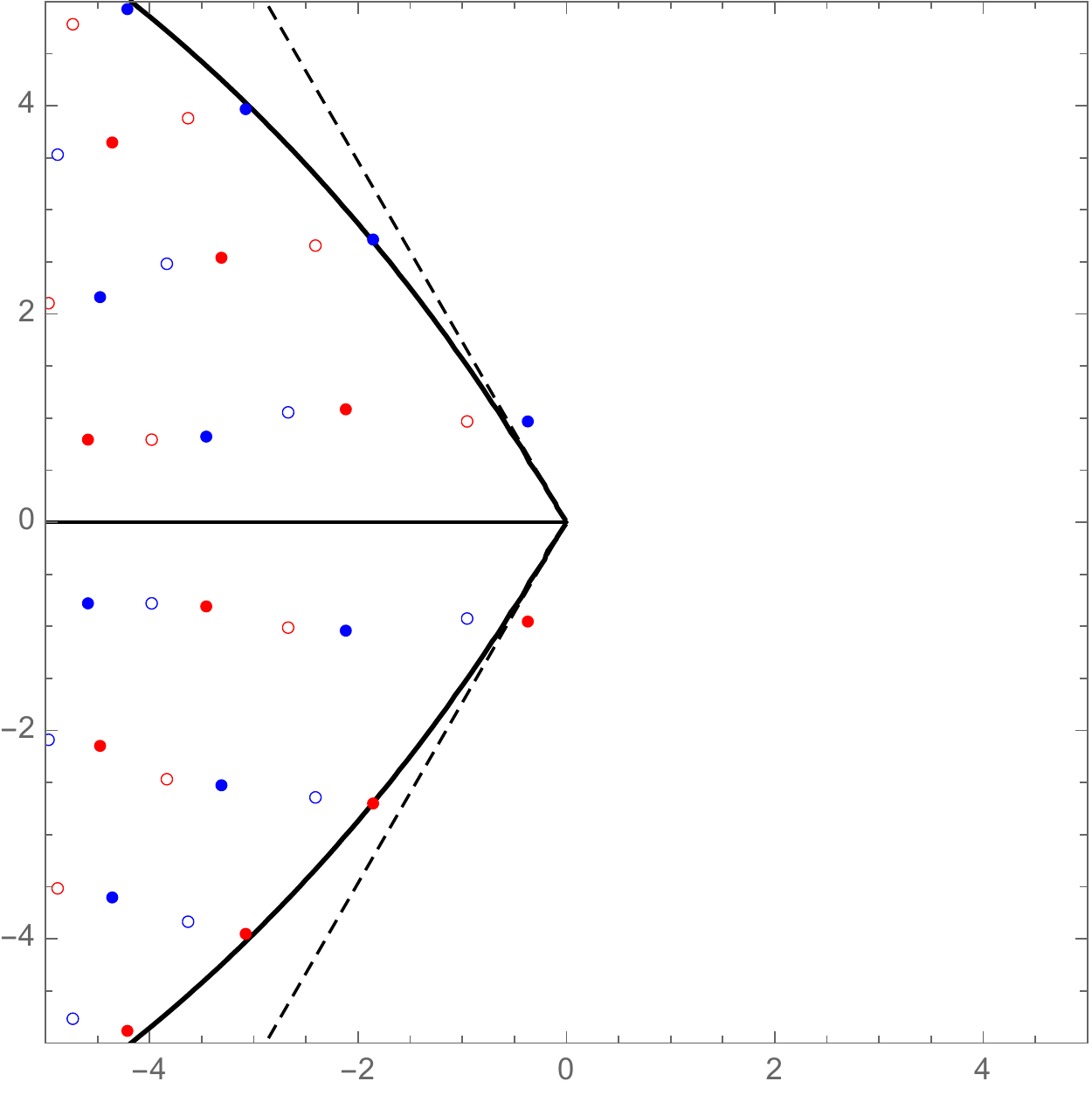}
\end{center}
\caption{As in Figure~\ref{fig:9} (zooming into the lower corner of the domain $E$) but for $m=\tfrac{4}{5}\ii$.}
\label{fig:11}
\end{figure}

In each of these figures, the three plots for consecutive reasonably large values of $n$ are nearly indistinguishable to the eye, suggesting convergence to a particular solution of \eqref{eq:PII} independent of $n$. To try to identify the relevant particular solutions, we may start with the outer approximation given in Conjecture~\ref{conjecture:outside} and re-express it in terms of the recentered and rescaled independent variables $\xi^\pm$, taking careful account of the principal branch interpretation of the square root in \eqref{eq:V0-four-answers}.  Thus, $u_n(x;m)\approx \ii p_0^+(y)=\ii p_0^+(n^{-1}x)=\pm\ii 2^{1/6}n^{-1/3}(\xi^\pm)^{1/2}+O(n^{-2/3}\xi^\pm)$ assuming that Conjecture~\ref{conjecture:outside} holds and that $\xi^\pm$ is small compared to $n^{2/3}$.  If this expression is to agree in some overlap domain with an approximation based on the Painlev\'e-II equation \eqref{eq:PII}, we should express $W=W^\pm$ in terms of $u_n(x;m)\approx \ii p_0^+(y)$.  Thus, 
$W^\pm=(\tfrac{1}{4}n)^{1/3}(1\pm\ii u_n(x;m))\approx (\tfrac{1}{4}n)^{1/3}(1\mp p_0^+(y)) = \pm\ii (\tfrac{1}{2}\xi^\pm)^{1/2}+O(n^{-1/3}\xi^\pm)$ if also $\xi^\pm$ is small compared to $n^{1/3}$.  Assumption of an overlap domain then suggests that the relevant solutions of the Painlev\'e-II equation \eqref{eq:PII} should satisfy $\dot{W}^\pm\sim \pm \ii (\tfrac{1}{2}\xi^\pm)^{1/2}$ as $\xi^\pm\to\infty$ in the exterior domain where the outer approximation is valid. In the limit $n\to\infty$, this region corresponds to the sector $\mathrm{Arg}(\xi^\pm)\in (-\tfrac{2}{3}\pi,\tfrac{2}{3}\pi)$.
It is known that \cite[Chapter 11]{FokasIKN06} for each complex $m$ there are two and only two solutions of the Painlev\'e-II equation \eqref{eq:PII} denoted $\dot{W}=\dot{W}^\pm(\xi;m)$ with the asymptotic behavior $\dot{W}^\pm(\xi;m)\sim \pm\ii(\tfrac{1}{2}\xi)^{1/2}$ as $\xi\to\infty$ with $|\mathrm{Arg}(\xi)|\le \tfrac{2}{3}\pi-\epsilon$ for $\epsilon>0$ sufficiently small, where the one-half power denotes the principal branch.  These are known as \emph{tritronqu\'ee} solutions of \eqref{eq:PII}.  We are led to formulate the following conjecture.
\begin{conjecture}
Let $m\in\mathbb{C}$ be fixed.  Then,
\begin{equation}
\lim_{n\to\infty}\left(\frac{1}{4}n\right)^{1/3}(1\pm\ii u_n(\pm\ii(\tfrac{1}{2} n +(\tfrac{1}{32}n)^{1/3}\xi);m)=\dot{W}^\pm(\xi;m),
\end{equation}
where $\dot{W}=\dot{W}^\pm(\xi;m)$ are the aforementioned
tritronqu\'ee solutions of the Painlev\'e-II equation \eqref{eq:PII}. 
\label{conjecture:PII}
\end{conjecture}
The convergence might be expected to be uniform on compact subsets of the $\xi$-plane from which arbitrarily small open disks centered at the poles of the tritronqu\'ee solution in question have been excised. The assertion that the particular solutions of \eqref{eq:PII} should be of tritronqu\'ee type means that they are asymptotically analytic in a sector of the complex $\xi$-plane of opening angle $\tfrac{4}{3}\pi$, consistent with the plots in Figures~\ref{fig:8}--\ref{fig:11}.  Tronqu\'ee and tritronqu\'ee solutions of the Painlev\'e-II equation \eqref{eq:PII} were originally studied long ago by Boutroux; see also Joshi and Mazzocco \cite{JoshiM03}.  
When $m=0$, the Painlev\'e-II equation \eqref{eq:PII} has the obvious symmetry $\dot{W}(\xi)\mapsto -\dot{W}(\xi)$, and by uniqueness of the two tritronqu\'ee solutions this means that $\dot{W}^-(\xi;0)=-\dot{W}^+(\xi;0)$.  Comparing Figures~\ref{fig:8}--\ref{fig:9} we therefore expect a sign change while the figures clearly show instead some sort of reciprocation, with poles and zeros of $u_n(x;m)$ being exchanged.  The explanation for this lies in the relation $u=\pm\ii(1-(\tfrac{1}{4}n)^{-1/3}W^\pm)$, which shows that both poles and zeros of $u$ correspond to $W^\pm$ becoming very large; in other words, both the red and the blue dots in Figures~\ref{fig:8}--\ref{fig:11} should be attracted in the limit $n\to+\infty$ toward the fixed simple poles of the corresponding tritronqu\'ee solution of the Painlev\'e-II equation \eqref{eq:PII}. More to the point, assuming the validity of Conjecture~\ref{conjecture:PII} with the suggested nature of convergence, one may apply the argument principle to the rational function $u_n(\pm\ii(\tfrac{1}{2}n+(\tfrac{1}{32}n)^{1/3}\xi);m)$ about a Jordan curve $C$ in the $\xi$-plane that encloses exactly one pole of the corresponding tritronqu\'ee solution of \eqref{eq:PII}.  The index (increment of the argument) of $u_n$ about $C$ is zero for sufficiently large $n$ because $u_n$ converges uniformly on $C$ to $\pm\ii$ as $\dot{W}^\pm$ is analytic and therefore bounded on $C$.  This means that in fact \emph{each pole of the Painlev\'e-II tritronqu\'ee would be expected to attract (in the $\xi$-plane) an equal number of poles and zeros of $u_n$ in the large-$n$ limit}.  One can see the indicated pairing of poles with zeros in Figures~\ref{fig:8}--\ref{fig:11}, although with larger values of $n$ the phenomenon should become even more obvious to the eye.
\begin{rem}  
While tritronqu\'ee solutions are by definition asymptotically (i.e., for large $|\xi|$) pole-free in a certain sector of the complex plane, the pole-free property is not a priori guaranteed in any bounded region of the complex-plane.  However, recently it was shown \cite{CostinHT14} that all tritronqu\'ee solutions of the Painlev\'e-I equation are actually analytic down to the origin in the asymptotically pole-free sector, proving a conjecture of Dubrovin.  See \cite{Bertola12} for related results on certain solutions of the Painlev\'e-II equation \eqref{eq:PII}.  It is not known whether the tritronqu\'ee solutions $\dot{W}^\pm(\xi;m)$ of the Painlev\'e-II equation are exactly pole-free in the sector $-\tfrac{2}{3}\pi<\mathrm{Arg}(\xi)<\tfrac{2}{3}\pi$.  Because we expect pole/zero pairs of $u_n$ to converge toward fixed poles of $\dot{W}^\pm$ in the $\xi$-plane, in our opinion the plots shown in Figures~\ref{fig:8}--\ref{fig:11} are not sufficiently resolved (i.e., $n$ is not sufficiently large) to provide convincing evidence one way or the other, even though Figure~\ref{fig:11} shows some poles and zeros of $u_n$ lying in the asymptotic pole-free sector for $\dot{W}^-(\xi;\tfrac{4}{5}\ii)$ near the origin.
\end{rem}

The origin $x=0$ is a fixed singular point of the Painlev\'e-III equation \eqref{eq:PIII} and its presence appears to affect the pattern of poles and zeros of $u_n(x;m)$ close to the origin if $m\not\in\mathbb{Z}+\tfrac{1}{2}$, as can be seen in Figures~\ref{fig:1}--\ref{fig:3}.  In particular, the density of the regular distribution of poles and zeros within the domain $E$ seems to blow up as $y_0\to 0$, a phenomenon that is confirmed by the asymptotic analysis in \cite{BothnerM18}.  However, this accumulation phenomenon cannot be uniformly valid in any neighborhood of the origin because $u_n(x;m)$ is rational.  Our numerical computations suggest that the $x$-distance of the smallest poles and zeros of $u_n(x;m)$ to the origin scales as $n^{-1}$ when $n$ is large, which suggests 
introducing into \eqref{eq:PIII-again} the scaling $x=n^{-1}z$ and considering $n$ large for $m$ bounded.  Then \eqref{eq:PIII-again} becomes
\begin{equation}
\frac{\dd^2u}{\dd z^2} = \frac{1}{u}\left(\frac{\dd u}{\dd z}\right)^2 -\frac{1}{z}\frac{\dd u}{\dd z} +\frac{4u^2+4}{z}+O(n^{-1}),
\end{equation}
which is a perturbation of the parameter-free $\mathrm{PIII}_3$ equation 
\begin{equation}
\frac{\dd^2\dot{u}}{\dd z^2}=\frac{1}{\dot{u}}\left(\frac{\dd\dot{u}}{\dd z}\right)^2-\frac{1}{z}\frac{\dd\dot{u}}{\dd z}+\frac{4\dot{u}^2+4}{z}
\label{eq:PIII-special}
\end{equation}
(arising from the general Painlev\'e-III equation in the special case $\gamma=\delta=0$, see \cite[Section 2.2]{GamayunIL13}).  We may therefore expect that $u_n(n^{-1}z;m)$ should behave like a particular solution (or possibly a family of particular solutions parametrized by $m$ and/or $n$) of this limiting equation when $n$ is large and $z$ is bounded.  To explore this possibility, we plotted the poles and zeros of $u_n(n^{-1}z;m)$ in the complex $z$-plane for two different fixed values of $m$ and increasing large $n$ in Figures~\ref{fig:12} and \ref{fig:13}.  
\begin{figure}[h]
\begin{center}
\includegraphics[width=0.3\linewidth]{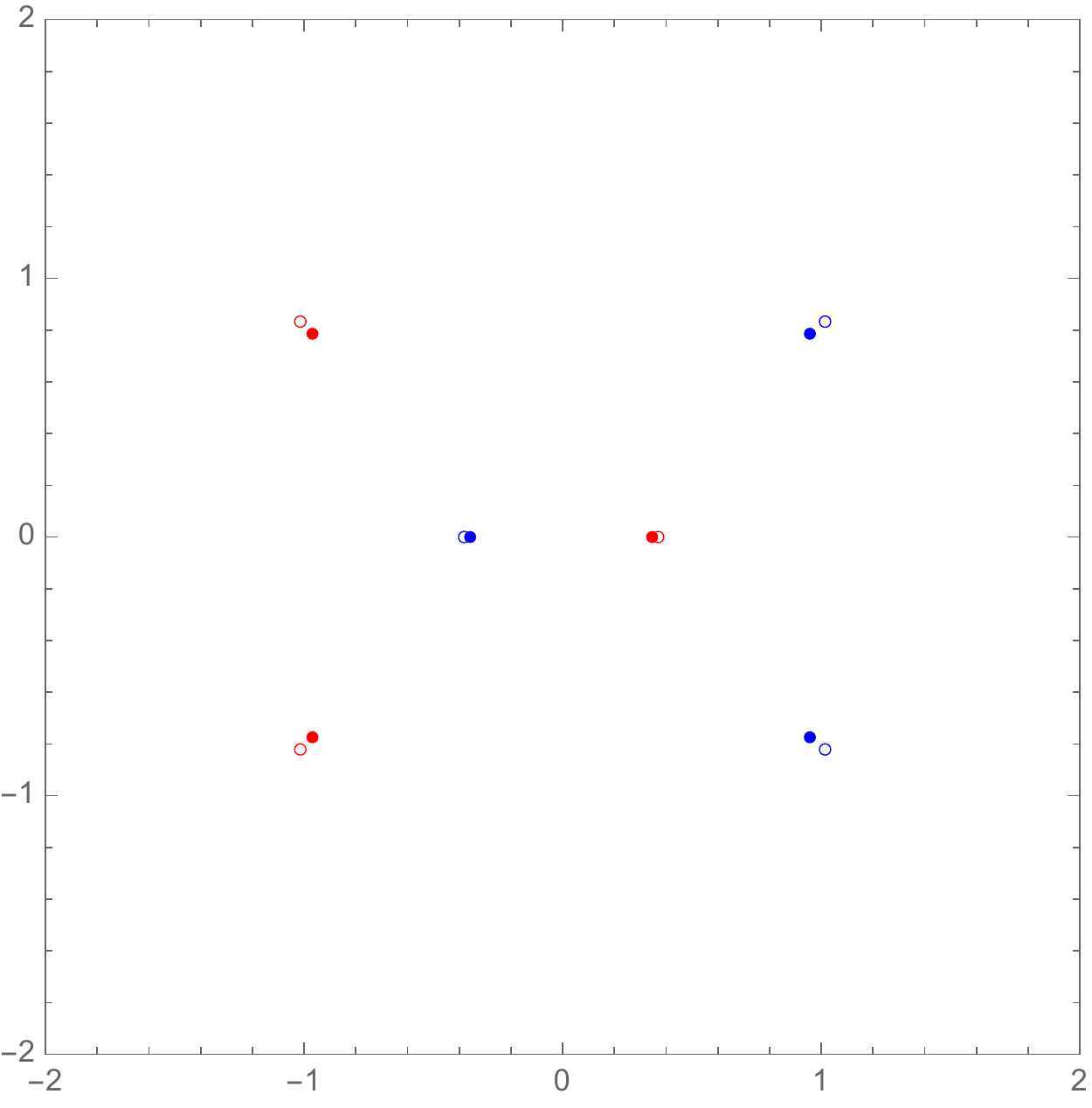}\hspace{0.03\linewidth}%
\includegraphics[width=0.3\linewidth]{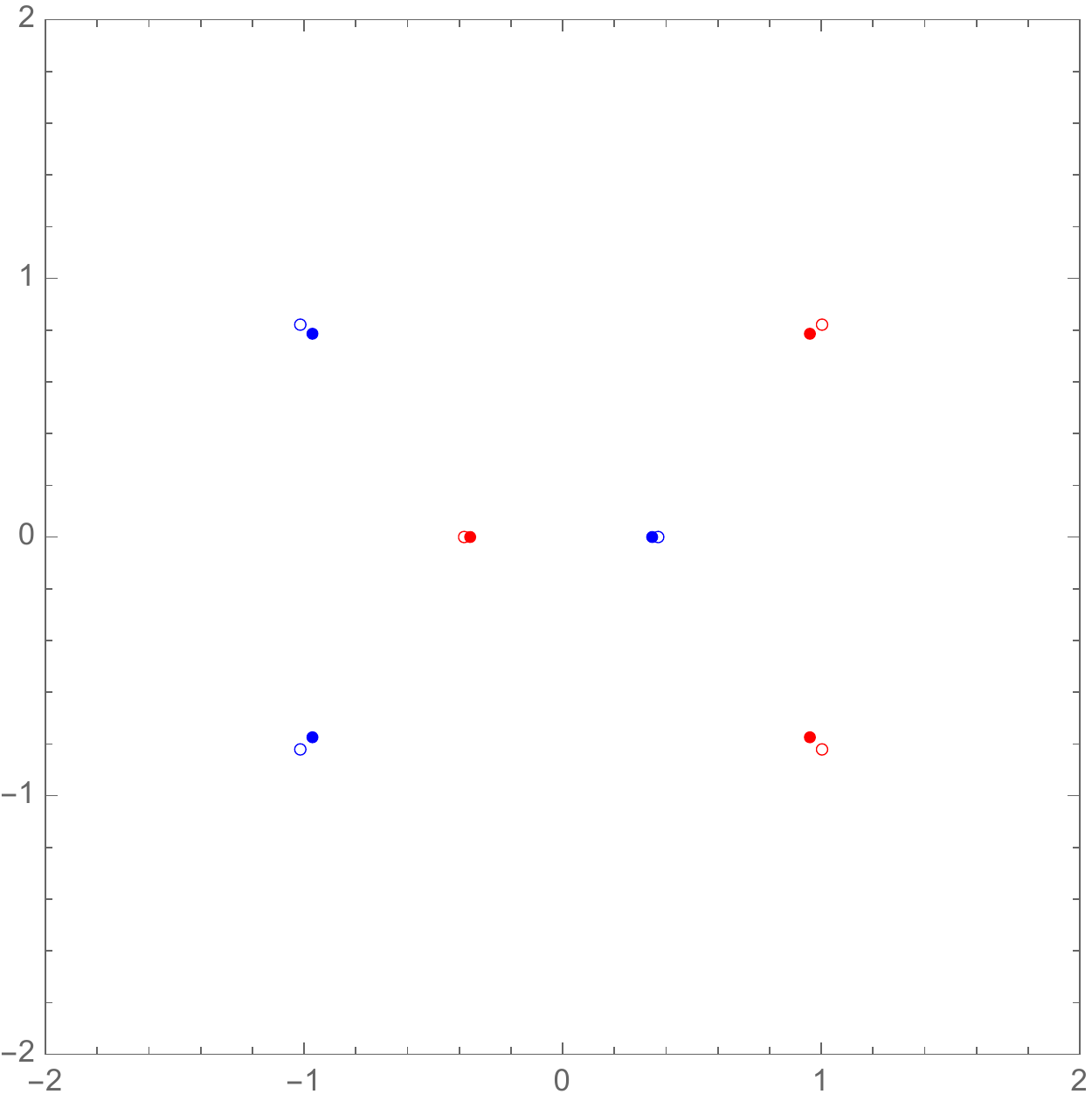}\hspace{0.03\linewidth}%
\includegraphics[width=0.3\linewidth]{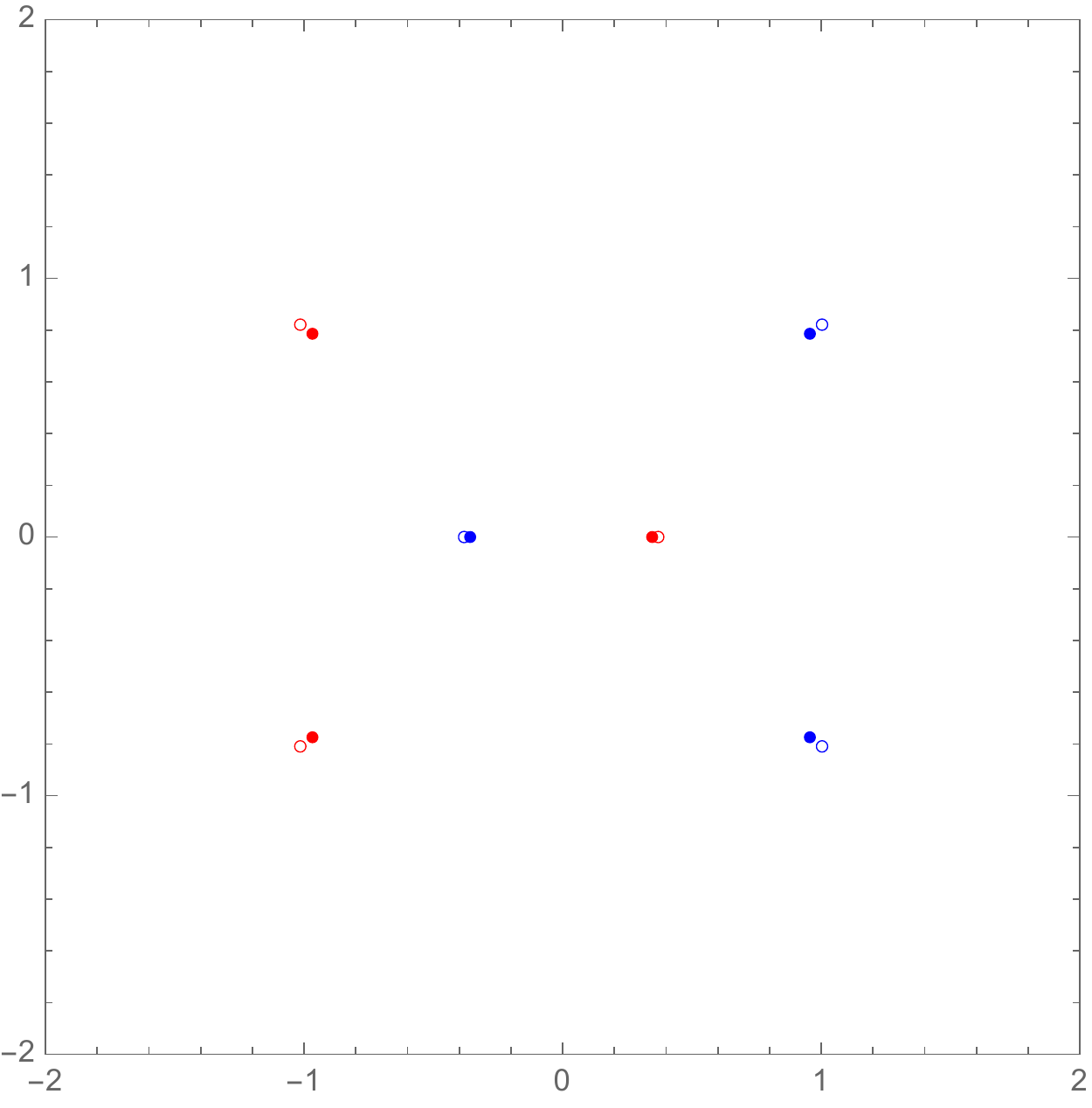}
\end{center}
\caption{As in Figure~\ref{fig:1} but plotted in the $z$-plane for $m=0$ and $n=18$ (left), $n=19$ (center), and $n=20$ (right).}
\label{fig:12}
\end{figure}
\begin{figure}[h]
\begin{center}
\includegraphics[width=0.3\linewidth]{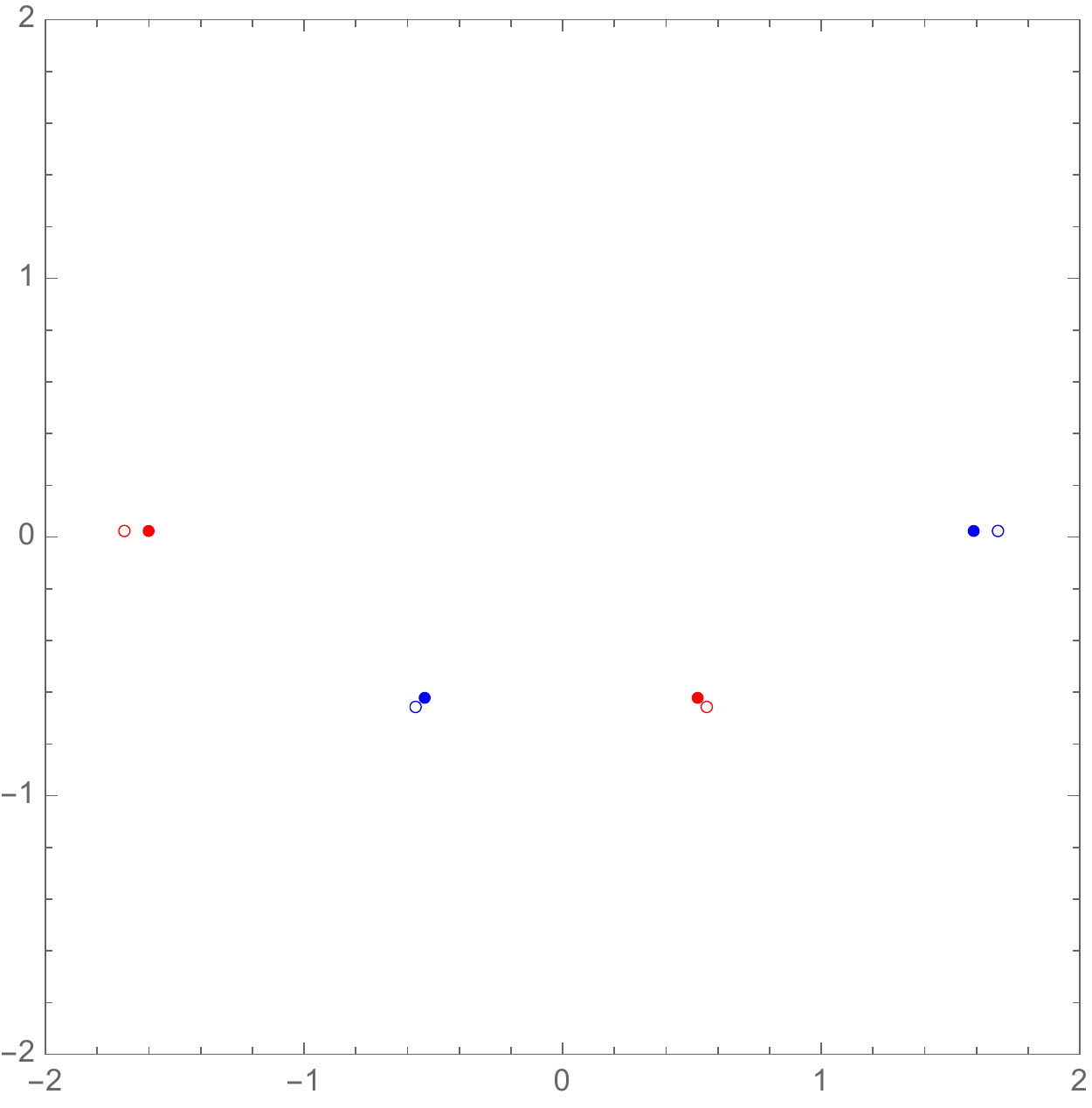}\hspace{0.03\linewidth}%
\includegraphics[width=0.3\linewidth]{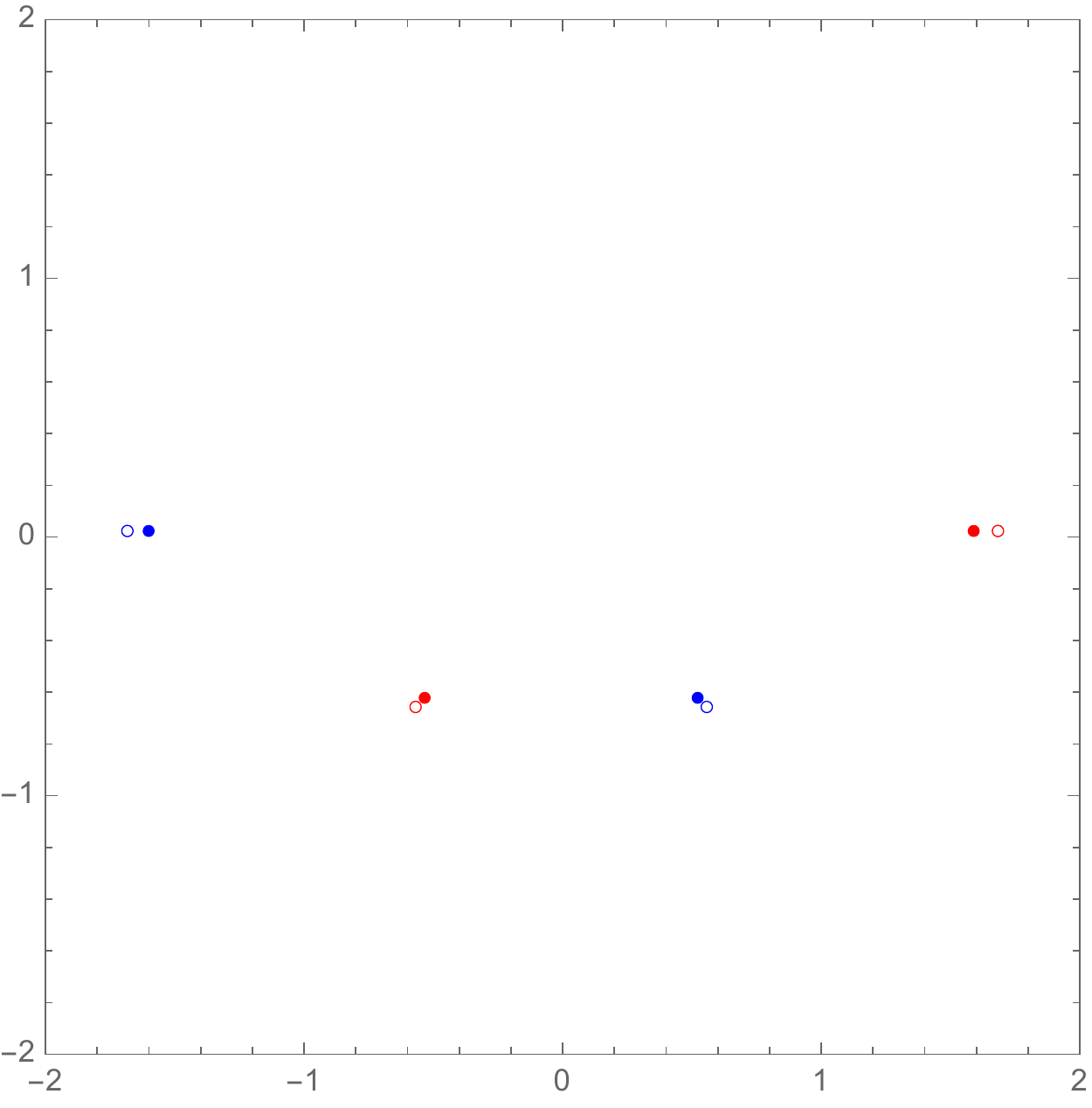}\hspace{0.03\linewidth}%
\includegraphics[width=0.3\linewidth]{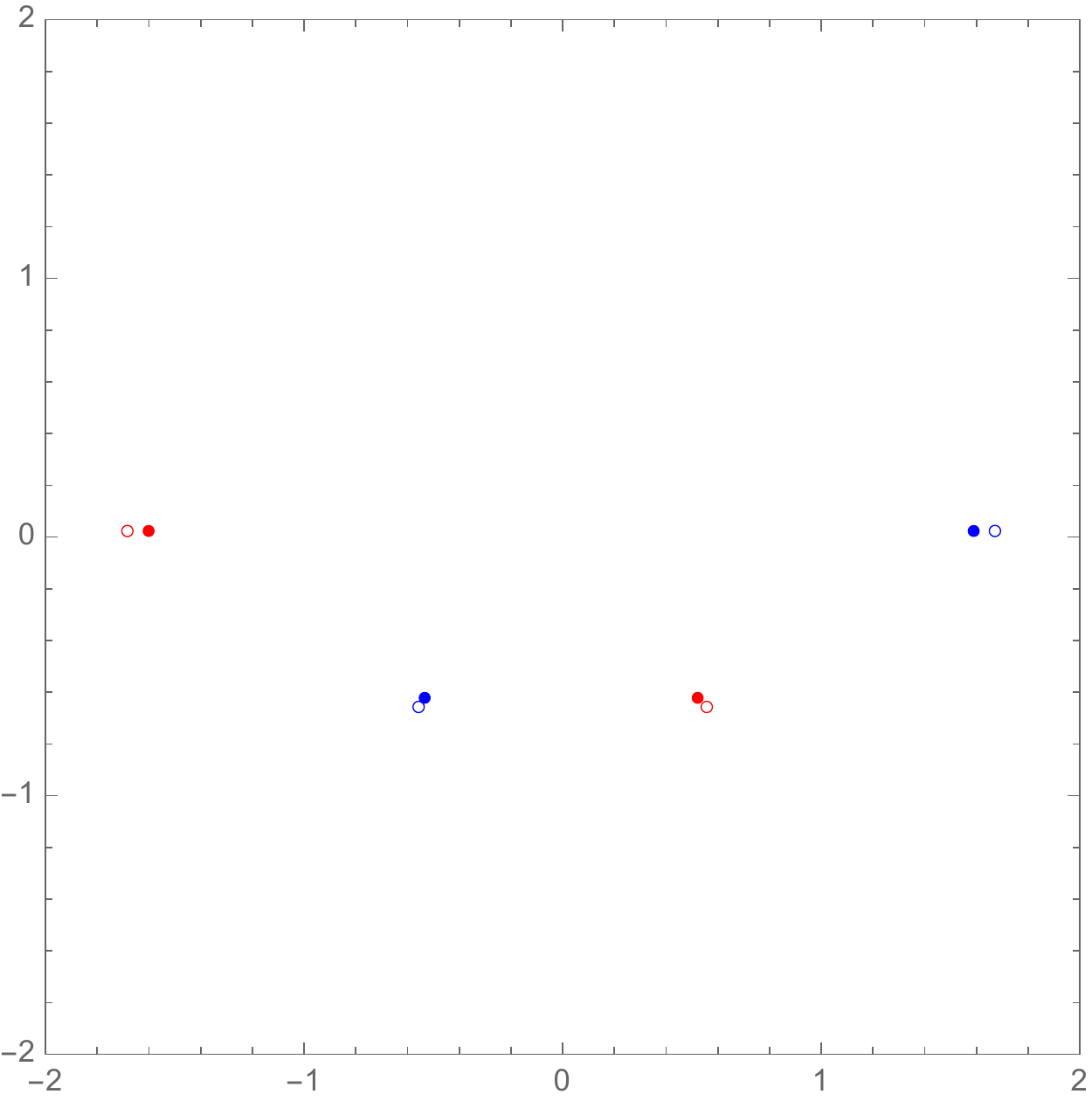}
\end{center}
\caption{As in Figure~\ref{fig:1} but plotted in the $z$-plane for $m=\tfrac{4}{5}\ii$ and $n=18$ (left), $n=19$ (center), and $n=20$ (right).}
\label{fig:13}
\end{figure}
Noting the alternation in the pattern of poles and zeros with increasing $n$ in each case and taking into account the symmetry $\dot{u}\mapsto -\dot{u}^{-1}$ of \eqref{eq:PIII-special} leads to the following conjecture.
\begin{conjecture}
Let $m\in\mathbb{C}\setminus(\mathbb{Z}+\tfrac{1}{2})$ be given.  Then there exists a corresponding particular solution $\dot{u}(z;m)$ of the $m$-independent model equation \eqref{eq:PIII-special} such that 
\begin{equation}
\lim_{j\to\infty} u_{2j}((2j)^{-1}z;m)=\dot{u}(z;m)\quad\text{and}\quad
\lim_{j\to\infty} u_{2j+1}((2j+1)^{-1}z;m)=-\dot{u}(z;m)^{-1}.
\end{equation}
\label{conjecture:origin}
\end{conjecture}
The reason for excluding half-integral values of $m$ from this statement is that $u_n(x;m)$ has either a simple pole or a simple zero at the origin \cite{ClarksonLL16} for such $m$ and asymptotic analysis \cite{BothnerM18} shows convergence to a function of $y=x/n$ (the analytic continuation of $\ii p^+_0(y)$ to the complement of the ``eyebrow''), which would correspond under rescaling either to $\dot{u}\equiv 0$ or $\dot{u}\equiv\infty$; moreover, this limit is independent of whether $n$ is odd or even.  Naturally, this discrepancy begs again the question of how the solution behaves near the origin in a double-scaling limit of large $n$ and $m$ close to a half-integer.  

The asymptotic analysis to establish Conjectures~\ref{conjecture:PII} and \ref{conjecture:origin} using Theorem~\ref{thm:RH-representation} is work in progress.   The proof of Conjecture~\ref{conjecture:origin} is expected to be particularly challenging because Riemann-Hilbert Problem~\ref{rhp:renormalized} cannot even be formulated for $x=0$.

%The case that %$P_4(\dot{V};y_0,C)$ 
%$P(\dot{\lNaught};y_0,C)$ is a perfect square turns out to be particularly relevant.  The double roots given by the latter two values in \eqref{eq:V0-four-answers} are analytic functions of $y_0$ in a neighborhood of $y_0=\infty$ and approach the values %$1$ and $-1$ 
%$-\ii$ and $\ii$ respectively (with the principal branch interpretation of the square root) as $y_0\to\infty$.
%As we will see, these two functions $\dot{\lNaught}=\lNaught_0=\lNaught_0^\pm(y_0)$ provide excellent approximations for the rational Painlev\'e-III solution %$u=V$ 
%$u=\ii \lNaught$ with seed $\pm 1$ provided that $y_0$ is outside of some specific bounded domain $E$ whose boundary contains the two branch points of $\lNaught_0^\pm(y_0)$, namely $y_0=\pm \ii/2$.  Essentially, when $y_0$ crosses $\partial E$ from outside to inside, the way that $C$ depends on $y_0$ changes so %$P_4(\dot{V};y_0,C(y_0))$ 
%$P(\dot{\lNaught};y_0,C(y_0))$ now has distinct roots, in which case the general solution of \eqref{eq:dotV-ODE} is an elliptic function with modulus depending on $y_0$.

\section{Lax Pair and Isomonodromy Theory for the Painlev\'e-III Equation}
\label{sec:isomonodromy-review}
The representation of the Painlev\'e-III equation \eqref{eq:PIII} as the compatibility condition for a Lax pair of first-order linear systems was discovered by Jimbo and Miwa \cite{JimboM81b}.  Consider the linear differential equations 
\begin{equation}
\frac{\partial\mathbf{\Psi}}{\partial\lambda}(\lambda;x)=
\mathbf{A}(\lambda;x)\mathbf{\Psi}(\lambda;x),\quad
\mathbf{A}(\lambda;x):=\frac{\ii x}{2}\sigma_3+\frac{1}{\lambda}\begin{bmatrix}-\frac{1}{2}\Theta_{\infty} & y\\ v & \frac{1}{2}\Theta_{\infty}\end{bmatrix}+\frac{1}{\lambda^2}\begin{bmatrix}\tfrac{1}{2}\ii x-\ii st & \ii s\\ -\ii t(st-x) & -\tfrac{1}{2}\ii x+\ii st\end{bmatrix},
\label{eq:lambda-eqn}
\end{equation}
and
\begin{equation}
\frac{\partial\mathbf{\Psi}}{\partial x}(\lambda;x)=\mathbf{B}(\lambda;x)\mathbf{\Psi}(\lambda;x),\quad\mathbf{B}(\lambda;x):=\frac{\ii\lambda}{2}\sigma_3+\frac{1}{x}\begin{bmatrix}0 & y\\ v& 0\end{bmatrix}-\frac{1}{\lambda x}\begin{bmatrix}\tfrac{1}{2}\ii x-\ii st & \ii s\\ -\ii t(st-x) & -\tfrac{1}{2}\ii x+\ii st\end{bmatrix}.
\label{eq:x-eqn}
\end{equation}
Here, $\Theta_\infty$ is a constant parameter and $y=y(x)$, $v=v(x)$, $s=s(x)$, and $t=t(x)$ are coefficient functions\footnote{Our parametrization of the Lax system \eqref{eq:lambda-eqn}--\eqref{eq:x-eqn} differs from that of Jimbo and Miwa \cite{JimboM81b}, who instead of $s(x)$ and $t(x)$ worked with the combinations (in the notation of \cite{FokasIKN06}) $U(x):=s(x)t(x)$ and $w(x):=t(x)^{-1}$.  The parametrization \eqref{eq:lambda-eqn}--\eqref{eq:x-eqn} has the advantage that the singularities of the potentials $y$, $v$, and $s$ are exactly the singularities of the simultaneous solution  $\mathbf{\Psi}$ with respect to the parameter $x$.} (potentials).  
The matrix coefficient of $\lambda^{-2}$ in \eqref{eq:lambda-eqn} and of $-(\lambda x)^{-1}$ in \eqref{eq:x-eqn} looks complicated, but it simply represents the most general matrix having $\pm \tfrac{1}{2}\ii x$ as its eigenvalues (all such matrices depend on two parameters whose roles are played by $s(x)$ and $t(x)$).  The compatibility condition $\mathbf{A}_x-\mathbf{B}_\lambda + [\mathbf{A},\mathbf{B}]=\mathbf{0}$ for the simultaneous equations \eqref{eq:lambda-eqn}--\eqref{eq:x-eqn} is the first-order system of nonlinear differential equations
\begin{equation}
\begin{gathered}
	x\frac{\dd y}{\dd x}=-2xs+\Theta_{\infty}y,\ \ \ \ \ \ \ x\frac{\dd v}{\dd x}=-2xt(st-x)-\Theta_{\infty}v,
\\
	x\frac{\dd s}{\dd x}=(1-\Theta_\infty)s-2xy+4yst,\ \ \ \ \ \ \ x\frac{\dd t}{\dd x}=\Theta_\infty t-2yt^2+2v.
\end{gathered}
\label{eq:PIII-system}
\end{equation}
This system admits an integral of motion:
\begin{equation}
I:=\frac{2\Theta_\infty}{x}st-\Theta_\infty-\frac{2}{x}yt(st-x)+\frac{2}{x}vs
%	I:=2\frac{U(x)}{x}\Theta_{\infty}-\Theta_\infty-2\frac{(U(x)-x)y(x)}{w(x)x}+2\frac{w(x)v(x)U(x)}{x}
\label{eq:Integral}
\end{equation}
is a conserved quantity, i.e, \eqref{eq:PIII-system} implies that $\dd I/\dd x=0$ holds identically.
%%%%%
%
%Note that I=2\widehat{J}=2J-\Theta_\infty in Thomas' notation.
%
%%%%%
Using \eqref{eq:PIII-system} one can show that the combination 
\begin{equation}
u(x):=-\frac{y(x)}{s(x)}
\label{eq:u-recover}
\end{equation}
satisfies the differential equation
\begin{equation}
	x\frac{\dd u}{\dd x}=2x-(1-2\Theta_{\infty})u+4stu^2-2xu^2.
\label{eq:first-order-PIII}
\end{equation}
Taking another $x$-derivative and letting $\Theta_0$ denote the constant value of the integral $I$ one then obtains
the Painlev\'e-III equation in the form \eqref{eq:PIII}. (For some details of these calculations, see the last lines of the proof of Lemma~\ref{l:inverse} in Section~\ref{sec:inverse-monodromy} below.) The isomonodromy method algorithm for solving the initial-value problem for \eqref{eq:PIII} with initial conditions $u(x_0)=u_0$ and $u'(x_0)=u_0'$ is then the following \cite{FokasIKN06}.  Given 
constants $(\Theta_0,\Theta_\infty,x_0,u_0,u'_0)\in\mathbb{C}^5$ with $x_0u_0\neq 0$, 
\begin{enumerate}
\item 
Choose an arbitrary nonzero initial value of $y$:  $y(x_0)=y_0\neq 0$.  Then from \eqref{eq:u-recover} at $x=x_0$ one obtains the initial value of $s$:  $s_0:=s(x_0)=-y_0/u_0$, which is well-defined and nonzero.  Next, since $s_0u_0^2=-u_0y_0\neq 0$, $t_0:=t(x_0)$ is well-defined from \eqref{eq:first-order-PIII} at $x=x_0$:
\begin{equation}
t_0=\frac{1}{4u_0y_0}\left(2x_0-(1-2\Theta_\infty)u_0-2x_0u_0^2-x_0u_0'\right).
\end{equation}
Finally, from \eqref{eq:Integral} using $I=\Theta_0$ and substituting for $s_0$ and $t_0$ we get the initial value of $v$:  $v_0:=v(x_0)$ where
\begin{equation}
v_0=\frac{1}{16y_0u_0^2}\left(4x_0^2+(1-4\Theta_\infty^2)u_0^2-4x_0u_0-8\Theta_0x_0u_0^3-4x_0^2u_0^4-4x_0^2u_0'+2x_0u_0u_0'+x_0^2u_0^{\prime 2}\right).
\end{equation}
%Determine $U(x_0)$ from \eqref{eq:first-order-PIII}:
%\begin{equation}
%U(x_0)=\frac{1}{4u_0^2}\left(x_0u_0'-2x_0+u_0(1-2\Theta_\infty)+2x_0u_0^2\right),
%\label{eq:U-zero}
%\end{equation}
%finite and nonzero by the genericity assumption.  Then take an arbitrary value of
%$y(x_0)\neq 0$, and 
%obtain $w(x_0)$ from \eqref{eq:u-recover}:
%\begin{equation}
%w(x_0)=-\frac{y(x_0)}{u_0U(x_0)},
%\label{eq:w-zero}
%\end{equation}
%which is also finite and nonzero.
%Finally, obtain $v(x_0)$ from \eqref{eq:Integral} using $I=\Theta_0$:
%\begin{equation}
%v(x_0)=\frac{x_0}{2w(x_0)U(x_0)}\left(\Theta_0+\Theta_\infty-2\frac{U(x_0)}{x_0}\Theta_\infty+2\frac{(U(x_0)-x_0)y(x_0)}{w(x_0)x_0}\right).
%\label{eq:v-zero}
%\end{equation}
Note that $s_0$ is proportional, while $t_0$ and $v_0$ are inversely proportional, to the arbitrary\footnote{Given any constant $\alpha\neq 0$, the system of equations \eqref{eq:PIII-system} is obviously invariant under the substitution $(y(x),v(x),s(x),t(x))\mapsto (\alpha y(x),\alpha^{-1}v(x),\alpha s(x),\alpha^{-1}t(x))$, which also leaves $u(x)$ defined by \eqref{eq:u-recover} invariant.} nonzero constant $y_0$.
\item Taking $y=y_0$, $v=v_0$, $s=s_0$, $t=t_0$, and $x=x_0\neq 0$, 
seek four specific fundamental solution matrices of \eqref{eq:lambda-eqn} called \emph{canonical solutions}, namely two satisfying the normalization condition
\begin{equation}
\mathbf{\Psi} \lambda^{\Theta_\infty\sigma_3/2}\ee^{-\ii x\lambda\sigma_3/2}\to\mathbb{I},\quad\lambda\to\infty
\label{eq:infinity-norm}
\end{equation}
in two different abutting sectors with opening angle $\pi$ and bisected by directions in which the factors $\ee^{\pm \ii x\lambda}$ are oscillatory; and two satisfying the normalization condition
\begin{equation}
\begin{bmatrix}a(x) & b(x)s(x)\\a(x)t(x) & b(x)(s(x)t(x)-x)\end{bmatrix}^{-1}\mathbf{\Psi}\lambda^{-\Theta_0\sigma_3/2}\ee^{\ii x\lambda^{-1}\sigma_3/2}\to\mathbb{I},\quad\lambda\to 0,
\label{eq:zero-norm-general}
\end{equation}
in two different abutting sectors with opening angle $\pi$ and bisected by directions in which the factors $\ee^{\pm \ii x\lambda^{-1}}$ are oscillatory.  In \eqref{eq:zero-norm-general}, $a(x)$ and $b(x)$ are arbitrary except that the determinant of the matrix factor on the left should be equal to $1$ and therefore $a(x)b(x)=-x^{-1}$.
The two fundamental matrices near $\lambda=0$ are obviously related by right-multiplication by one $\lambda$-independent \emph{Stokes matrix} for each of the two sector boundary arcs; similarly for the fundamental solution matrices near $\lambda=\infty$.  A fifth \emph{connection matrix} relates the solution in one sector near $\lambda=0$ to that in one sector near $\lambda=\infty$. The four Stokes matrices and the connection matrix constitute the solution of the \emph{direct monodromy problem}.  
\item The equation \eqref{eq:x-eqn} implies that the Stokes matrices and the connection matrix are independent of $x$ when $y$, $v$, $s$, and $t$ evolve according to \eqref{eq:PIII-system}; this is the isomonodromy property of the representation \eqref{eq:lambda-eqn}--\eqref{eq:x-eqn}.  Hence, letting $x\in\mathbb{C}$ be arbitrary, solve the \emph{inverse monodromy (Riemann-Hilbert) problem} of determining the four fundamental solution matrices from the jump conditions relating them via right-multiplication by the Stokes matrices and the connection matrix and from the asymptotic normalization conditions \eqref{eq:infinity-norm}--\eqref{eq:zero-norm}. From the solution of this problem the coefficients $(y,v,s,t)$ of equation \eqref{eq:lambda-eqn} can then be extracted and from them $u$ is obtained for $x\neq x_0$ from \eqref{eq:u-recover}.
\end{enumerate}

%The genericity assumption can be weakened to just the assumption that $x_0u_0\neq 0$ by means of the observation that it is only necessary to be able to define the coefficients in the Lax pair \eqref{eq:lambda-eqn}--\eqref{eq:x-eqn}.  With $x_0u_0\neq 0$, $U$ is well-defined at $x=x_0$ by \eqref{eq:U-zero}, which by the weakened assumption may vanish.  Nonetheless, given $y(x_0)\neq 0$ arbitrary, we reinterpret \eqref{eq:w-zero} as properly defining both $w^{-1}$ and the product $wU$ at $x=x_0$ even if $U(x_0)=0$.  Since $wU=-y(x_0)/u_0\neq 0$ at $x=x_0$, $v(x_0)$ is then well-defined by \eqref{eq:v-zero}.    

\section{Monodromy Data for $u(x)=u_0(x;m)=1$}
\label{sec:direct-monodromy}
In the special case that $\Theta_0=\Theta_\infty-1$, i.e., $n=0$ for arbitrary $m\in\mathbb{C}$, the Painlev\'e-III equation \eqref{eq:PIII} has the rational (constant) solutions $u(x)=\pm 1$.
Our aim in this section is to calculate the necessary monodromy data so that the solution $u(x)=1$ can be obtained from an appropriate Riemann-Hilbert problem.  Although this appears to involve the study of the direct problem \eqref{eq:lambda-eqn} alone, our approach will be to leverage the compatibility with the isomonodromic deformation \eqref{eq:x-eqn} to solve the latter equation instead and then build in additional dependence on $\lambda$ via integration constants to satisfy \eqref{eq:lambda-eqn} as well.  With these results in hand, in Section~\ref{sec:Schlesinger} we will apply Schlesinger transformations to increment/decrement by $2$ the value of the difference $\Theta_\infty-\Theta_0=1-2n$ and thus obtain a Riemann-Hilbert representation for the B\"acklund chain of rational solutions with seed solution $u(x)=1$.

\subsection{The Lax pair for $\Theta_0=\Theta_\infty-1$ and $u(x)=1$}
\label{sec:LaxPair-n0}
Since we will be exploiting the differential equation \eqref{eq:x-eqn} to construct the monodromy data, we need to know how the coefficients $(y,v,s,t)$ depend on $x$.  
From \eqref{eq:u-recover} with $u(x)\equiv 1$ we find that $s(x)\equiv -y(x)$, so the differential equation for $y(x)$ in \eqref{eq:PIII-system} closes as a linear equation with solution
\begin{equation}
y(x)=-\frac{1}{4}K\ee^{2x}x^{\Theta_\infty} \quad\text{and hence also}\quad
s(x)=\frac{1}{4}K\ee^{2x}x^{\Theta_\infty},
\label{eq:y-special}
\end{equation}
where $K\neq 0$ is an arbitrary constant of integration.  Using this result and $u(x)\equiv 1$ in \eqref{eq:first-order-PIII} we obtain $t(x)$:
\begin{equation*}
t(x)=(1-2\Theta_\infty)K^{-1}\ee^{-2x}x^{-\Theta_\infty}.
\end{equation*}
Finally, using these along with $I=\Theta_0=\Theta_\infty-1$ in \eqref{eq:Integral}, we solve for $v(x)$:
\begin{equation*}
v(x)=-\frac{1}{4}(1-2\Theta_\infty)(4x+1+2\Theta_\infty)K^{-1}\ee^{-2x}x^{-\Theta_\infty}.
\end{equation*}
%Taking $\Theta_0=\Theta_\infty-1$ and assuming the initial conditions $u_0=1$ and $u_0'=0$ produces the seed solution $u(x)=1$ of \eqref{eq:PIII} for any initial point $x_0\neq 0$.  Using \eqref{eq:U-zero}--\eqref{eq:v-zero} we then obtain the initial values
%\begin{equation}
%U(x_0)=\frac{1}{4}(1-2\Theta_\infty),\quad y(x_0)=C,\quad w(x_0)=\frac{4C}{2\Theta_\infty-1},\quad v(x_0)=\frac{1}{16C}(1+4x_0-8\Theta_\infty x_0).
%\end{equation}
In order that the coefficients in the Lax pair are well-defined, we assume for the purposes of this calculation that $x\in\mathbb{C}\setminus\mathbb{R}_-$ and agree to label the argument of $x$ as being in the interval $(-\pi,\pi)$, i.e., we use the principal branch $\arg(x)=\mathrm{Arg}(x)$.  The arbitrary constant $K$ plays a similar role as the arbitrary nonzero initial value $y_0=y(x_0)$ in the solution of the initial-value problem for \eqref{eq:PIII} by the isomonodromy method. Next, introducing into \eqref{eq:x-eqn} the well-defined substitution 
\begin{equation*}
\mathbf{\Psi}=\ee^{x\sigma_3}x^{\Theta_\infty\sigma_3/2}x^{-1/2}\mathbf{W},
\end{equation*}
one finds that the first-row matrix entries $W_{1j}$
are solutions $W$ of the confluent hypergeometric equation (cf., \cite[Eq.~13.14.1]{DLMF})
\begin{equation}
\frac{\dd^2W}{\dd\zeta^2}+\left[-\frac{1}{4}+\frac{\kappa}{\zeta}+\frac{1-4\mu^2}{4\zeta^2}\right]W=0,\quad \mu=\frac{1}{4},\quad \kappa=\frac{1}{2}(\Theta_\infty-1),
\label{eq:confluent-hypergeometric}
\end{equation}
where $\zeta:=\ii x(\lambda+2\ii-\lambda^{-1})$.  The elements $W_{2j}$ of the second row are obtained from those in the first row by the formula
\begin{equation*}
W_{2j}=-\frac{4\zeta(W_{1j}'(\zeta)-\tfrac{1}{2}W_{1j}(\zeta))+(4\kappa-\ii (1-2\Theta_\infty)\lambda^{-1})W_{1j}(\zeta)}{K(1+\ii\lambda^{-1})}.
\end{equation*}
%Here, $C_+$ is an arbitrary constant that appears in the coefficient matrices $\mathbf{A}$ and $\mathbf{B}$ \textcolor{magenta}{as described in Thomas' notes}.  
If we fix a fundamental pair of solutions of \eqref{eq:confluent-hypergeometric} that depend on $\lambda$ only through the variable $\zeta$ as the first row of the matrix $\mathbf{W}$, then the general solution of \eqref{eq:x-eqn} can be written in the form
\begin{equation}
\mathbf{\Psi}=\ee^{x\sigma_3}x^{\Theta_\infty\sigma_3/2}x^{-1/2}\mathbf{W}\mathbf{C}(\lambda),
\label{eq:Psi-general}
\end{equation}
where $\mathbf{C}(\lambda)$ cannot depend on $x$ but might depend on $\lambda$. Having found the general solution of the ``$x$-equation'' \eqref{eq:x-eqn} in the Lax pair for the Painlev\'e-III equation, we can now determine $\mathbf{C}(\lambda)$ such that the expression \eqref{eq:Psi-general} is simultaneously a solution of both (compatible, because $y(x)$, $v(x)$, $w(x)$, and $U(x)$ satisfy \eqref{eq:PIII-system}) equations \eqref{eq:lambda-eqn}--\eqref{eq:x-eqn}.  Upon substitution of \eqref{eq:Psi-general} into \eqref{eq:lambda-eqn} one easily finds that 
\begin{equation*}
\mathbf{C}(\lambda)=(\lambda+\ii)^{-1/2}\mathbf{C}, 
\end{equation*}
where $\mathbf{C}$ is a matrix independent of both $x$ and $\lambda$.  

\subsection{Normalized simultaneous solutions for $\mathrm{Im}(x)\neq 0$}
\label{sec:normalized-solutions}
For the moment, we assume that $\mathrm{Im}(x)\neq 0$ and define $x^p$ (e.g., in \eqref{eq:Psi-general}) by taking $\arg(x)=\mathrm{Arg}(x)\in (-\pi,\pi)$.  Later in Section~\ref{sec:real-axis-limit} we will consider the exceptional cases $\arg(\pm x)=0$. Our goal now is to determine the values of the matrix $\mathbf{C}$ in order to define the four canonical fundamental solution matrices satisfying the normalization conditions \eqref{eq:infinity-norm}--\eqref{eq:zero-norm-general}.  Note that \eqref{eq:zero-norm-general} here takes the form
\begin{equation}
\begin{bmatrix} a(x) & b(x)\tfrac{1}{4}K\ee^{2x}x^{\Theta_\infty}\\ a(x)K^{-1}(1-2\Theta_\infty)\ee^{-2x}x^{-\Theta_\infty} & \displaystyle b(x)\tfrac{1}{4}(1-2\Theta_\infty-4x)\end{bmatrix}^{-1}\mathbf{\Psi}\lambda^{-\Theta_0\sigma_3/2}\ee^{\ii x\lambda^{-1}\sigma_3/2}\to\mathbb{I},\quad\lambda\to 0
\label{eq:zero-norm}
\end{equation}
where
\begin{equation}
a(x)b(x)=-\frac{1}{x}.
\label{eq:unimodularity}
\end{equation}
To specify these four solutions carefully, we should make sure that the power functions $\lambda^p$ for various $p$ appearing in the normalization conditions, as well as the scalar factor $(\lambda+\ii)^{-1/2}$ 
and the solutions $W$ of the confluent hypergeometric equation \eqref{eq:confluent-hypergeometric} that are chosen for the first row of the matrix $\mathbf{W}$ are all unambiguous.  We do this as follows.  Firstly, we note that according to the Wronskian identity \cite[Eq.~13.14.30]{DLMF}, we may choose as a fundamental pair of solutions of \eqref{eq:confluent-hypergeometric} the two Whittaker functions $W_{11}:=W_{-\kappa,\mu}(-\zeta)$ and $W_{12}:=W_{\kappa,\mu}(\zeta)$.  Now, $W_{\pm\kappa,\mu}(z)$ are multi-valued functions, and to be completely unambiguous we select in both cases the principal branches, whose argument $z$ lies in the domain $\arg(z)\in (-\pi,\pi)$.  These solutions are related by the identity (cf., \cite[Eq.~13.14.13]{DLMF})
\begin{equation}
\lim_{\epsilon\downarrow 0}W_{\pm\kappa,\mu}(-z+\ii\epsilon)=\ee^{\pm 2\pi \ii\kappa}\lim_{\epsilon\downarrow 0}W_{\pm\kappa,\mu}(-z-\ii\epsilon)+\frac{2\pi \ii \ee^{\pm \ii\pi\kappa}}{\Gamma(\tfrac{1}{2}+\mu\mp\kappa)\Gamma(\tfrac{1}{2}-\mu\mp\kappa)}W_{\mp\kappa,\mu}(z),\quad z>0,
\label{eq:W-jump}
\end{equation}
and its (negative) derivative
\begin{equation}
\lim_{\epsilon\downarrow 0}W'_{\pm\kappa,\mu}(-z+\ii\epsilon)=\ee^{\pm 2\pi \ii\kappa}\lim_{\epsilon\downarrow 0}W'_{\pm\kappa,\mu}(-z-\ii\epsilon)-\frac{2\pi \ii \ee^{\pm \ii\pi\kappa}}{\Gamma(\tfrac{1}{2}+\mu\mp\kappa)\Gamma(\tfrac{1}{2}-\mu\mp\kappa)}W'_{\mp\kappa,\mu}(z),\quad z>0,
\label{eq:W-prime-jump}
\end{equation}
which express jump conditions for $W_{\pm\kappa,\mu}(z)$ and its derivative across the branch cut on the negative real $z$-axis.  We also have the asymptotic behavior (cf., \cite[Eq.~13.14.21]{DLMF})
\begin{equation*}
W_{\pm\kappa,\mu}(z)=\ee^{-z/2}z^{\pm\kappa}(1+O(z^{-1})),\quad z\to\infty,\quad \arg(z)\in (-\pi,\pi),
\end{equation*}
as well as 
\begin{equation*}
\zeta (W_{11}'(\zeta)-\tfrac{1}{2}W_{11}(\zeta))=-\kappa \ee^{\zeta/2}(-\zeta)^{-\kappa}(1+O(\zeta^{-1})),\quad \zeta\to\infty,\quad\arg(-\zeta)\in (-\pi,\pi),
\end{equation*}
and
\begin{equation*}
\zeta (W_{12}'(\zeta)-\tfrac{1}{2}W_{12}(\zeta))=-\ee^{-\zeta/2}\zeta^{\kappa+1}(1+O(\zeta^{-1})),\quad\zeta\to\infty,\quad\arg(\zeta)\in (-\pi,\pi),
\end{equation*}
and in these last three relations the indicated power functions all have their principal values.  Now, with the principal branches selected, given $\mathrm{Arg}(x)\in (-\pi,\pi)$, the matrix $\mathbf{W}$ becomes a well-defined analytic function of $\lambda$, henceforth denoted $\mathbf{W}=\mathbf{W}(x,\lambda)$, defined in the complement of the preimage under $\zeta$ of the real axis.  This $x$-dependent preimage is therefore the jump contour $L$ for $\mathbf{W}$, and it takes different forms for $-\pi<\mathrm{Arg}(x)<0$ and $0<\mathrm{Arg}(x)<\pi$; see Figure~\ref{fig:jump-contour}.  
%(We will deal with the exceptional cases $\mathrm{Arg}(\pm x)=0$ later; for now we assume that $\mathrm{Im}(x)\neq 0$ and we define powers of $x$ by principal branches using $\mathrm{Arg}(x)\in (-\pi,\pi)$.)
\begin{figure}[h]
\begin{center}
\includegraphics{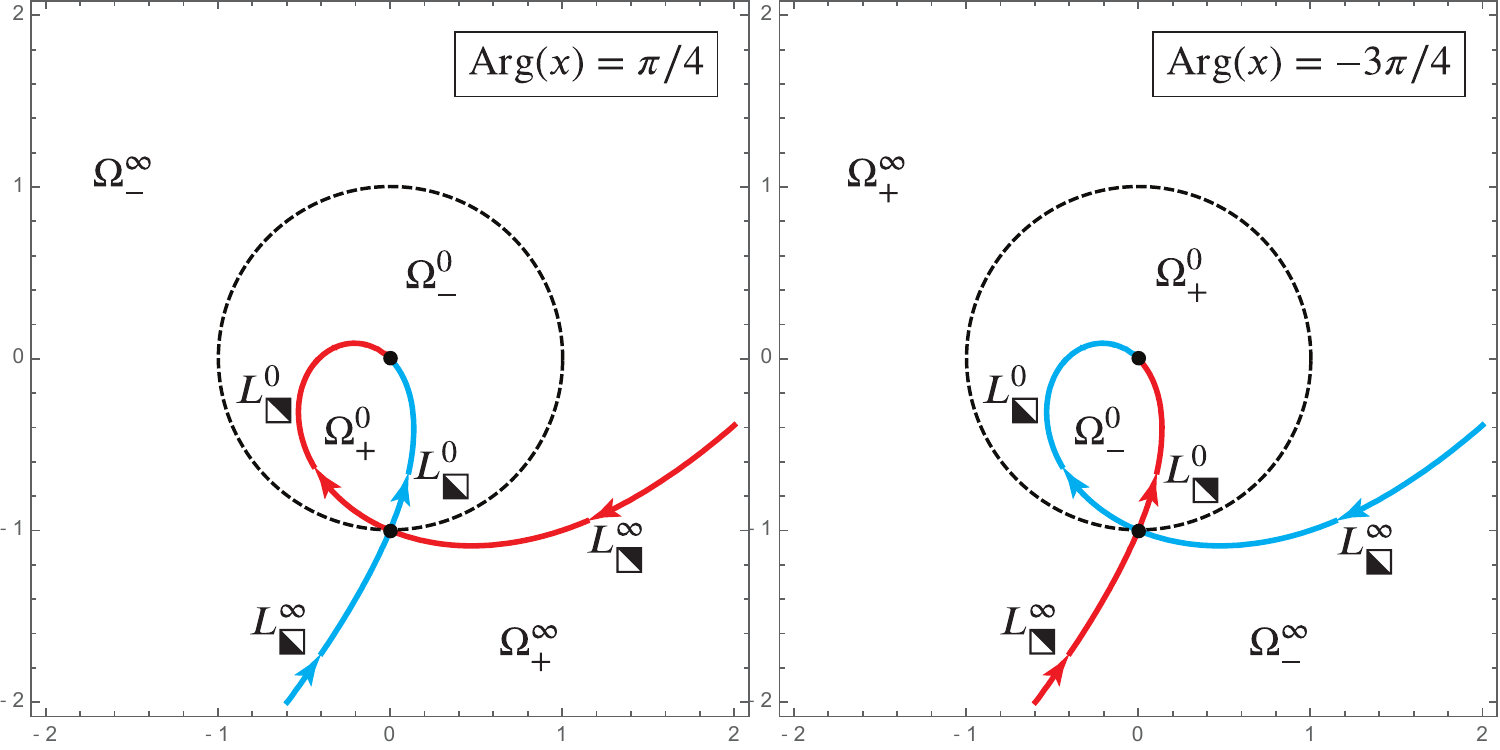}
\end{center}
\caption{The jump contour $L$ for the Whittaker matrix $\mathbf{W}(x,\lambda)$ takes a different form depending on whether $0<\mathrm{Arg}(x)<\pi$ (left) or $-\pi<\mathrm{Arg}(x)<0$ (right).  The arcs $\LInftyRed$ and $\LZeroRed$ (red) are where $\zeta<0$, and the arcs $\LInftyBlue$ and $\LZeroBlue$ (cyan) are where $\zeta>0$.  All four contour arcs meet at the only zero of $\zeta$, namely $\lambda=-\ii$.  Together with the unit circle (dotted), the contour arcs divide the complex $\lambda$-plane into four disjoint domains as indicated, $\Omega_\pm^0$ adjacent to $\lambda=0$ and where $\pm\mathrm{Im}(\zeta)>0$ holds, and unbounded domains $\Omega_\pm^\infty$ where $\pm\mathrm{Im}(\zeta)>0$ holds.
%$\Omega_1^\infty$, $\Omega_2^\infty$, $\Omega_1^0$, and $\Omega_2^0$.  
The subscript notation $\squarellblack$/$\squareurblack$ on the contour arcs is a mnemonic for the lower/upper triangular structure of jump matrices defined below (cf., \eqref{eq:Red-Jumps}--\eqref{eq:Blue-Jumps}) that will be carried by the corresponding contour arcs.}
\label{fig:jump-contour}
\end{figure}
%To be concrete, let us suppose that $\mathrm{Im}(x)>0$ and define powers of $x$ by principal branches using $0<\mathrm{Arg}(x)<\pi$.  The jump contour for $\mathbf{W}$ is then the union of the red and blue curves shown in Figure~\ref{fig:jump-contour}.
Given a value of $x$ with $\mathrm{Im}(x)\neq 0$ and a corresponding jump contour $L$ as illustrated in this figure, we will now define the multivalued functions $\lambda^p$ and $(\lambda+\ii)^{-1/2}$ precisely as follows.  For $\lambda^p$, we take as a branch cut $\LInftyBlue\cup \LZeroBlue$.  Furthermore, noting that as $x$ varies in the upper half-plane $\LInftyBlue$ sweeps through the left half $\lambda$-plane, we define $\arg(\lambda)=0$ for sufficiently large positive $\lambda$ when $\mathrm{Im}(x)>0$.  Similarly, as $x$ varies in the lower half-plane $\LInftyBlue$ sweeps through the right half $\lambda$-plane and we therefore define $\arg(\lambda)=\pi$ for $\lambda<0$ of sufficiently large magnitude when $\mathrm{Im}(x)<0$.  This choice of branch along with the cut $\LInftyBlue\cup\LZeroBlue$ unambiguously determines $\arg(\lambda)$ and hence $\lambda^p$ for any $p\in\mathbb{C}$ given $x$ with $\mathrm{Im}(x)\neq 0$.  We use the notation $\OurPower{\lambda}{p}$ to indicate this branch.  Note that if $\OurArg{\lambda}$ denotes the value of the argument corresponding to this choice of branch we have
\begin{equation}
-\frac{\pi}{2}-\mathrm{Arg}(x)<\OurArg{\lambda}<\frac{3\pi}{2}-\mathrm{Arg}(x),\quad |\lambda|\to\infty,
\label{eq:arg-lambda-infinity}
\end{equation}
while
\begin{equation}
\mathrm{Arg}(x)-\frac{\pi}{2}<\OurArg{\lambda}<\mathrm{Arg}(x)+\frac{3\pi}{2},\quad |\lambda|\to 0.
\label{eq:arg-lambda-zero}
\end{equation}
Then, to define $(\lambda+\ii)^{-1/2}$, we select $\LInftyBlue$ as the branch cut and for $\mathrm{Im}(x)>0$ we take $(\lambda+\ii)^{-1/2}$ to be positive for sufficiently positive values of $\lambda+\ii$, while for $\mathrm{Im}(x)<0$ we take $(\lambda+\ii)^{-1/2}$ to be negative imaginary for sufficiently negative values of $\lambda+\ii$.  We denote the resulting well-defined function as $\OurPower{(\lambda+\ii)}{-1/2}$.  With this choice, we have in particular that 
\begin{equation}
\OurPower{(\lambda+\ii)}{-1/2}=\ee^{-\ii\pi/4}+O(\lambda),\quad\lambda\to 0,
\label{eq:factor-small}
\end{equation}
and 
\begin{equation}
\OurPower{(\lambda+\ii)}{-1/2}=\OurPower{\lambda}{-1/2}(1+O(\lambda^{-1})),\quad\lambda\to\infty.
\label{eq:factor-large}
\end{equation}
With these definitions in hand, we now construct the four normalized solutions for $u(x)=1$ as analytic functions of $\lambda$ in the four disjoint domains $\Omega^\infty_\pm$ and $\Omega^0_\pm$.  We will denote the resulting piecewise-analytic simultaneous matrix solution of \eqref{eq:lambda-eqn}--\eqref{eq:x-eqn} by $\mathbf{\Psi}(\lambda;x)$.
\subsubsection{Defining $\mathbf{\Psi}(\lambda;x)$ for $\lambda\in\Omega_+^\infty$}
We define $\mathbf{\Psi}(\lambda;x)$ for $\lambda\in\Omega_+^\infty$ by the formula
\begin{equation}
\mathbf{\Psi}(\lambda;x)=\ee^{x\sigma_3}x^{\Theta_\infty\sigma_3/2}x^{-1/2}\OurPower{(\lambda+\ii)}{-1/2}\mathbf{W}(x,\lambda)\mathbf{C}_+^\infty,\quad\lambda\in\Omega_+^\infty,
\label{eq:PsiInftyPlus}
\end{equation}
and we determine the constant matrix $\mathbf{C}_+^\infty$ so that $\mathbf{\Psi}=\mathbf{\Psi}(\lambda;x)$ satisfies \eqref{eq:infinity-norm} (with $\lambda^{\Theta_\infty\sigma_3/2}$ defined carefully as $\OurPower{\lambda}{\Theta_\infty\sigma_3/2}$) in the limit $\lambda\to\infty$ in $\Omega_+^\infty$.  Note that the precisely-defined factor $\OurPower{(\lambda+\ii)}{-1/2}$ satisfies \eqref{eq:factor-large}, and that
%\begin{equation}
%\OurPower{(\lambda+\ii)}{-1/2}=\lambda^{-1/2}(1+O(\lambda^{-1})),\quad\lambda\to\infty,
%%\label{eq:factor-large}
%\end{equation}
%where the power function $\lambda^{-1/2}$ is defined subject to the argument condition \eqref{eq:arg-lambda-infinity}.  
when $\lambda\to\infty$ the Whittaker matrix $\mathbf{W}(x,\lambda)$ takes the following asymptotic form:
\begin{equation}
\mathbf{W}(x,\lambda)=\left(\begin{bmatrix}1 & 1\\
0 & 4K^{-1}\zeta 
\end{bmatrix}+O(\lambda^{-1})\right)\begin{bmatrix}\ee^{\zeta/2}(-\zeta)^{-\kappa} & 0\\0 & \ee^{-\zeta/2}\zeta^\kappa\end{bmatrix},\quad\lambda\to\infty.
\label{eq:W-large}
\end{equation}
This can be further simplified by recalling that $\zeta=\ii x(\lambda+2i-\lambda^{-1})$ is large when $\lambda$ is large, and making use of the fact that the expressions $(\pm\zeta)^{\pm\kappa}$ refer to the principal branch.  
Indeed, by definition $\mathrm{Im}(\zeta)>0$ and $\mathrm{Im}(-\zeta)<0$ hold for $\lambda$ in the domain $\Omega_+^\infty$.  Therefore to define $(-\zeta)^{-\kappa}$ by the principal branch we need to have $-\pi<\arg(-\zeta)<0$ or, for large $\lambda$, $-\pi<\arg(-\ii x\lambda(1+O(\lambda^{-1}))<0$.  Writing $\arg(-\ii x\lambda(1+O(\lambda^{-1})))=-\tfrac{1}{2}\pi + \mathrm{Arg}(x)+\OurArg{\lambda} + \mathrm{Arg}(1+O(\lambda^{-1})) + 2\pi \ell$, $\ell\in\mathbb{Z}$, where $\OurArg{\lambda}$ satisfies (according to Figure~\ref{fig:jump-contour} and \eqref{eq:arg-lambda-infinity} for large $\lambda\in\Omega_1^\infty$) $\OurArg{\lambda}+\mathrm{Arg}(x)\in (-\tfrac{1}{2}\pi,\tfrac{1}{2}\pi)$, 
%and $n\in\mathbb{Z}$, 
we see that $\ell=0$, and therefore 
$(-\zeta)^{-\kappa}=\ee^{\ii \pi\kappa/2}x^{-\kappa}\OurPower{\lambda}{-\kappa}(1+O(\lambda^{-1}))$ as $\lambda\to\infty$ in $\Omega_+^\infty$, where $x^{-\kappa}$ refers to the principal branch.
Similarly, to define $\zeta^\kappa$ by the principal branch we need to have $0<\arg(\zeta)<\pi$ or for large $\lambda$, $0<\arg(\ii x\lambda(1+O(\lambda^{-1})))<\pi$.  Writing $\arg(\ii x\lambda(1+O(\lambda^{-1})))=\tfrac{1}{2}\pi+\mathrm{Arg}(x)+\OurArg{\lambda}+\mathrm{Arg}(1+O(\lambda^{-1}))+2\pi \ell$ and again using $\OurArg{\lambda}+\mathrm{Arg}(x)\in (-\tfrac{1}{2}\pi,\tfrac{1}{2}\pi)$ gives $\ell=0$ so that $\zeta^\kappa=\ee^{\ii\pi\kappa/2}x^\kappa\OurPower{\lambda}{\kappa}(1+O(\lambda^{-1}))$ as $\lambda\to\infty$ in $\Omega_+^\infty$, where again $x^\kappa$ is the principal branch.  Putting these results together gives
\begin{multline}
\mathbf{\Psi}(\lambda;x)\ee^{-\ii x\lambda\sigma_3/2}\OurPower{\lambda}{\Theta_\infty\sigma_3/2}=
\left(\begin{bmatrix}\ee^{\ii \pi\kappa/2} & 0\\0 & 4K^{-1}\ee^{\ii\pi(\kappa+1)/2}
\end{bmatrix}+O(\lambda^{-1})\right)\\
{}\cdot\OurPower{\lambda}{-\Theta_\infty\sigma_3/2}\ee^{\ii x\lambda\sigma_3/2}\mathbf{C}_+^\infty \ee^{-\ii x\lambda\sigma_3/2}\OurPower{\lambda}{\Theta_\infty\sigma_3/2},\quad\lambda\to\infty,\quad\lambda\in\Omega_+^\infty.
\end{multline}
Since $\Omega_+^\infty$ contains directions in which both exponential factors $\ee^{\pm\ii x\lambda}$ are exponentially large as $\lambda\to\infty$, this can only have a finite limit if $\mathbf{C}_+^\infty$ is a diagonal matrix, in which case the correct normalization requires that
\begin{equation}
\mathbf{C}_+^\infty:=\begin{bmatrix}\ee^{-\ii\pi\kappa/2} & 0\\0 & -\tfrac{\ii}{4}K\ee^{-\ii\pi \kappa/2}\end{bmatrix}.  
\end{equation}
Using this formula for $\mathbf{C}_+^\infty$ in \eqref{eq:PsiInftyPlus} completes the precise definition of $\mathbf{\Psi}(\lambda;x)$ for $\lambda\in\Omega_+^\infty$.
\subsubsection{Defining $\mathbf{\Psi}(\lambda;x)$ for $\lambda\in\Omega_-^\infty$}
In a similar way, we define $\mathbf{\Psi}(\lambda;x)$ for $\lambda\in\Omega_-^\infty$ by the formula
\begin{equation}
\mathbf{\Psi}(\lambda;x)=\ee^{x\sigma_3}x^{\Theta_\infty\sigma_3/2}x^{-1/2}\OurPower{(\lambda+\ii)}{-1/2}\mathbf{W}(x,\lambda)\mathbf{C}_-^\infty,\quad\lambda\in\Omega_-^\infty
\label{eq:Psi-2-infinity}
\end{equation}
and we determine $\mathbf{C}_-^\infty$ so that $\mathbf{\Psi}=\mathbf{\Psi}(\lambda;x)$ satisfies \eqref{eq:infinity-norm} with $\lambda^{\Theta_\infty\sigma_3/2}$ interpreted as $\OurPower{\lambda}{\Theta_\infty\sigma_3/2}$ in the limit $\lambda\to\infty$ with $\lambda\in\Omega_-^\infty$.  Again we may use both \eqref{eq:factor-large} and \eqref{eq:W-large}, and it remains to interpret the principal branch power functions $(\pm\zeta)^{\pm\kappa}$ appearing in \eqref{eq:W-large}.  Now by definition, $\mathrm{Im}(\zeta)<0$ and $\mathrm{Im}(-\zeta)>0$ hold for $\lambda\in\Omega_-^\infty$, so for the principal branch powers we have $-\pi<\arg(\zeta)<0$ and $0<\arg(-\zeta)<\pi$.  Writing $\arg(\zeta)=\arg(\ii x\lambda(1+O(\lambda^{-1})))=\tfrac{1}{2}\pi + \mathrm{Arg}(x)+\OurArg{\lambda}+\mathrm{Arg}(1+O(\lambda^{-1})) + 2\pi \ell$, $\ell\in\mathbb{Z}$, and taking into account that $\OurArg{\lambda}+\mathrm{Arg}(x)\in (\tfrac{1}{2}\pi,\tfrac{3}{2}\pi)$ according to Figure~\ref{fig:jump-contour} and \eqref{eq:arg-lambda-infinity} we find that $\ell=-1$ and so $\zeta^\kappa=\ee^{-3\pi\ii\kappa/2}x^\kappa\OurPower{\lambda}{\kappa}(1+O(\lambda^{-1}))$ as $\lambda\to\infty$ from $\Omega_-^\infty$ where $x^\kappa$ is the principal branch.  Similarly writing $\arg(-\zeta)=\arg(-\ii x\lambda(1+O(\lambda^{-1})))=-\tfrac{1}{2}\pi+\mathrm{Arg}(x)+\OurArg{\lambda}+\mathrm{Arg}(1+O(\lambda^{-1}))+2\pi \ell$ we get that $\ell=0$ and so $(-\zeta)^{-\kappa}=\ee^{\ii\pi\kappa/2}x^{-\kappa}\OurPower{\lambda}{-\kappa}(1+O(\lambda^{-1}))$ as $\lambda\to\infty$ from $\Omega_-^\infty$ where $x^{-\kappa}$ is the principal branch.  Using this information and imposing the normalization condition \eqref{eq:infinity-norm} on the formula \eqref{eq:Psi-2-infinity} we learn that the matrix $\mathbf{C}_-^\infty$ must again be diagonal for the required limit to exist, and then
\begin{equation*}
\mathbf{C}_-^\infty = \begin{bmatrix}\ee^{-\ii\pi\kappa/2} & 0\\0 & -\tfrac{\ii}{4}K\ee^{3\pi\ii\kappa/2}\end{bmatrix}.
\end{equation*}
Combining this with \eqref{eq:Psi-2-infinity} completes the definition of $\mathbf{\Psi}(\lambda;x)$ for $\lambda\in\Omega_-^\infty$.
\subsubsection{Defining $\mathbf{\Psi}(\lambda;x)$ for $\lambda\in\Omega_-^0$}
We write $\mathbf{\Psi}(\lambda;x)$ for $\lambda\in\Omega_-^0$ in the form
\begin{equation}
\mathbf{\Psi}(\lambda;x)=\ee^{x\sigma_3}x^{\Theta_\infty\sigma_3/2}x^{-1/2}\OurPower{(\lambda+\ii)}{-1/2}\mathbf{W}(x,\lambda)\mathbf{C}_-^0,\quad\lambda\in\Omega_-^0,
\label{eq:Psi-Zero-Minus}
\end{equation}
and try to determine the constant matrix $\mathbf{C}_-^0$ such that \eqref{eq:zero-norm} holds (with $\lambda^{-\Theta_0\sigma_3/2}$ carefully interpreted as $\OurPower{\lambda}{-\Theta_0\sigma_3/2}$) for some appropriate $a$ and $b$ in the limit $\lambda\to 0$ from $\Omega_-^0$.  Note that the precisely-defined factor $\OurPower{(\lambda+\ii)}{-1/2}$ is analytic near $\lambda=0$ and satisfies \eqref{eq:factor-small}, while in the limit
%\begin{equation}
%\OurPower{(\lambda+\ii)}{-1/2}=\ee^{-\ii\pi/4}+O(\lambda),\quad\lambda\to 0.
%\label{eq:factor-small}
%\end{equation}
%When 
$\lambda\to 0$, the Whittaker matrix $\mathbf{W}(x,\lambda)$ takes the following asymptotic form:
\begin{equation}
\mathbf{W}(x,\lambda)=\left(\begin{bmatrix} 1 & 1\\K^{-1}(1-2\Theta_\infty) & K^{-1}(1-2\Theta_\infty-4x)\end{bmatrix}+O(\lambda)\right)\begin{bmatrix}\ee^{\zeta/2}(-\zeta)^{-\kappa} & 0\\
0 & \ee^{-\zeta/2}\zeta^\kappa\end{bmatrix},\quad\lambda\to 0.
\label{eq:W-small}
\end{equation}
We carefully interpret the principal branch powers appearing in \eqref{eq:W-small} by noting that $\lambda\in\Omega_-^0$ means by definition that $\mathrm{Im}(\zeta)<0$ so we need to have $-\pi<\arg(\zeta)<0$ and $0<\arg(-\zeta)<\pi$.
Writing $\arg(\zeta)=\arg(-\ii x\lambda^{-1}(1+O(\lambda)))=-\tfrac{1}{2}\pi+\mathrm{Arg}(x)-\OurArg{\lambda}+\mathrm{Arg}(1+O(\lambda))+2\pi \ell$, $\ell\in\mathbb{Z}$, and observing from Figure~\ref{fig:jump-contour} and \eqref{eq:arg-lambda-zero}
that $\lambda$ small and in $\Omega_-^0$ means $\OurArg{\lambda}-\mathrm{Arg}(x)\in (-\tfrac{1}{2}\pi,\tfrac{1}{2}\pi)$, we see that $\ell=0$ and so $\zeta^\kappa=\ee^{-\ii\pi\kappa/2}x^\kappa\OurPower{\lambda}{-\kappa}(1+O(\lambda))$ as $\lambda\to 0$ from $\Omega_-^0$ where $x^\kappa$ is the principal branch.
Similarly, writing $\arg(-\zeta)=\arg(\ii x\lambda^{-1}(1+O(\lambda))=\tfrac{1}{2}\pi+\mathrm{Arg}(x)-\OurArg{\lambda}+\mathrm{Arg}(1+O(\lambda))+2\pi \ell$ and again using $\OurArg{\lambda}-\mathrm{Arg}(x)\in (-\tfrac{1}{2}\pi,\tfrac{1}{2}\pi)$ we find that $\ell=0$ and so $(-\zeta)^{-\kappa}=\ee^{-\ii\pi\kappa/2}x^{-\kappa}\OurPower{\lambda}{\kappa}(1+O(\lambda))$ as $\lambda\to 0$ from $\Omega_-^0$ where $x^{-\kappa}$ denotes the principal branch.  Using this information in \eqref{eq:zero-norm} we see that again $\mathbf{C}_-^0$ must be a diagonal matrix, say
\begin{equation}
\mathbf{C}_-^0=\begin{bmatrix}c & 0\\0 & d\end{bmatrix}
\label{eq:C-Zero-Minus}
\end{equation}
with $c$ and $d$ independent of both $x$ and $\lambda$, and then $\mathbf{\Psi}=\mathbf{\Psi}(\lambda;x)$ indeed satisfies \eqref{eq:zero-norm} provided that 
\begin{equation}
\begin{split}
a(x)&= \ee^{-\ii\pi\kappa/2}\ee^{-\ii\pi/4}c\\
b(x)&=4K^{-1}\ee^{-\ii\pi\kappa/2}\ee^{-\ii\pi/4}x^{-1}d.
\end{split}
\label{eq:ab-cd}
\end{equation}
Note that $a(x)$ is independent of $x$.  The unimodularity condition \eqref{eq:unimodularity} is then equivalent to the following condition on the constants $c$ and $d$:
\begin{equation}
\det(\mathbf{C}_-^0)=cd=-\frac{1}{4}\ii K\ee^{\ii\pi\kappa}.
\label{eq:cd-identity}
\end{equation}
Therefore, to completely define $\mathbf{\Psi}(\lambda;x)$ we should simply choose convenient values for $c$ and $d$ consistent with \eqref{eq:cd-identity} and then combine \eqref{eq:C-Zero-Minus} with \eqref{eq:Psi-Zero-Minus}.
\subsubsection{Defining $\mathbf{\Psi}(\lambda;x)$ for $\lambda\in\Omega_+^0$}
We write $\mathbf{\Psi}(\lambda;x)$ for $\lambda\in\Omega_+^0$ in the form
\begin{equation}
\mathbf{\Psi}(\lambda;x)=\ee^{x\sigma_3}x^{\Theta_\infty\sigma_3/2}x^{-1/2}\OurPower{(\lambda+\ii)}{-1/2}\mathbf{W}(x,\lambda)\mathbf{C}_+^0,\quad\lambda\in\Omega_+^0,
\label{eq:Psi-Zero-Plus}
\end{equation}
for a constant matrix $\mathbf{C}_+^0$ to be determined from the normalization condition \eqref{eq:zero-norm} in which $\lambda^{-\Theta_0\sigma_3/2}$ is interpreted as $\OurPower{\lambda}{-\Theta_0\sigma_3/2}$.  We may again use \eqref{eq:factor-small} and \eqref{eq:W-small} and it remains to interpret the principal branch power functions $\zeta^\kappa$ and $(-\zeta)^{-\kappa}$ for $\lambda\in\Omega_+^0$.  By definition, $\lambda\in\Omega_+^0$ means $\mathrm{Im}(\zeta)>0$, so $0<\arg(\zeta)<\pi$ and $-\pi<\arg(-\zeta)<0$.  Writing $\arg(\zeta)=\arg(-\ii x\lambda^{-1}(1+O(\lambda)))=-\tfrac{1}{2}\pi+\mathrm{Arg}(x)-\OurArg{\lambda}+\mathrm{Arg}(1+O(\lambda)) + 2\pi \ell$, $\ell\in\mathbb{Z}$, and noting from Figure~\ref{fig:jump-contour} and \eqref{eq:arg-lambda-zero} that
$\lambda$ small in $\Omega_+^0$ means that $\OurArg{\lambda}-\mathrm{Arg}(x)\in (\tfrac{1}{2}\pi,\tfrac{3}{2}\pi)$, we obtain $\ell=1$ and therefore 
%$\zeta^\kappa=\ee^{-3\pi\ii\kappa/2}x^\kappa\lambda^{-\kappa}(1+O(\lambda))$ 
%\textcolor{magenta}{(oops, should be 
$\zeta^\kappa=\ee^{3\pi\ii\kappa/2}x^\kappa\OurPower{\lambda}{-\kappa}(1+O(\lambda))$)
%} 
as $\lambda\to 0$ from $\Omega_+^0$ where $x^\kappa$ is the principal branch.  Likewise writing $\arg(-\zeta)=\arg(\ii x\lambda^{-1}(1+O(\lambda)))=\tfrac{1}{2}\pi+\mathrm{Arg}(x)-\OurArg{\lambda}+\mathrm{Arg}(1+O(\lambda))+2\pi \ell$ we see that $\ell=0$ and therefore $(-\zeta)^{-\kappa}=\ee^{-\ii\pi\kappa/2}x^{-\kappa}\OurPower{\lambda}{\kappa}(1+O(\lambda))$ as $\lambda\to 0$ from $\Omega_+^0$ where $x^{-\kappa}$ is the principal branch.  Using this information in \eqref{eq:zero-norm} we see that the matrix $\mathbf{C}_+^0$ must be diagonal:
\begin{equation}
\mathbf{C}_+^0=\begin{bmatrix}g & 0\\0 & h\end{bmatrix}
\label{eq:C-Zero-Plus}
\end{equation}
where the constants $g$ and $h$ are related to $a(x)$ and $b(x)$ by
%\begin{equation}
%\begin{split}
%a&=C_+^{-1}\ee^{3\pi\ii\kappa/2}\ee^{-\ii\pi/4}\ee^{-2x}x^{-\Theta_\infty}g\\
%b&=\ee^{-\ii\pi\kappa/2}\ee^{-\ii\pi/4}\ee^{2x}x^{\Theta_\infty-1}h.
%\end{split}
%\label{eq:ab-gh}
%\end{equation}
%\textcolor{magenta}{Oops, this should be instead
\begin{equation}
\begin{split}
a(x)&=\ee^{-\ii\pi\kappa/2}\ee^{-\ii\pi/4}g\\
b(x)&=4K^{-1}\ee^{3\pi\ii\kappa/2}\ee^{-\ii\pi/4}x^{-1}h.
\end{split}
\label{eq:ab-gh}
\end{equation}
%}
Once again, $a(x)$ is independent of $x$, and the unimodularity condition \eqref{eq:unimodularity} is then equivalent to
\begin{equation}
\det(\mathbf{C}_+^0)=gh=-\frac{1}{4}\ii K\ee^{-\ii\pi\kappa}.
\label{eq:gh-identity}
\end{equation}
Choosing any constants $g$ and $h$ consistent with \eqref{eq:gh-identity} therefore determines $\mathbf{\Psi}(\lambda;x)$ for $\lambda\in\Omega_+^0$ by combining \eqref{eq:C-Zero-Plus} with \eqref{eq:Psi-Zero-Plus}. 

\subsection{Jump matrices for $\mathrm{Im}(x)\neq 0$}
Before computing the jump matrices, we will remove the ambiguity of the constants $c,d,g,h$ still present in the definition of $\mathbf{\Psi}(\lambda;x)$ for $\lambda\in\Omega_\pm^0$ in the following way:
\begin{itemize}
\item If $\mathrm{Im}(x)>0$, we choose $c$ and $d$ so that $\mathbf{C}_-^0 = \mathbf{C}_-^\infty$.  This is allowed because the diagonal elements of $\mathbf{C}_-^\infty$ obviously also satisfy \eqref{eq:cd-identity} because $2\kappa+1=\Theta_\infty$.  Similarly, if $\mathrm{Im}(x)<0$, we choose $g$ and $h$ such that $\mathbf{C}_+^0=\mathbf{C}_+^\infty$, which is consistent because the diagonal elements of $\mathbf{C}_+^\infty$ satisfy \eqref{eq:gh-identity}.
\item We then insist that the normalization factors $a(x)$ and $b(x)$ appearing in \eqref{eq:zero-norm} are exactly the same regardless of whether $\lambda\to 0$ from $\Omega_-^0$ or from $\Omega_+^0$.
\end{itemize}
The first choice implies that at every point $\lambda\neq -\ii$ of the unit circle forming the common boundary of $\Omega_-^\infty$ and $\Omega_-^0$ (for $\mathrm{Im}(x)>0$) or the common boundary of $\Omega_+^\infty$ and $\Omega_+^0$ (for $\mathrm{Im}(x)<0$), the boundary values taken by $\mathbf{\Psi}(\lambda;x)$ agree, i.e., \emph{the jump matrix for $\mathbf{\Psi}(\lambda;x)$ across the unit circle $S^1\setminus\{-\ii\}$ is exactly the identity matrix}.  The second choice together with the first implies, in light of \eqref{eq:ab-cd} and \eqref{eq:ab-gh}, that the matrices $\mathbf{C}_\pm^0$ are necessarily given by
%\begin{equation}
%\mathbf{C}_2^0 = \begin{bmatrix}\ee^{-5\pi \ii\kappa/2}&0\\0 & \displaystyle
%-\ii \ee^{3\pi\ii\kappa/2}\frac{1}{4}C_+(1-2\Theta_\infty)\end{bmatrix}.
%\end{equation}
%\textcolor{magenta}{Oops, this should be instead:
\begin{equation*}
\mathbf{C}_-^0 = \begin{bmatrix}\ee^{-\ii\pi\kappa/2} & 0\\0 & -\tfrac{1}{4}\ii K\ee^{3\pi\ii\kappa/2}
\end{bmatrix}\quad\text{and}\quad
\mathbf{C}_+^0 = \begin{bmatrix}\ee^{-\ii\pi \kappa/2}&0\\0 & \displaystyle
-\tfrac{1}{4}\ii K\ee^{-\ii\pi\kappa/2}\end{bmatrix}.
\end{equation*}
%}
Note that these formul\ae\ do not depend on the sign of $\mathrm{Im}(x)$.
Thus, the matrix function $\mathbf{\Psi}(\lambda;x)$ has been determined modulo only the value of the constant $K\neq 0$, as an analytic function of $\lambda\in\mathbb{C}\setminus L$ where $L=\LInftyRed\cup\LZeroRed\cup\LInftyBlue\cup\LZeroBlue$ is the jump contour for the Whittaker matrix $\mathbf{W}$ illustrated with red and cyan curves in Figure~\ref{fig:jump-contour}.\bigskip
% $\Omega_1^\infty$, $\Omega_2^\infty\cup (S^1\setminus\{-\ii\})\cup\Omega_1^0$, and $\Omega_2^0$.

%The only nontrivial jumps for $\mathbf{\Psi}(\lambda;x)$ occur across the four contours $\LInftyBlue$, $\LZeroBlue$, $\LInftyRed$, and $\LZeroRed$.  These 
The jump conditions satisfied by $\mathbf{\Psi}(\lambda;x)$ across the four arcs of $L$ oriented as shown in Figure~\ref{fig:jump-contour} 
are computed by comparing the formul\ae\ for $\mathbf{\Psi}(\lambda;x)$ on either side using the identities \eqref{eq:W-jump}--\eqref{eq:W-prime-jump} together with the fact that $\zeta<0$ along $\LZeroRed$ and $\LInftyRed$ while $\zeta>0$ along $\LZeroBlue$ and $\LInftyBlue$.  One also has to take into account that the factor $\OurPower{(\lambda+\ii)}{-1/2}$ changes sign across $\LInftyBlue$ by definition, but otherwise is analytic.  The jump conditions are as follows:
\begin{itemize}
\item The arc $\LInftyRed$ separates the domain $\Omega_+^\infty$ on its left from $\Omega_-^\infty$ on its right.  Using 
%, we have $\mathbf{\Psi}_+=\mathbf{\Psi}_1^\infty$ and $\mathbf{\Psi}_-=\mathbf{\Psi}_2^\infty$ given the orientation in Figure~\ref{fig:jump-contour}, and using 
$\zeta<0$ for $\lambda\in\LInftyRed$ we deduce that
\begin{equation}
\mathbf{\Psi}_+(\lambda;x)=\mathbf{\Psi}_-(\lambda;x)\VInftyRed,\quad \lambda\in\LInftyRed
\label{eq:jump-infty-red}
\end{equation}
where
\begin{equation}
\VInftyRed:=\begin{bmatrix}1 & \displaystyle\tfrac{1}{4}K\ee^{\ii\pi\kappa}\cdot\frac{2\pi }{\Gamma(\tfrac{1}{2}+\mu-\kappa)\Gamma(\tfrac{1}{2}-\mu-\kappa)}\\
0 & 1\end{bmatrix}.
\label{eq:V-infty-red}
\end{equation}
\item The arc $\LZeroRed$ separates the domain $\Omega_-^0$ on its left from $\Omega_+^0$ on its right.  Using
%, we have $\mathbf{\Psi}_+=\mathbf{\Psi}_1^0$ and $\mathbf{\Psi}_-=\mathbf{\Psi}_2^0$ and using 
$\zeta<0$ we get
\begin{equation}
\mathbf{\Psi}_+(\lambda;x)=\mathbf{\Psi}_-(\lambda;x)\VZeroRed,\quad\lambda\in\LZeroRed
\label{eq:jump-0-red}
\end{equation}
where
%\begin{equation}
%{\color{red}\mathbf{V}^0_\mathrm{red}}:=\begin{bmatrix}\ee^{2\pi \ii\kappa} & \displaystyle
%\frac{1}{4}C_+(1-2\Theta_\infty)\cdot\frac{-2\pi \ee^{3\pi\ii\kappa}}{\Gamma(\tfrac{1}{2}+\mu-\kappa)\Gamma(\tfrac{1}{2}-\mu-\kappa)}\\0 & \ee^{-2\pi\ii\kappa}\end{bmatrix}=\begin{bmatrix}
%-\ee^{\ii\pi\Theta_\infty} & \displaystyle\frac{-2\pi \ii C_+ \ee^{3\pi \ii\Theta_\infty/2}}{\Gamma(\tfrac{1}{4}-\tfrac{1}{2}\Theta_\infty)\Gamma(\tfrac{3}{4}-\tfrac{1}{2}\Theta_\infty)}\\0 & -\ee^{-\ii\pi\Theta_\infty}\end{bmatrix}.
%\end{equation}
%\textcolor{magenta}{Oops, this should be instead:
\begin{equation}
\VZeroRed:=\begin{bmatrix}1 & \displaystyle
-\tfrac{1}{4}K\ee^{\ii\pi\kappa}\cdot\frac{2\pi}{\Gamma(\tfrac{1}{2}+\mu-\kappa)\Gamma(\tfrac{1}{2}-\mu-\kappa)}\\0 & 1\end{bmatrix}.
\label{eq:V-0-red}
\end{equation}
%}
\item The arc $\LZeroBlue$ separates the domain $\Omega_+^0$ on its left from $\Omega_-^0$ on its right.
Using 
%  , we have $\mathbf{\Psi}_+=\mathbf{\Psi}_2^0$ and $\mathbf{\Psi}_-=\mathbf{\Psi}_1^0$ and using 
$\zeta>0$ we arrive at
\begin{equation}
\mathbf{\Psi}_+(\lambda;x)=\mathbf{\Psi}_-(\lambda;x)\VZeroBlue,\quad\lambda\in\LZeroBlue
\label{eq:jump-0-blue}
\end{equation}
where
%\begin{equation}
%{\color{blue}\mathbf{V}^0_\mathrm{blue}}:=\begin{bmatrix}1 & 0\\\displaystyle\frac{4}{C_+(1-2\Theta_\infty)}\cdot\frac{2\pi \ee^{-3\pi\ii\kappa}}{\Gamma(\tfrac{1}{2}+\mu+\kappa)\Gamma(\tfrac{1}{2}-\mu+\kappa)} & 1\end{bmatrix}=\begin{bmatrix}1 & 0\\\displaystyle\frac{2\pi\ii C_+^{-1}\ee^{-3\pi\ii\Theta_\infty/2}}{\Gamma(\tfrac{1}{4}+\tfrac{1}{2}\Theta_\infty)\Gamma(\tfrac{3}{4}+\tfrac{1}{2}\Theta_\infty)} & 1\end{bmatrix}.
%\end{equation}
%\textcolor{magenta}{Oops, this should be instead:
\begin{equation}
\VZeroBlue:=\begin{bmatrix}\ee^{2\pi \ii\kappa} & 0\\\displaystyle (\tfrac{1}{4}K\ee^{\ii\pi\kappa})^{-1}\cdot\frac{2\pi}{\Gamma(\tfrac{1}{2}+\mu+\kappa)\Gamma(\tfrac{1}{2}-\mu+\kappa)} & \ee^{-2\pi\ii\kappa}\end{bmatrix}.
\label{eq:V-0-blue}
\end{equation}
%}
\item Finally, the arc $\LInftyBlue$ separates the domain $\Omega_-^\infty$ on its left from $\Omega_+^\infty$ on its right.  Using $\zeta>0$ and taking into account that $\OurPower{(\lambda+\ii)}{-1/2}$ changes sign across $\LInftyBlue$ we obtain
%.  Since $\mathbf{\Psi}_+=\mathbf{\Psi}_2^\infty$ while $\mathbf{\Psi}_-=\mathbf{\Psi}_1^\infty$, we use $\zeta>0$ to deduce
\begin{equation}
\mathbf{\Psi}_+(\lambda;x)=\mathbf{\Psi}_-(\lambda;x)\VInftyBlue,\quad\lambda\in\LInftyBlue,
\label{eq:jump-infty-blue}
\end{equation}
where
\begin{equation}
\VInftyBlue:=\begin{bmatrix}-\ee^{-2\pi\ii\kappa} & 0\\
\displaystyle (\tfrac{1}{4}K\ee^{\ii\pi\kappa})^{-1}\cdot\frac{2\pi}{\Gamma(\tfrac{1}{2}+\mu+\kappa)\Gamma(\tfrac{1}{2}-\mu+\kappa)}& -\ee^{2\pi \ii\kappa}\end{bmatrix}.
\label{eq:V-infty-blue}
\end{equation}
\end{itemize}
These formul\ae\ may be simplified further by recalling the definitions $\mu=\tfrac{1}{4}$ and $\kappa=\tfrac{1}{2}(\Theta_\infty-1)$ (so $\Theta_\infty=m+1$ for $n=0$ implies $\kappa=\tfrac{1}{2}m$), using the duplication formula \cite[Eq.~5.5.5]{DLMF} 
$\Gamma(2z)=\pi^{-1/2}2^{2z-1}\Gamma(z)\Gamma(z+\tfrac{1}{2})$, and choosing
\begin{equation*}
K=2^{m+2}\ee^{-\ii\pi m/2}.
\end{equation*}
Thus we find
\begin{equation}
\VInftyRed=\VInftyRed(m):=\begin{bmatrix}1 & \displaystyle\frac{\sqrt{2\pi}}{\Gamma(\tfrac{1}{2}-m)}\\0 & 1\end{bmatrix},\quad
\VZeroRed=\VZeroRed(m):=\begin{bmatrix}1 & \displaystyle-\frac{\sqrt{2\pi}}{\Gamma(\tfrac{1}{2}-m)}\\0 & 1\end{bmatrix},
\label{eq:Red-Jumps}
\end{equation}
\begin{equation}
\VZeroBlue=\VZeroBlue(m):=\begin{bmatrix}\ee^{\ii\pi m} & 0\\
\displaystyle \frac{\sqrt{2\pi}}{\Gamma(\tfrac{1}{2}+m)} & \ee^{-\ii\pi m}\end{bmatrix},\quad\VInftyBlue=\VInftyBlue(m):=\begin{bmatrix}-\ee^{-\ii\pi m} & 0\\
\displaystyle \frac{\sqrt{2\pi}}{\Gamma(\tfrac{1}{2}+m)} & -\ee^{\ii\pi m}\end{bmatrix}.
\label{eq:Blue-Jumps}
\end{equation}

In the general theory \cite{FokasIKN06} of the direct monodromy problem for \eqref{eq:PIII}, the Stokes constants are subject to an identity known as the \emph{cyclic relation}.  In this setting, the cyclic relation is simply equivalent to the statement that for consistency, the ordered product of the jump matrices around the self-intersection point $\lambda=-\ii$ must be the identity:
\begin{equation}
\VInftyBlue(m)^{-1}\VInftyRed(m)^{-1}
\VZeroBlue(m)\VZeroRed(m)=\mathbb{I}.
\label{eq:cyclic}
\end{equation}
While it is straightforward to check directly that \eqref{eq:cyclic} holds, this identity is in fact a simple consequence of the way the jump matrices were computed, namely by 
comparing four functions, each of which admits analytic continuation to a full neighborhood of the self-intersection point $\lambda=-\ii$ and that differ only by right-multiplication by constant matrices.  In other words, \eqref{eq:cyclic} holds as a (\v{C}ech-)cohomological identity.

\subsection{The limiting cases of $x>0$ and $x<0$}
\label{sec:real-axis-limit}
The jump contour $L$ for the Whittaker matrix $\mathbf{W}(x,\lambda)$ undergoes a bifurcation when $x$ crosses either the positive or negative real axes.  The bifurcation that occurs as $\mathrm{Arg}(x)$ passes through zero is illustrated in Figure~\ref{fig:NearPositive-x}.
\begin{figure}[h]
\begin{center}
\includegraphics{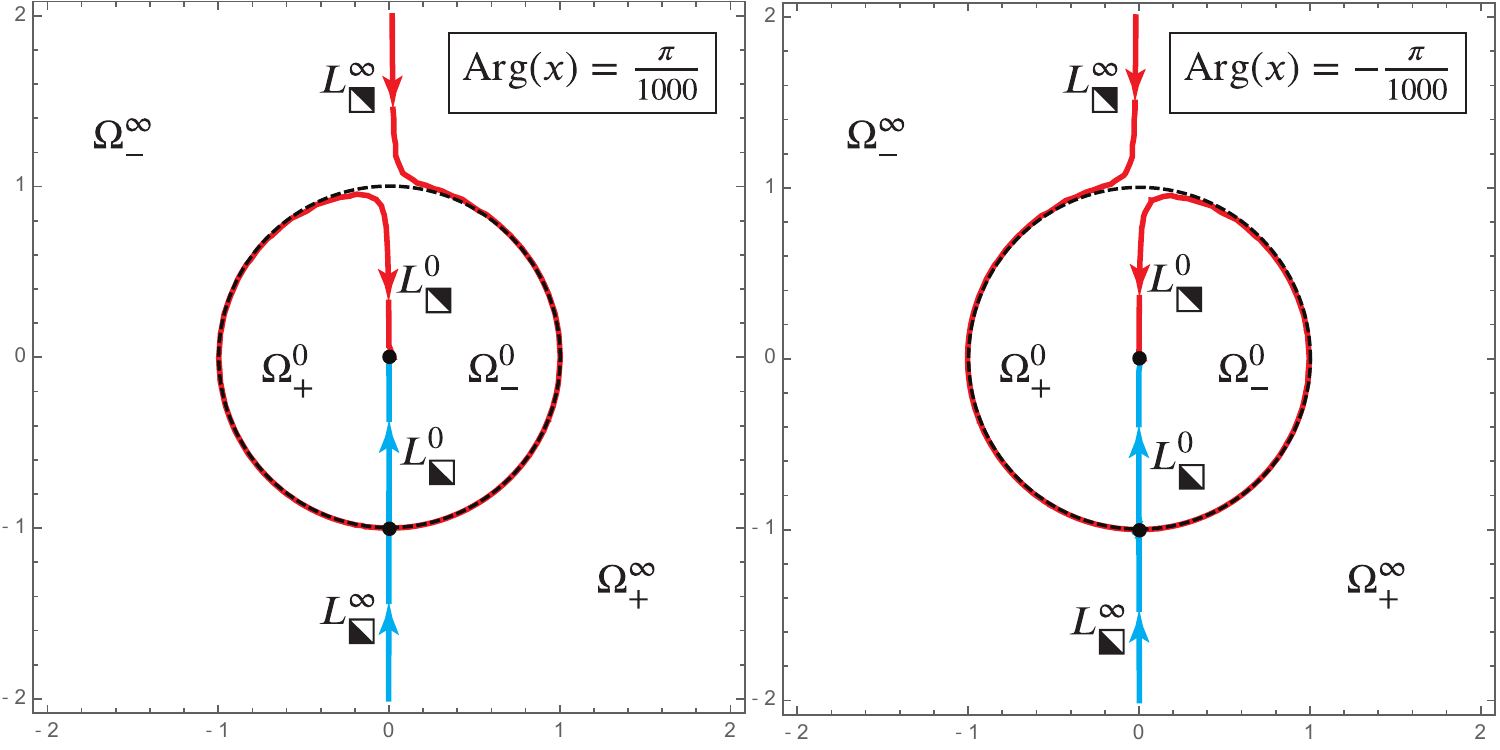}
\end{center}
\caption{As in Figure~\ref{fig:jump-contour} except for values of $x$ close to the positive real axis.}
\label{fig:NearPositive-x}
\end{figure}
Clearly, the arcs $\LZeroBlue$ and $\LInftyBlue$ depend continuously on $\mathrm{Arg}(x)$ near $\mathrm{Arg}(x)=0$, but the parts of $\LZeroRed$ and $\LInftyRed$ close to the unit circle become interchanged as $\mathrm{Arg}(x)$ passes through zero.  However, noting that the matrices $\VInftyRed(m)$ and $\VZeroRed(m)$ as defined in \eqref{eq:Red-Jumps} are inverse to each other, we easily conclude that the jump conditions satisfied by the matrix $\mathbf{\Psi}(\lambda;x)$ actually depend continuously on $\mathrm{Arg}(x)$ near $\mathrm{Arg}(x)=0$.  This makes it possible to define the jump conditions by continuity for $\mathrm{Arg}(x)=0$.  Note also that not only are the branch cuts of the functions $\OurPower{\lambda}{p}$ and $\OurPower{(\lambda+\ii)}{-1/2}$ continuous with respect to $\mathrm{Arg}(x)$ near $\mathrm{Arg}(x)=0$, but so also are the functions themselves.  

On the other hand, as $x$ approaches the negative real axis from above and below, the bifurcation as illustrated in Figure~\ref{fig:NearNegative-x} 
\begin{figure}[h]
\begin{center}
\includegraphics{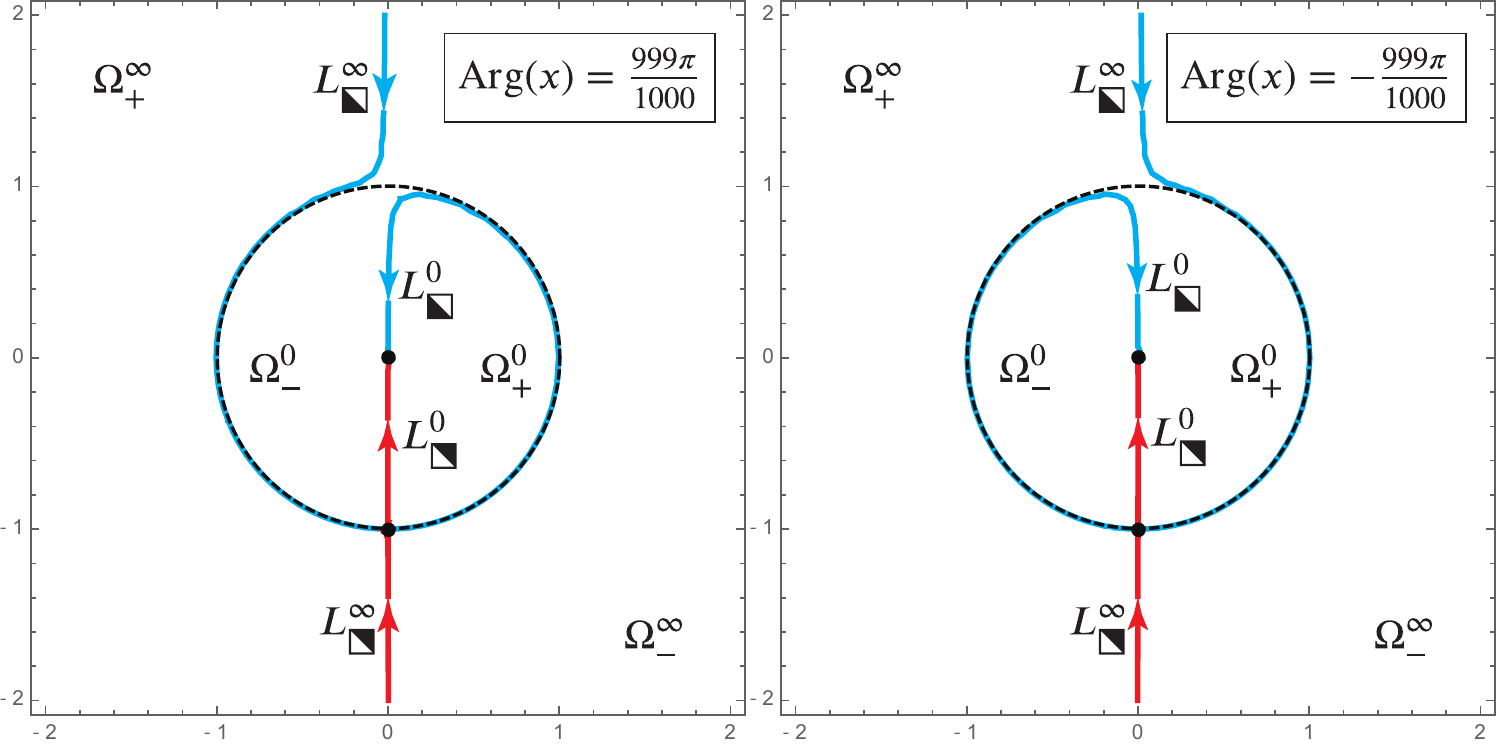}
\end{center}
\caption{As in Figure~\ref{fig:jump-contour} except for values of $x$ close to the negative real axis.}
\label{fig:NearNegative-x}
\end{figure}
is apparently more serious.  Indeed, the arcs of $\LInftyBlue$ and $\LZeroBlue$ near the unit circle are now interchanged while $\LInftyRed$ and $\LZeroRed$ depend continuously on $\mathrm{Arg}(-x)$.  Since, according to \eqref{eq:Blue-Jumps}, $\VZeroBlue(m)\VInftyBlue(m)=-\mathbb{I}$, it is not hard to see that in the limit $\mathrm{Arg}(-x)\to 0$ the limiting jump conditions from $\mathrm{Im}(x)>0$ and $\mathrm{Im}(x)<0$ differ precisely on the unit circle, by a sign.  In terms of the matrix $\mathbf{\Psi}(\lambda;x)$ itself, 
\begin{equation*}
\lim_{\epsilon\downarrow 0}\mathbf{\Psi}(\lambda;x+\ii\epsilon)=\mathrm{sgn}(\ln|\lambda|)\lim_{\epsilon\downarrow 0}\mathbf{\Psi}(\lambda;x-\ii\epsilon),\quad x<0.
\end{equation*}
Naturally, both limiting values correspond to simultaneous solutions of the Painlev\'e-III Lax pair \eqref{eq:lambda-eqn}--\eqref{eq:x-eqn} for exactly the same solution $u(x)=1$; the apparent monodromy in the function $\mathbf{\Psi}(\lambda;x)$ about $x=0$ can be absorbed into a sign change in the arbitrary constants $a$ and $b$ appearing in \eqref{eq:zero-norm}.  For practical calculations one has to be careful about the values of the power functions $\OurPower{\lambda}{p}$ for $|\lambda|<1$ in taking the limit of $\mathbf{\Psi}(\lambda;x)$ as $x$ approaches a negative real value from the upper/lower half-planes.  Indeed, keeping track of the dependence of $\OurArg{\lambda}$ on $x$ with the augmented notation $\OurArg{\lambda;x}$, we have the identity
\begin{equation*}
\lim_{\epsilon\downarrow 0}\OurArg{\lambda;x+\ii\epsilon}=\lim_{\epsilon\downarrow 0}\OurArg{\lambda;x-\ii\epsilon} - 2\pi\mathrm{sgn}(\ln|\lambda|),\quad x<0.
\end{equation*}

\section{Schlesinger-B\"acklund Transformations}
\label{sec:Schlesinger}
\subsection{Schlesinger transformations to increment/decrement $n$}
Now suppose that $\VInftyBlue$, $\VInftyRed$, $\VZeroBlue$, and $\VZeroRed$ are any unimodular $2\times 2$ matrices satisfying the cyclic relation \eqref{eq:cyclic}, and that $\mathbf{\Psi}(\lambda;x)$ is an analytic function of $\lambda$ in the domain $\mathbb{C}\setminus L$, $L:=\LInftyBlue\cup\LInftyRed\cup\LZeroBlue\cup\LZeroRed$, satisfying jump conditions of the form \eqref{eq:jump-infty-red}, \eqref{eq:jump-0-red}, \eqref{eq:jump-0-blue}, \eqref{eq:jump-infty-blue}, as well as asymptotic conditions of the form
\begin{equation}
\mathbf{\Psi}(\lambda;x)\OurPower{\lambda}{\Theta_\infty\sigma_3/2}\ee^{-\ii x\lambda\sigma_3/2}=\mathbb{I} + \mathbf{\Psi}^\infty_1(x)\lambda^{-1} + \cdots,\quad\lambda\to\infty
\label{eq:psi-general-infty}
\end{equation}
and
\begin{equation}
\mathbf{\Psi}(\lambda;x)\OurPower{\lambda}{-\Theta_0\sigma_3/2}\ee^{\ii x\lambda^{-1}\sigma_3/2}=\mathbf{\Psi}^0_0(x) + \mathbf{\Psi}^0_1(x)\lambda + \cdots,\quad\lambda\to 0.
\label{eq:psi-general-0}
\end{equation}
Here, $\mathbf{\Psi}^\infty_k(x)$, $k\ge 1$ and $\mathbf{\Psi}^0_k(x)$, $k\ge 0$, are certain matrix coefficients.  Since it necessarily holds that $\det(\mathbf{\Psi}(\lambda;x))=1$, it follows that $\det(\mathbf{\Psi}_0^0(x))=1$ and $\mathrm{tr}(\mathbf{\Psi}_1^\infty(x))=0$.  We define the Pauli-type matrices
\begin{equation*}
\widehat{\sigma}:=\begin{bmatrix}1 & 0\\0 & 0\end{bmatrix}\quad\text{and}\quad
\widecheck{\sigma}:=\begin{bmatrix}0 & 0 \\0 & 1\end{bmatrix},
\end{equation*}
and supposing further that the matrix element $\Psi^0_{0,11}(x)$ is not identically zero, we consider the Schlesinger transformation (also known as a Darboux transformation) given by 
\begin{equation}
\widehat{\mathbf{\Psi}}(\lambda;x):=(\widehat{\sigma}\OurPower{\lambda}{1/2}+\widehat{\mathbf{B}}(x)\OurPower{\lambda}{-1/2})\mathbf{\Psi}(\lambda;x),
\label{eq:Schlesinger}
\end{equation}
where
\begin{equation}
\widehat{\mathbf{B}}(x):=\begin{bmatrix}\Psi^0_{0,21}(x)\Psi^\infty_{1,12}(x)/\Psi_{0,11}^0(x) & -\Psi^\infty_{1,12}(x)\\
-\Psi^0_{0,21}(x)/\Psi^0_{0,11}(x) & 1\end{bmatrix}.
\label{eq:B-up}
\end{equation}
Note that $\det(\widehat{\mathbf{\Psi}}(\lambda;x))=\det(\mathbf{\Psi}(\lambda;x))$ by direct calculation.
Since $\OurPower{\lambda}{\pm 1/2}$ are analytic except on $\LZeroBlue\cup\LInftyBlue$ across which these factors change sign, $\widehat{\mathbf{\Psi}}(\lambda;x)$ is also analytic for $\lambda\in\mathbb{C}\setminus L$,
and it is a direct matter to check the following jump conditions:
\begin{equation}
\widehat{\mathbf{\Psi}}_+(\lambda;x)=\widehat{\mathbf{\Psi}}_-(\lambda;x)\begin{cases}
\VZeroRed,&\quad \lambda\in\LZeroRed,\\
\VInftyRed,&\quad\lambda\in\LInftyRed,\\
-\VZeroBlue,&\quad\lambda\in\LZeroBlue,\\
-\VInftyBlue,&\quad\lambda\in\LInftyBlue.
\end{cases}
\label{eq:transformed-jumps}
\end{equation}
Next, combining \eqref{eq:psi-general-infty} and \eqref{eq:Schlesinger}, observe that in the limit $\lambda\to\infty$ we have
\begin{equation*}
\begin{split}
\widehat{\mathbf{\Psi}}(\lambda;x)\OurPower{\lambda}{(\Theta_\infty-1)\sigma_3/2}\ee^{-\ii x\lambda\sigma_3/2}&=
(\widehat{\sigma}\OurPower{\lambda}{1/2}+\widehat{\mathbf{B}}(x)\OurPower{\lambda}{-1/2})(\mathbb{I}+\mathbf{\Psi}_1^\infty(x)\lambda^{-1}+\cdots)\OurPower{\lambda}{-\sigma_3/2}\\
&=\lambda(\widehat{\sigma}+\widehat{\mathbf{B}}(x)\lambda^{-1})(\mathbb{I}+\mathbf{\Psi}_1^\infty(x)\lambda^{-1}+\cdots)(\widecheck{\sigma}+\widehat{\sigma}\lambda^{-1})\\
&=\widehat{\sigma}\widecheck{\sigma}\lambda  + [\widehat{\sigma}^2+\widehat{\sigma}\mathbf{\Psi}_1^\infty(x)\widecheck{\sigma}+\widehat{\mathbf{B}}(x)\widecheck{\sigma}]+\widehat{\mathbf{\Psi}}_1^\infty(x)\lambda^{-1}+\cdots\\
&=\mathbb{I}+\widehat{\mathbf{\Psi}}_1^\infty(x)\lambda^{-1} + \cdots,
\end{split}
\end{equation*}
where
\begin{equation}
\widehat{\mathbf{\Psi}}_1^\infty:=\widehat{\sigma}\mathbf{\Psi}_1^\infty(x)\widehat{\sigma} +\widehat{\sigma}\mathbf{\Psi}_2^\infty(x)\widecheck{\sigma}+\widehat{\mathbf{B}}(x)\widehat{\sigma}+\widehat{\mathbf{B}}(x)\mathbf{\Psi}_1^\infty(x)\widecheck{\sigma}.
\label{eq:Modified-infty-leading-matrix}
\end{equation}
Similarly, combining \eqref{eq:psi-general-0} with \eqref{eq:Schlesinger} shows that in the limit $\lambda\to 0$ we have
\begin{equation*}
\begin{split}
\widehat{\mathbf{\Psi}}(\lambda;x)\OurPower{\lambda}{-(\Theta_0+1)\sigma_3/2}\ee^{\ii x\lambda^{-1}\sigma_3/2}&=
(\widehat{\sigma}\OurPower{\lambda}{1/2}+\widehat{\mathbf{B}}(x)\OurPower{\lambda}{-1/2})(\mathbf{\Psi}_0^0(x)+\mathbf{\Psi}_1^0(x)\lambda + \cdots)\OurPower{\lambda}{-\sigma_3/2}\\
&=\lambda^{-1}(\widehat{\mathbf{B}}(x)+\widehat{\sigma}\lambda)(\mathbf{\Psi}_0^0(x)+\mathbf{\Psi}_1^0(x)\lambda+\cdots)(\widehat{\sigma}+\widecheck{\sigma}\lambda)\\
&=\widehat{\mathbf{B}}(x)\mathbf{\Psi}_0^0(x)\widehat{\sigma}\lambda^{-1} + \widehat{\mathbf{\Psi}}_0^0(x) + \widehat{\mathbf{\Psi}}_1^0(x)\lambda + \cdots\\
&=\widehat{\mathbf{\Psi}}_0^0(x) + \widehat{\mathbf{\Psi}}_1^0(x)\lambda + \cdots,
\end{split}
\end{equation*}
where
\begin{equation}
\widehat{\mathbf{\Psi}}_0^0(x):= \widehat{\mathbf{B}}(x)\mathbf{\Psi}_0^0(x)\widecheck{\sigma} + \widehat{\mathbf{B}}(x)\mathbf{\Psi}_1^0(x)\widehat{\sigma} + \widehat{\sigma}\mathbf{\Psi}_0^0(x)\widehat{\sigma}.
\label{eq:Modified-0-leading-matrix}
\end{equation}
Thus, the Schlesinger transformation \eqref{eq:Schlesinger} results in a simple modification of the jump conditions and preserves the form of the asymptotic conditions \eqref{eq:psi-general-infty}--\eqref{eq:psi-general-0}, but with the replacements $\Theta_\infty\mapsto\widehat{\Theta}_\infty:=\Theta_\infty-1$ and
$\Theta_0\mapsto\widehat{\Theta}_0:=\Theta_0+1$.  Comparing with \eqref{eq:Thetas-m-n}, we see that these replacements have the effect of incrementing the value of $n$ by $1$ and holding $m$ fixed.
%It is a consequence of Lemma~\ref{l:inverse} formulated below that 
%$\Psi_{0,11}^0(x)$ not identically zero implies that $\widehat{\Psi}_{0,11}^0(x)$ is also not identically zero.  Therefore, the transformation can be iterated.  \textcolor{red}{(This is probably not exactly the right statement, but some version of it holds for the rational chain.)}
%Because formula \eqref{eq:Modified-0-leading-matrix} implies that
%\begin{equation}
%\tilde{\Psi}_{0,11}^0(x)=\Psi_{0,11}^0(x),
%\end{equation}
%and since the latter is not identically zero by assumption, it becomes clear that the Schlesinger transformation can be iterated.
Similarly, assuming that $\Psi^0_{0,22}(x)$ is not identically zero and setting
\begin{equation}
\widecheck{\mathbf{\Psi}}(\lambda;x):=(\widecheck{\sigma}\OurPower{\lambda}{1/2}+\widecheck{\mathbf{B}}(x)\OurPower{\lambda}{-1/2})\mathbf{\Psi}(\lambda;x),
\label{eq:Schlesinger-decrease-n}
\end{equation}
where 
\begin{equation}
\widecheck{\mathbf{B}}(x):=\begin{bmatrix}
1 & -\Psi^0_{0,12}(x)/\Psi^0_{0,22}(x)\\
-\Psi^\infty_{1,21}(x) & \Psi^0_{0,12}(x)\Psi^\infty_{1,21}(x)/\Psi^0_{0,22}(x)\end{bmatrix}
\label{eq:B-down}
\end{equation}
respectively, one finds that again $\det(\widecheck{\mathbf{\Psi}}(\lambda;x))=\det(\mathbf{\Psi}(\lambda;x))$ and \eqref{eq:transformed-jumps} holds with $\widecheck{\mathbf{\Psi}}$ replacing $\widehat{\mathbf{\Psi}}$, but now as $\lambda\to\infty$,
\begin{equation*}
\widecheck{\mathbf{\Psi}}(\lambda;x)\OurPower{\lambda}{(\Theta_\infty+1)\sigma_3/2}\ee^{-\ii x\lambda\sigma_3/2}=
\mathbb{I} +\widecheck{\mathbf{\Psi}}_1^\infty(x)\lambda^{-1}+\cdots,
\end{equation*}
where 
\begin{equation*}
\widecheck{\mathbf{\Psi}}_1^\infty(x):=\widecheck{\sigma}\mathbf{\Psi}_1^\infty(x)\widecheck{\sigma} +\widecheck{\sigma}\mathbf{\Psi}_2^\infty(x)\widehat{\sigma} +\widecheck{\mathbf{B}}(x)\widecheck{\sigma} +\widecheck{\mathbf{B}}(x)\mathbf{\Psi}_1^\infty(x)\widehat{\sigma},
\end{equation*}
and similarly, as $\lambda\to 0$,
\begin{equation*}
\widecheck{\mathbf{\Psi}}(\lambda;x)\OurPower{\lambda}{-(\Theta_0-1)\sigma_3/2}\ee^{\ii x\lambda^{-1}\sigma_3/2}=
\widecheck{\mathbf{\Psi}}_0^0(x)+\widecheck{\mathbf{\Psi}}_1^0(x)\lambda+\cdots
\end{equation*}
where 
\begin{equation*}
\widecheck{\mathbf{\Psi}}_0^0(x):=\widecheck{\mathbf{B}}(x)\mathbf{\Psi}_0^0(x)\widehat{\sigma} +\widecheck{\mathbf{B}}(x)\mathbf{\Psi}_1^0(x)\widecheck{\sigma} +\widecheck{\sigma}\mathbf{\Psi}_0^0(x)\widecheck{\sigma}.
\end{equation*}
Therefore, the Schlesinger transformation \eqref{eq:Schlesinger-decrease-n} also results in a simple modification of the jump conditions and preserves the form of the asymptotic conditions \eqref{eq:psi-general-infty}--\eqref{eq:psi-general-0}, but now with the replacements $\Theta_\infty\mapsto\widecheck{\Theta}_\infty:=\Theta_\infty+1$ and
$\Theta_0\mapsto\widecheck{\Theta}_0:=\Theta_0-1$, replacements having the effect of decrementing the value of $n$ by $1$ and holding $m$ fixed. We now show that the transformations \eqref{eq:Schlesinger} and \eqref{eq:Schlesinger-decrease-n} are in fact inverse to each other:
\begin{lem}
$\widecheck{\widehat{\mathbf{\Psi}}}(\lambda;x)=\widehat{\widecheck{\mathbf{\Psi}}}(\lambda;x)=\mathbf{\Psi}(\lambda;x).$
\label{lemma:inverse}
\end{lem}
\begin{proof}
Fix $x\in\mathbb{C}$ such that $\mathbf{\Psi}(\lambda;x)$ exists satisfying the appropriate analyticity, jump, and normalization conditions; hence in particular the diagonal elements of $\mathbf{\Psi}^0_0(x)$ are finite.
If $\Psi^0_{0,11}(x)\neq 0$ so that $\widehat{\mathbf{\Psi}}(\lambda;x)$ exists, then according to \eqref{eq:Modified-0-leading-matrix} with \eqref{eq:B-up}, the fact that $\det(\mathbf{\Psi}_0^0(x))=1$ implies
that $\widehat{\Psi}_{0,22}^0(x)=1/\Psi_{0,11}^0(x)\neq 0$. Therefore, \eqref{eq:Schlesinger-decrease-n} can be applied to $\widehat{\mathbf{\Psi}}(\lambda;x)$ with the elements of $\widecheck{\mathbf{B}}(x)$ obtained from $\widehat{\mathbf{\Psi}}_1^\infty(x)$ and $\widehat{\mathbf{\Psi}}_0^0(x)$ rather than $\mathbf{\Psi}_1^\infty(x)$ and $\mathbf{\Psi}_0^0(x)$.  Both rows of the latter matrix are proportional to $[1, -\widehat{\Psi}^0_{0,12}(x)/\widehat{\Psi}^0_{0,22}(x)]$, while both columns of $\widehat{\mathbf{B}}(x)$ are proportional to $[-\Psi^\infty_{1,12}(x),1]^\top$, with the inner product being
\begin{equation*}
-\Psi^\infty_{1,12}(x)-\frac{\widehat{\Psi}^0_{0,12}(x)}{\widehat{\Psi}^0_{0,22}(x)}=
-\Psi^\infty_{1,12}(x)-\widehat{\Psi}^0_{0,12}(x)\Psi^0_{0,11}(x)=0,
\end{equation*}
again using \eqref{eq:Modified-0-leading-matrix} with \eqref{eq:B-up}.
Therefore, since $\widecheck{\sigma}\widehat{\sigma}=\mathbf{0}$,
\begin{equation*}
\widecheck{\widehat{\mathbf{\Psi}}}(\lambda;x)=\begin{bmatrix}
1 & 0\\
-\Psi_{0,21}^0(x)/\Psi_{0,11}^0(x)-\widehat{\Psi}^\infty_{1,21}(x) & 1\end{bmatrix}\mathbf{\Psi}(\lambda;x)=\mathbf{\Psi}(\lambda;x),
\end{equation*}
with the help of \eqref{eq:Modified-infty-leading-matrix} and \eqref{eq:B-up}.  Another proof of this result is simply to note that the matrices $\widecheck{\widehat{\mathbf{\Psi}}}(\lambda;x)$ and $\mathbf{\Psi}(\lambda;x)$ satisfy exactly the same analyticity, jump, and normalization conditions, and therefore since $\det(\mathbf{\Psi}(\lambda;x))=1$, Liouville's theorem shows that $\widecheck{\widehat{\mathbf{\Psi}}}(\lambda;x)\mathbf{\Psi}(\lambda;x)^{-1}=\mathbb{I}$.  The proof that \eqref{eq:Schlesinger} can be applied to $\widecheck{\mathbf{\Psi}}(\lambda;x)$ provided that $\Psi^0_{0,22}(x)\neq 0$ so that the latter exists, with the result that $\widehat{\widecheck{\mathbf{\Psi}}}(\lambda;x)=\mathbf{\Psi}(\lambda;x)$, is completely analogous.
\end{proof}

\subsection{The defining inverse monodromy problem for the rational solution $u_n(x;m)$}
\label{sec:inverse-monodromy}
%\textcolor{magenta}{Properly formulate the RHP conditions for $\Psi^{(n)}$.  Explicitly relate $u_n(x;m)$ to the expansion coefficients of the solution.}
%
Let $\mathbf{\Psi}^{(0)}(\lambda;x,m):=\mathbf{\Psi}(\lambda;x)$ be the matrix function defined in Sections~\ref{sec:LaxPair-n0}--\ref{sec:normalized-solutions}, which satisfies \eqref{eq:psi-general-infty}--\eqref{eq:psi-general-0} with $\Theta_0=m$ and $\Theta_\infty=m+1$, and for which $\Psi_{0,11}^0(x)=a(x)=\ee^{-\ii\pi\kappa/2}\ee^{-\ii\pi/4}c\neq 0$ and $\Psi_{0,22}^0(x)\not\equiv 0$ for $b(x)\not\equiv 0$ (note that both inequalities follow from \eqref{eq:ab-cd}--\eqref{eq:cd-identity}). We now apply the Schlesinger transformations \eqref{eq:Schlesinger} and \eqref{eq:Schlesinger-decrease-n} repeatedly, assuming that after each iteration, the condition $\Psi_{0,11}^0(x)\Psi_{0,22}^0(x)\not\equiv 0$ persists\footnote{See statement 2 of Lemma~\ref{l:inverse}.} to obtain for each integer $n\in\mathbb{Z}$ a matrix function $\mathbf{\Psi}^{(n)}(\lambda;x,m)$ that satisfies \eqref{eq:psi-general-infty}--\eqref{eq:psi-general-0} as well as the jump conditions
\begin{equation}
\mathbf{\Psi}^{(n)}_+(\lambda;x,m)=\mathbf{\Psi}^{(n)}_-(\lambda;x,m)\begin{cases}
\VZeroRed(m),&\quad \lambda\in\LZeroRed,\\
\VInftyRed(m),&\quad\lambda\in\LInftyRed,\\
(-1)^n\VZeroBlue(m),&\quad\lambda\in\LZeroBlue,\\
(-1)^n\VInftyBlue(m),&\quad\lambda\in\LInftyBlue,
\end{cases}
\label{eq:jumps-n}
\end{equation}
where now the matrices $\VZeroRed(m)$ and $\VInftyRed(m)$ are defined in \eqref{eq:Red-Jumps} and $\VZeroBlue(m)$ and $\VInftyBlue(m)$ are defined in \eqref{eq:Blue-Jumps}.
%$\VInftyBlue$, $\VInftyRed$, $\VZeroBlue$, and $\VZeroRed$ are
%precisely the matrices defined by \eqref{eq:V-infty-blue}, \eqref{eq:V-infty-red}, \eqref{eq:V-0-blue},
%and \eqref{eq:V-0-red} respectively.
%%As this Schlesinger transformation acts to decrement $n$ and hold $m$ fixed, we use it to extend the hierarchy $\Psi^{(n)}(\lambda;x)$ to negative integers $n$.  
Since $\det(\mathbf{\Psi}^{(0)}(\lambda;x,m))=1$ it follows that $\det(\mathbf{\Psi}^{(n)}(\lambda;x,m))=1$ for all $n\in\mathbb{Z}$.  The \emph{inverse monodromy problem} consists of fixing $n\in\mathbb{Z}$, $m\in\mathbb{C}$, and $x\in\mathbb{C}\setminus\{0\}$ and attempting to determine $\mathbf{\Psi}^{(n)}(\lambda;x,m)$ from the following conditions only:
\begin{itemize}
\item Analyticity:  $\mathbf{\Psi}^{(n)}(\lambda;x,m)$ is analytic for $\lambda\in\mathbb{C}\setminus L$ and analyticity extends to the the contour $L$ from each component of its complement.
\item Jump conditions:  The boundary values taken by $\mathbf{\Psi}^{(n)}(\lambda;x,m)$ on the four oriented arcs of $L$ are to be related by the jump conditions \eqref{eq:jumps-n}.
\item Behavior for small and large $\lambda$:  $\mathbf{\Psi}^{(n)}(\lambda;x,m)$ satisfies the two conditions \eqref{eq:psi-general-infty}--\eqref{eq:psi-general-0} in which $\Theta_0$ and $\Theta_\infty$ are defined in terms of $m$ and $n$ by \eqref{eq:Thetas-m-n}.
\end{itemize}

By its construction in Sections~\ref{sec:LaxPair-n0}--\ref{sec:normalized-solutions}, $\mathbf{\Psi}^{(0)}(\lambda;x,m)$ is the simultaneous solution of a Lax pair of linear problems.  We now show that this is also true for $\mathbf{\Psi}^{(n)}(\lambda;x,m)$, $\forall n\in\mathbb{Z}$, establishing simultaneously some related important properties.  
\begin{lem}\label{l:inverse} 
Let $n\in\mathbb{Z}$ and $m\in\mathbb{C}$ be fixed and suppose the above inverse monodromy problem for $\mathbf{\Psi}(\lambda;x)=\mathbf{\Psi}^{(n)}(\lambda;x,m)$ 
%with asymptotic behavior \eqref{e:1} and \eqref{e:2} 
is solvable for $x$ in some domain $D\subset\mathbb{C}\setminus\{0\}$.  
%$\lambda\in\mathbb{C}\backslash(L_{\textnormal{blue,red}}^0\cup L_{\textnormal{blue,red}}^{\infty})$ and 
%$x\in\mathbb{C}\backslash\mathcal{D}$. 
\begin{itemize}
	\item[1.] For $\lambda\in\mathbb{C}\setminus L$, the function $\mathbf{\Psi}(\lambda;x)=\mathbf{\Psi}^{(n)}(\lambda;x,m)$ is a simultaneous solution matrix of the Lax system \eqref{eq:lambda-eqn}--\eqref{eq:x-eqn} in which
%	\begin{eqnarray}
%		\frac{\partial\Psi^{(n)}}{\partial\lambda}(\lambda;x)&=&\left(\frac{\ii x}{2}\sigma_3+\frac{1}{\lambda}\begin{bmatrix}-\frac{1}{2}\Theta_{\infty} & y\\ v & \frac{1}{2}\Theta_{\infty}\end{bmatrix}+\frac{1}{\lambda^2}\begin{bmatrix}\frac{\ii x}{2}-\ii U & \ii wU\\ -\ii\frac{U-x}{w} & -\frac{\ii x}{2}+\ii U\end{bmatrix}\right)\Psi^{(n)}(\lambda;x);\label{e:3}\\
%		\frac{\partial\Psi^{(n)}}{\partial x}(\lambda;x)&=&\left(\frac{\ii\lambda}{2}\sigma_3+\frac{1}{x}\begin{bmatrix}0 & y\\ v& 0\end{bmatrix}-\frac{1}{\lambda x}\begin{bmatrix}\frac{\ii x}{2}-\ii U & \ii wU\\ -\ii\frac{U-x}{w} & -\frac{\ii x}{2}+\ii U\end{bmatrix}\right)\Psi^{(n)}(\lambda;x)\label{e:4}
%	\end{eqnarray}
the $x$-dependent coefficients $y$, $v$, $s$, and $t$ are given in terms of the leading matrix coefficients in the expansions \eqref{eq:psi-general-infty}--\eqref{eq:psi-general-0} by
\begin{equation}\label{e:5}
y(x)=-\ii x\Psi_{1,12}^{\infty}(x),\quad v(x)=\ii x\Psi_{1,21}^{\infty}(x),\quad s(x)=-x\Psi_{0,11}^0(x)\Psi_{0,12}^{0}(x),\quad t(x)=\frac{\Psi_{0,21}^{0}(x)}{\Psi_{0,11}^{0}(x)}.
\end{equation}	
	\item[2.] None of the three matrix elements $\Psi_{0,11}^{0}(x)$, $\Psi_{0,12}^0(x)$, nor $\Psi_{0,22}^0(x)$ of the leading coefficient in the expansion \eqref{eq:psi-general-0} of $\mathbf{\Psi}(\lambda;x)=\mathbf{\Psi}^{(n)}(\lambda;x,m)$ vanishes identically on the domain $D$. \item[3.] The combination $u(x):=-y(x)/s(x)$ (cf., \eqref{eq:u-recover})
is a solution of the Painlev\'e-III equation \eqref{eq:PIII} meromorphic on $D$ with parameters $\Theta_0$ and $\Theta_\infty$ given by \eqref{eq:Thetas-m-n}.  
%in the form
%	\begin{equation}\label{PIII1}
%		\frac{\dd^2 u}{\dd x^2}=\frac{1}{u}\left(\frac{\dd u}{\dd x}\right)^2-\frac{1}{x}\frac{\dd u}{\dd x}+\frac{4\Theta_0 u^2+4(1-\Theta_{\infty})}{x}+4u^3-\frac{4}{u}
%	\end{equation}
\end{itemize}
\end{lem}
\begin{proof} 
%We first assume that statement 1 holds, implying that none of the functions $\Psi_{0,11}^0(x)$, $\Psi_{0,21}^0(x)$, nor $\Psi_{0,12}^0(x)$ vanishes identically for $x\in D$, and prove statements 2 and 3.
It is a standard result based on Liouville's theorem and the fact that the jump matrices are all unimodular that there can be at most one solution of the inverse monodromy conditions and that this solution satisfies $\det(\mathbf{\Psi}^{(n)}(\lambda;x,m))=1$.  Applying analytic Fredholm theory to a suitable singular integral equation equivalent to the inverse monodromy problem and parametrized analytically by $x\in\mathbb{C}\setminus\{0\}$, existence of a solution for $x\in D$ implies that for each $m\in\mathbb{C}$ and for each fixed $\lambda$ disjoint from the jump contour $L$ for all $x\in D$, $x\mapsto\mathbf{\Psi}^{(n)}(\lambda;x,m)$ is analytic on $D$.  In particular, in a neighborhood of such fixed $\lambda$ and any $x\in D$, $\mathbf{\Psi}^{(n)}(\lambda;x,m)$ is jointly differentiable with respect to both $\lambda$ and $x$.  
Because the jump matrices in \eqref{eq:jumps-n} are independent of both $\lambda$ (on each arc) and $x$,
% and since 
%$\mathbf{\Psi}^{(n)}(\lambda;x,m)$ takes continuous boundary values on $L$ from each sector of analyticity at the intersection point $\lambda=-\ii$,
it follows that the matrices
\begin{equation*}
\mathbf{A}^{(n)}(\lambda;x,m):=\frac{\partial\mathbf{\Psi}^{(n)}}{\partial\lambda}(\lambda;x,m)\mathbf{\Psi}^{(n)}(\lambda;x,m)^{-1}\quad\text{and}\quad
\mathbf{B}^{(n)}(\lambda;x,m):=\frac{\partial\mathbf{\Psi}^{(n)}}{\partial x}(\lambda;x,m)\mathbf{\Psi}^{(n)}(\lambda;x,m)^{-1}
\end{equation*}
are both analytic functions of $(\lambda,x)$ in the domain $(\mathbb{C}\setminus\{0\})\times D$.  Note that to define $\mathbf{B}^{(n)}(\lambda;x,m)$, we may take the jump contour $L$ to be locally independent of $x$ because the boundary values taken from each sector on $L$ are analytic functions of $\lambda$.
%Here we also used the fact that $\det(\Psi^{(n)}(\lambda;x))=1$.  
From \eqref{eq:psi-general-infty} we see that in the limit $\lambda\to\infty$,
\begin{equation}
\begin{split}
\mathbf{A}^{(n)}(\lambda;x,m)&=\frac{\ii x}{2}\sigma_3 + \left(\frac{\ii x}{2}\big[\mathbf{\Psi}_1^{\infty}(x),\sigma_3\big]-\frac{\Theta_{\infty}}{2}\sigma_3\right)\lambda^{-1}\\
&\quad\quad{}+\left(-\mathbf{\Psi}_1^{\infty}(x)-\frac{\Theta_{\infty}}{2}\big[\mathbf{\Psi}_1^{\infty}(x),\sigma_3\big]+\frac{\ii x}{2}\Big\{\big[\mathbf{\Psi}_2^{\infty}(x),\sigma_3\big]-\big[\mathbf{\Psi}_1^{\infty}(x),\sigma_3\big]\mathbf{\Psi}_1^{\infty}(x)\Big\}\right)\lambda^{-2}+O(\lambda^{-3}),\\
\mathbf{B}^{(n)}(\lambda;x,m)&=\frac{\ii}{2}\sigma_3\lambda +\frac{\ii}{2}[\mathbf{\Psi}_1^{\infty}(x),\sigma_3] + \left(\mathbf{\Psi}_{1}^{\infty\prime}(x)+\frac{\ii}{2}\big[\mathbf{\Psi}_2^{\infty}(x),\sigma_3\big]-\frac{\ii}{2}\big[\mathbf{\Psi}_1^{\infty}(x),\sigma_3\big]\mathbf{\Psi}_1^{\infty}(x)\right)\lambda^{-1}+O(\lambda^{-2}).
\end{split}
\label{eq:AnBn-large}
\end{equation}
Similarly, in the limit $\lambda\to 0$, from \eqref{eq:psi-general-0} we get
\begin{equation}
\begin{split}
\mathbf{A}^{(n)}(\lambda;x,m)&=\frac{\ii x}{2}\mathbf{\Psi}_0^{0}(x)\sigma_3\mathbf{\Psi}_0^{0}(x)^{-1}\lambda^{-2}\\
&\quad{}+\left(\frac{\Theta_0}{2}\mathbf{\Psi}_0^{0}(x)\sigma_3\mathbf{\Psi}_0^{0}(x)^{-1}+
\frac{\ii x}{2}\mathbf{\Psi}_1^{0}(x)\sigma_3\mathbf{\Psi}_0^{0}(x)^{-1}-\frac{\ii x}{2}\mathbf{\Psi}_0^{0}(x)\sigma_3
\mathbf{\Psi}_0^{0}(x)^{-1}\mathbf{\Psi}_1^{0}(x)\mathbf{\Psi}_0^{0}(x)^{-1}\right)\lambda^{-1}\\
&\quad{}+O(1)\\
\mathbf{B}^{(n)}(\lambda;x,m)&=-\frac{\ii}{2}\mathbf{\Psi}_0^{0}(x)\sigma_3\mathbf{\Psi}_0^{0}(x)^{-1}\lambda^{-1}+\mathbf{\Psi}_{0}^{0\prime}(x)\mathbf{\Psi}_0^0(x)^{-1}+\frac{\ii}{2}\big[\mathbf{\Psi}_0^0(x)\sigma_3\mathbf{\Psi}_0^0(x)^{-1},\mathbf{\Psi}_1^0(x)\mathbf{\Psi}_0^0(x)^{-1}\big]+O(\lambda).
\end{split}
\label{eq:AnBn-small}
\end{equation}
Therefore, Liouville's theorem shows that $\mathbf{A}^{(n)}(\lambda;x,m)$ and $\mathbf{B}^{(n)}(\lambda;x,m)$ are Laurent polynomials:
\begin{equation}
\mathbf{A}^{(n)}(\lambda;x,m)=\frac{\ii x}{2}\sigma_3 +
\left(\frac{\ii x}{2}\big[\mathbf{\Psi}_1^{\infty}(x),\sigma_3\big]-\frac{\Theta_{\infty}}{2}\sigma_3\right)\lambda^{-1}
+ \frac{\ii x}{2}\mathbf{\Psi}_0^{0}(x)\sigma_3\mathbf{\Psi}_0^{0}(x)^{-1}\lambda^{-2}
\label{eq:An}
\end{equation}
and
\begin{equation}
\mathbf{B}^{(n)}(\lambda;x,m)=
\frac{\ii}{2}\sigma_3\lambda + \frac{\ii}{2}[\mathbf{\Psi}_1^{\infty}(x),\sigma_3] -\frac{\ii}{2}\mathbf{\Psi}_0^{0}(x)\sigma_3
\mathbf{\Psi}_0^{0}(x)^{-1}\lambda^{-1}.
\label{eq:Bn}
\end{equation}
Furthermore, the coefficients of different powers of $\lambda$ in \eqref{eq:An}--\eqref{eq:Bn} are  analytic matrix-valued functions of $x$ on $D$.
Since $\mathbf{\Psi}^{(n)}_\lambda(\lambda;x,m)=\mathbf{A}^{(n)}(\lambda;x,m)\mathbf{\Psi}^{(n)}(\lambda;x,m)$ and $\mathbf{\Psi}^{(n)}_x(\lambda;x,m)=\mathbf{B}^{(n)}(\lambda;x,m)\mathbf{\Psi}^{(n)}(\lambda;x,m)$, matching \eqref{eq:An}--\eqref{eq:Bn} with \eqref{eq:lambda-eqn}--\eqref{eq:x-eqn} using also $\det(\mathbf{\Psi}^0_0(x))=1$ yields the expressions \eqref{e:5} and proves statement 1.\bigskip
% under the stated assumption that $\Psi_{0,21}^0(x)\not\equiv 0$ on $D$, which guarantees in particular that $w(x)$ is well-defined for $x\in D$ \textcolor{red}{(only need $w(x)^{-1}$ and $w(x)U(x)$ to be well-defined on $D$ so it is irrelevant whether $\Psi_{0,21}^0(x)$ vanishes identically or not)}.  
% It is a consequence of analytic Fredholm theory that the elements of the expansion coefficients $\mathbf{\Psi}^\infty_1(x)$ and $\mathbf{\Psi}^0_0(x)$ are analytic functions of $x\in D$, which implies that $y(x)$, $v(x)$, and $U(x)$ are analytic on $D$, while $w(x)$ is meromorphic on $D$.

Suppose $\Psi^0_{0,11}(x)\equiv 0$ holds as an identity on $D$.  From $\det(\mathbf{\Psi}^0_0(x))\equiv 1$ we then get $\Psi^0_{0,12}(x)\Psi^0_{0,21}(x)\equiv -1$.  Therefore $s(x)\equiv 0$ and $\tfrac{1}{2}\ii x-\ii s(x)t(x)\equiv -\tfrac{1}{2}\ii x$, so the matrices $\mathbf{A}^{(n)}(\lambda;x,m)$ and $\mathbf{B}^{(n)}(\lambda;x,m)$ can be written in the alternate form
\begin{equation}
\begin{split}
	\mathbf{A}=\mathbf{A}^{(n)}(\lambda;x,m)&=\frac{\ii x}{2}\sigma_3+\frac{1}{\lambda}\begin{bmatrix}-\frac{1}{2}\Theta_{\infty} & y\\ v & \frac{1}{2}\Theta_{\infty}\end{bmatrix}+\frac{1}{\lambda^2}\begin{bmatrix} -\tfrac{1}{2}\ii x & 0\\ -\ii V & \tfrac{1}{2}\ii x\end{bmatrix}\\
	\mathbf{B}=\mathbf{B}^{(n)}(\lambda;x,m)&=\frac{\ii \lambda}{2}\sigma_3+\frac{1}{x}\begin{bmatrix} 0 & y\\ v & 0\end{bmatrix}-\frac{1}{\lambda x}\begin{bmatrix} -\tfrac{1}{2}\ii x & 0\\ -\ii V & \tfrac{1}{2}\ii x\end{bmatrix}
\end{split}
\label{eq:Lax-System-Case-2}
\end{equation}
with $y(x)$ and $v(x)$ defined as in \eqref{e:5}, while
\begin{equation*}
V(x):=-x\Psi_{0,21}^0(x)\Psi_{0,22}^0(x).
\end{equation*}
Existence of the simultaneous fundamental solution matrix $\mathbf{\Psi}^{(n)}(\lambda;x,m)$ of the Lax system implies that these coefficient matrices satisfy the zero-curvature compatibility condition $\mathbf{A}_x-\mathbf{B}_\lambda+[\mathbf{A},\mathbf{B}]=\mathbf{0}$, which in turn implies that $y(x)\equiv 0$ also, making $\mathbf{A}$ and $\mathbf{B}$ lower-triangular with explicit diagonal entries.  Therefore, the elements of the first row are determined from the Lax system up to overall constants $c_1$ and $c_2$ by
\begin{equation*}
\begin{bmatrix}\Psi_{11}^{(n)}(\lambda;x,m) & \Psi_{12}^{(n)}(\lambda;x,m)\end{bmatrix}=
\begin{bmatrix}c_1\ee^{\ii x(\lambda+\lambda^{-1})/2}\OurPower{\lambda}{-\Theta_\infty/2}& c_2\ee^{\ii x(\lambda+\lambda^{-1})/2}\OurPower{\lambda}{-\Theta_\infty/2}\end{bmatrix}.
\end{equation*}
Applying the condition \eqref{eq:psi-general-infty} then forces the choice $c_2=0$, so $\Psi^{(n)}_{12}(\lambda;x,m)\equiv 0$ and therefore also $\Psi^0_{0,12}(x)\equiv 0$ on $D$.  But since $\det(\mathbf{\Psi}_0^0(x))\equiv 1$, this contradicts the assumption that $\Psi^0_{0,11}(x)\equiv 0$.\bigskip  

Suppose next that $\Psi^0_{0,22}(x)\equiv 0$.  Then using $\det(\mathbf{\Psi}^0_{0}(x))\equiv 0$ shows that the combination $-\ii t(x)(s(x)t(x)-x)$ vanishes identically, and then the compatibility condition for the matrices $\mathbf{A}^{(n)}(\lambda;x,m)$ and $\mathbf{B}^{(n)}(\lambda;x,m)$ implies that also $v(x)\equiv 0$.  Therefore, the coefficient matrices are upper-triangular in this case, and since also $\tfrac{1}{2}\ii x-\ii s(x)t(x)\equiv -\tfrac{1}{2}\ii x$, the second row of $\mathbf{\Psi}^{(n)}(\lambda;x,m)$ takes the form
\begin{equation}
\begin{bmatrix}\Psi^{(n)}_{21}(\lambda;x,m) & \Psi^{(n)}_{22}(\lambda;x,m)\end{bmatrix}=
\begin{bmatrix}  c_1\ee^{-\ii x(\lambda+\lambda^{-1})/2}\OurPower{\lambda}{\Theta_\infty/2} & c_2\ee^{-\ii x(\lambda+\lambda^{-1})/2}\OurPower{\lambda}{\Theta_\infty/2}\end{bmatrix}
\end{equation}
where $c_1$ and $c_2$ are constants.  Applying as before the condition \eqref{eq:psi-general-infty} now forces $c_1=0$, so $\Psi^0_{0,21}(x)$ and $\Psi^0_{0,22}(x)$ both vanish identically in contradiction to $\det(\mathbf{\Psi}^0_{0}(x))\equiv 1$.\bigskip  

Finally, suppose that $\Psi^{0}_{0,12}(x)\equiv 0$ on $D$.  Then also $s(x)\equiv 0$ and $s(x)t(x)\equiv 0$, and the compatibility condition for the Lax system implies that also $y(x)\equiv 0$, making the coefficient matrices lower-triangular.  Solving for the first row of $\mathbf{\Psi}^{(n)}(\lambda;x,m)$ now yields
\begin{equation}
\begin{bmatrix}\Psi_{11}^{(n)}(\lambda;x,m) & \Psi_{12}^{(n)}(\lambda;x,m)\end{bmatrix}=
\begin{bmatrix}c_1\ee^{\ii x (\lambda-\lambda^{-1})/2}\OurPower{\lambda}{-\Theta_\infty/2} &
c_2\ee^{\ii x(\lambda-\lambda^{-1})/2}\OurPower{\lambda}{-\Theta_\infty/2}\end{bmatrix}
\end{equation}
for constants $c_1$ and $c_2$, and applying the normalization condition \eqref{eq:psi-general-infty} forces $c_1=1$ and $c_2=0$.  For this result to be compatible with \eqref{eq:psi-general-0} it is then necessary that $\Theta_0+\Theta_\infty=0$, i.e., that $m=-\tfrac{1}{2}$.  But, if $m=-\tfrac{1}{2}$, the jump condition across the arc $\LInftyRed$ implies that (using $\Theta_\infty=\tfrac{1}{2}-n$ for $m=-\tfrac{1}{2}$)
\begin{equation}
\Psi^{(n)}_{12+}(\lambda;x,-\tfrac{1}{2})-\Psi^{(n)}_{12-}(\lambda;x,-\tfrac{1}{2})=\sqrt{2\pi}\Psi_{11-}^{(n)}(\lambda;x,-\tfrac{1}{2})=\sqrt{2\pi}\ee^{\ii x(\lambda-\lambda^{-1})/2}\OurPower{\lambda}{n/2-1/4},\quad\lambda\in \LInftyRed.
\end{equation}
The right-hand side is nonzero on the indicated contour, which is obviously inconsistent with $\Psi^{(n)}_{12}(\lambda;x,-\tfrac{1}{2})\equiv 0$ implied by $c_2=0$.
All together, since assuming $\Psi^0_{0,11}(x)\equiv 0$, $\Psi^0_{0,22}(x)\equiv 0$, or $\Psi^0_{0,12}(x)\equiv 0$ leads in each case to a contradiction, we have established statement 2.\bigskip

The potentials $y(x)$, $v(x)$, and $s(x)$ are analytic on $D$ by analytic Fredholm theory, and by statement 2 it also holds that $t(x)$ is meromorphic on $D$.  In general, the compatibility condition $\mathbf{A}_x-\mathbf{B}_\lambda + [\mathbf{A},\mathbf{B}]=\mathbf{0}$ on the matrices \eqref{eq:An}--\eqref{eq:Bn} implies that these four functions satisfy the coupled nonlinear differential equations \eqref{eq:PIII-system}.  
%Since $\det(\mathbf{\Psi}^{(n)}(\lambda;x,m))=1$, the differential operators $\partial_\lambda-\mathbf{A}^{(n)}$
%and $\partial_x-\mathbf{B}^{(n)}$ have a common kernel spanned by the (independent) columns of the fundamental solution matrix $\mathbf{\Psi}^{(n)}(\lambda;x,m)$ and therefore they commute.  It follows that
%\begin{equation}
%\frac{\partial\mathbf{A}^{(n)}}{\partial x}(\lambda;x,m)-\frac{\partial\mathbf{B}^{(n)}}{\partial\lambda}(\lambda;x,m) +
%[\mathbf{A}^{(n)}(\lambda;x,m),\mathbf{B}^{(n)}(\lambda;x,m)]=\mathbf{0}
%\label{eq:ZCC}
%\end{equation}
%The left-hand side of \eqref{eq:ZCC} is a Laurent polynomial in $\lambda$ involving the powers $\lambda^0$, $\lambda^{-1}$, and $\lambda^{-2}$ only.  Separating out the coefficients of the powers yields a system of nonlinear ordinary differential equations in the independent variable $x$, namely
%the system \eqref{eq:PIII-system}.  Since these differential equations involve division by $w(x)$, we are now relying on the assumption $\Psi_{0,11}^0(x)\not\equiv 0$ for $x\in D$, which implies that $w(x)$ does not vanish identically on $D$.  
The system \eqref{eq:PIII-system} has a conserved quantity $I$ defined by \eqref{eq:Integral}; 
%which
%is finite because $xw(x)\not\equiv 0$.  To 
to determine its constant value, it suffices evaluate it at any $x\in D$ that makes each term in $I$ finite (it is only necessary to avoid the isolated zeros of $\Psi^0_{0,11}(x)$). Note that the direct monodromy problem \eqref{eq:lambda-eqn} has an irregular singular point of Poincar\'e rank $1$ at $\lambda=0$ and hence by general theory two fundamental solutions exist in a vicinity of $\lambda=0$ which are uniquely specified by their asymptotics as $\lambda\to 0$ in the associated Stokes sectors.  An explicit computation of the formal expansions directly from the differential equation \eqref{eq:lambda-eqn} (cf., \cite{W}) yields, upon comparison with the expansion \eqref{eq:psi-general-0} the identity $I=\Theta_0$. 
Now, the expression $u(x)=-y(x)/s(x)$ defines a meromorphic function on $D$ because the zeros of $s(x)$ are isolated by statement 2.  
%The assumption that $\Psi_{0,12}^0(x)\Psi_{0,21}^0(x)\not\equiv 0$ on $D$ means that $U(x)$ is not the zero function and hence $u(x)$ is well-defined as a meromorphic function on $D$ from \eqref{eq:u-recover} \textcolor{red}{(only need $w(x)U(x)$ to not vanish identically; so it is really $\Psi_{0,12}^0(x)\Psi_{0,11}^0(x)\not\equiv 0$ that is needed)}.  
Differentiating this expression using \eqref{eq:PIII-system} and eliminating $y(x)=-s(x)u(x)$, one finds that $u(x)$ and the product $s(x)t(x)$ are related by the first order differential equation \eqref{eq:first-order-PIII}.  Solving this identity for $s(x)t(x)$ in terms of $u(x)$ and $u'(x)$ and differentiating the result yields a second-order differential expression involving $u(x)$ alone.  On the other hand, the product $s(x)t(x)$ can be differentiated directly using \eqref{eq:PIII-system} after which $y(x)$ can be eliminated using $y(x)=-s(x)u(x)$, $v(x)$ can be eliminated using the integral of motion $I=\Theta_0$, and finally the product $s(x)t(x)$ can be eliminated once again using \eqref{eq:first-order-PIII}.  Equating these two equivalent expressions for the derivative of $s(x)t(x)$ 
%\textcolor{red}{(Explain here the details of how to get PIII from \eqref{eq:first-order-PIII}.  The following refers to old variables and needs to be modified.)}  Taking the derivative of this equation, (i) first eliminate $U'$ using \eqref{eq:PIII-system}, (ii) then eliminate the product $wv$ using the identity $I=\Theta_0$, (iii) then eliminate $y/w$ using the definition \eqref{eq:u-recover} of $u(x)$, and (iv) finally use \eqref{eq:first-order-PIII} again to eliminate $U$ in favor of $u$ and $u'$; the resulting equation 
yields precisely the Painlev\'e-III equation \eqref{eq:PIII} for $u(x)$.
This proves statement 3.
\end{proof}

Next, we have the following result.
\begin{lem}\label{l:exist}
Given $n\in\mathbb{Z}$ and $m\in\mathbb{C}$, there is a finite set $P_n(m)$ such that the inverse monodromy problem is uniquely solvable for $x\in \mathbb{C}\setminus (\mathbb{R}_-\cup P_n(m))$.  The corresponding solution $u(x)$ of the Painlev\'e-III equation \eqref{eq:PIII} is a rational function.
\end{lem}
\begin{proof}
Since existence of a solution implies uniqueness by a Liouville argument, it is sufficient to establish existence for suitable $x$. To this end we 
first consider $n=0$. The explicit solution $\mathbf{\Psi}^{(0)}(\lambda;x,m)$ of the direct monodromy problem constructed in Section~\ref{sec:direct-monodromy} obviously satisfies the conditions of the inverse monodromy problem as well, and it is well-defined for $x\in\mathbb{C}\setminus\mathbb{R}_-$.  A calculation shows that the leading term $\mathbf{\Psi}^0_0(x)$ takes the form
\begin{equation}
\mathbf{\Psi}^0_0(x)=\begin{bmatrix}\ee^{-\ii\pi/4}\ee^{-\ii\pi m/2} & 2^m\ee^{-3\pi\ii/4}\ee^{2x}x^m\\
\tfrac{1}{4}\ee^{3\pi\ii/4}2^{-m}(2m+1)x^{-1}\ee^{-2x}x^{-m} & \tfrac{1}{4}\ee^{\ii\pi/4}\ee^{\ii\pi m/2}(2m+1+4x)x^{-1}
\end{bmatrix},\quad n=0.
\end{equation}
Obviously, $\Psi^0_{0,11}(x)$, $\Psi^0_{0,22}(x)$, $\ee^{-2x}x^{-m}\Psi^0_{0,12}(x)$, and $\ee^{2x}x^m\Psi^0_{0,21}(x)$ are all rational functions (with poles at $x=0$ only).  Similar calculations give
\begin{equation}
\Psi^\infty_{1,12}(x)=-\ii 2^m\ee^{-\ii\pi m/2}\ee^{2x}x^m\quad\text{and}\quad
\Psi^\infty_{1,21}(x)=-\ii 2^{-(m+4)}\ee^{\ii\pi m/2}(2m+1)(4x-2m-1)\ee^{-2x}x^{-m},\quad n=0.
\end{equation}
Therefore also $\ee^{-2x}x^{-m}\Psi^\infty_{1,12}(x)$ and $\ee^{2x}x^m\Psi^\infty_{1,21}(x)$ are rational functions.  Clearly, $P_0(m)=\emptyset$ (the pole at $x=0$ is already excluded as $0\in\mathbb{R}_-$), and the corresponding solution $u(x)=-\ii\Psi^\infty_{1,12}(x)/(\Psi^0_{0,11}(x)\Psi^0_{0,12}(x))\equiv 1$ is clearly rational. Next, let $k\ge 0$ be an integer, and suppose that $P_k(m)$ is finite, that the inverse monodromy problem for $n=k$ is (uniquely) solvable for $m\in\mathbb{C}$ and $x\in\mathbb{C}\setminus(\mathbb{R}_-\cup P_k(m))$, and that for $n=k$ the expansion coefficients $\Psi^0_{0,11}(x)$, $\Psi^0_{0,22}(x)$, $\ee^{-2x}x^{-m}\Psi^0_{0,12}(x)$, $\ee^{2x}x^m\Psi^0_{0,21}(x)$, $\ee^{-2x}x^{-m}\Psi^\infty_{1,12}(x)$, and $\ee^{2x}x^m\Psi^\infty_{1,21}(x)$ are all rational functions.  Taking $D=\mathbb{C}\setminus(\mathbb{R}_-\cup P_k(m))$ and applying Lemma~\ref{l:inverse} we see that $\Psi^0_{0,11}(x)\not\equiv 0$ holds on $D$, so the Schlesinger transformation \eqref{eq:Schlesinger} exists on $D$ except at the finitely-many zeros of 
the rational function $\Psi^0_{0,11}(x)$ in $D$.  Letting $P_{k+1}(m)$ denote the union of the set of these zeros with $P_k(m)$, the matrix $\mathbf{\Psi}^{(k+1)}(\lambda;x,m):=\widehat{\mathbf{\Psi}}^{(k)}(\lambda;x,m)$ clearly satisfies all of the properties of the inverse monodromy problem for $n=k$, $m\in\mathbb{C}$, and $x\in\mathbb{C}\setminus(\mathbb{R}_-\cup P_k(m))$.  Since, according to \eqref{eq:B-up} and the inductive hypotheses in force, the matrix $\ee^{-x\sigma_3}x^{-m\sigma_3/2}\widehat{\mathbf{B}}(x)x^{m\sigma_3/2}\ee^{x\sigma_3}$ is a rational function of $x$, it then follows that the transformed expansion coefficients are such that $\widehat{\Psi}^0_{0,11}(x)$, $\widehat{\Psi}^0_{0,22}(x)$, $\ee^{-2x}x^{-m}\widehat{\Psi}^0_{0,12}(x)$, $\ee^{2x}x^m\widehat{\Psi}^0_{0,21}(x)$, $\ee^{-2x}x^{-m}\widehat{\Psi}^\infty_{1,12}(x)$, and $\ee^{2x}x^m\widehat{\Psi}^\infty_{1,21}(x)$ are all rational functions, as is $\widehat{u}(x)=-\ii\widehat{\Psi}^\infty_{1,12}(x)/(\widehat{\Psi}^0_{0,11}(x)\widehat{\Psi}^0_{0,12}(x))$, which by Lemma~\ref{l:inverse} satisfies the Painlev\'e-III equation with parameters $n=k+1$ and $m$.  The desired conclusion therefore holds for all integers $n\ge 0$ by induction on $n$.\smallskip

For $n\le 0$, we apply instead the transformation \eqref{eq:Schlesinger-decrease-n}--\eqref{eq:B-down} to decrease $n$, making use of the fact that $\Psi^0_{0,22}(x)\not\equiv 0$.  A parallel induction argument shows that the desired conclusion holds for all negative integers $n$ as well.
\end{proof}
We remark that the points at which the inverse monodromy problem fails to have a solution need not coincide with the poles or zeros of the rational function $u(x)$.

\subsection{Induced B\"acklund transformations}
\label{sec:Backlund}
%\textcolor{magenta}{Derive here the B\"acklund transformations for incrementing $n$.}

%\textcolor{magenta}{Things to do here include:
%\begin{itemize}
%\item 
%Use Lemma~\ref{lemma:inverse} to show that, if $u$ is a rational solution for integer $n$ and non-integer $m$, then it necessarily arises from $u_0(x;m)=\pm 1$ by iterated B\"acklund transformations (by ratcheting back towards $n=0$ by the inverse transformations, and using the elementary fact that if $n=0$ but $m\neq 0$ there are exactly two rational solutions, namely $u(x)=\pm 1$).  We want to arrive at the exact count for the rational solutions.  
%\item Check how the elementary symmetries of odd reflection, inversion, and rotation thread through the B\"acklund transformation.
%\end{itemize}
%}
%
The Schlesinger transformation \eqref{eq:Schlesinger} implies a corresponding B\"acklund transformation for the potentials $v(x)$, $y(x)$, $s(x)$ and $t(x)$:
\begin{equation}
\begin{split}
\widehat{v}(x)&:=-\ii xt(x)\\
\widehat{y}(x)&:=\frac{\ii}{x}\left( xs(x)-(\Theta_\infty-1)y(x)+y(x)^2t(x)\right)\\
\widehat{s}(x)&:= \frac{\ii y(x)}{x^2}\left(x^2+y(x)^2t(x)^2-\Theta_\infty y(x)t(x)-v(x)y(x)\right)\\
\widehat{t}(x)&:=\ii x\frac{y(x)t(x)^2-\Theta_\infty t(x)-v(x)}{x^2+y(x)^2t(x)^2-\Theta_\infty y(x)t(x)-v(x)y(x)}.
\end{split}
\label{eq:Backlund-n-plus-1}
\end{equation}
It is straightforward to confirm directly that whenever $(v,y,s,t)$ solves \eqref{eq:PIII-system}, then so does $(\widehat{v},\widehat{y},\widehat{s},\widehat{t})$ when $\Theta_\infty$ is replaced in \eqref{eq:PIII-system} by $\widehat{\Theta}_\infty:=\Theta_\infty -1$.  Defining $\widehat{u}(x):=-\widehat{y}(x)/\widehat{s}(x)$ and using \eqref{eq:Backlund-n-plus-1} along with $u(x)=-y(x)/s(x)$, the identity $I=\Theta_0$, and \eqref{eq:first-order-PIII}, one arrives at Gromak's transformation \eqref{eq:Backlund-n}.  This proves the following.
\begin{prop}
The rational function $u(x)$ obtained from the inverse monodromy problem with parameters $m\in\mathbb{C}$ and $n\in\mathbb{Z}_{\ge 0}$ coincides with the function $u(x)=u_n(x;m)$ obtained via $n$ iterations of the B\"acklund transformation \eqref{eq:Backlund-n} starting from the seed $u_0(x;m)\equiv 1$.
\label{p:u-sub-n}
\end{prop}
This result establishes the link between the algebraic representation \eqref{eq:s-recurrence}--\eqref{eq:un-fraction} of $u_n(x;m)$ and the analytic representation afforded by the inverse monodromy problem.  It is easy to check that the B\"acklund transformation \eqref{eq:Backlund-n} preserves the property $u(x)\to 1$ as $x\to\infty$, and therefore $u_n(x;m)$ and its odd reflection $R^2u_n(x;m)=-u_n(-x;m)$ are distinct rational solutions of the Painlev\'e-III equation \eqref{eq:PIII} for the same values of $n\in\mathbb{Z}$ and $m\in\mathbb{C}$.  Suppose that $m\not\in\mathbb{Z}$, but $u(x)$ is a rational solution of \eqref{eq:PIII} for parameters $(m,n)$.  We may invert the B\"acklund transformation (the corresponding explicit formula for the inverse can be obtained from the $n$-reducing Schlesinger transformation \eqref{eq:Schlesinger-decrease-n} in the same way that Gromak's transformation can be deduced from \eqref{eq:Schlesinger}) and apply the inverse $n$ times to $u(x)$, thereby arriving at a rational solution of \eqref{eq:PIII} with parameters $(m,0)$.  However, it has been shown that when $n=0$ and $m\not\in\mathbb{Z}$, the only rational solutions of \eqref{eq:PIII} are the constants $\pm 1$.  By Lemma~\ref{lemma:inverse}, the inverse transformation is injective and therefore it follows that either $u(x)=u_n(x;m)$ or $u(x)=R^2u_n(x;m)$, i.e., for $m\not\in\mathbb{Z}$ and $n\in\mathbb{Z}$, there are exactly two rational solutions.  From this it follows that for general $m$ it is sufficient to study the family of functions $\{u_n(x;m)\}_{n\in\mathbb{Z}}$ to analyze all rational solutions of \eqref{eq:PIII}.  This can be done using the inverse monodromy problem, suitably reformulated in the form of Riemann-Hilbert Problem~\ref{rhp:renormalized}, which we now are in a position to establish.

\subsection{Renormalization}
\label{sec:renormalized}
To study the asymptotic behavior of the rational solutions for $n$ a large integer and $m\in\mathbb{C}$ fixed, it is useful to study in place of $\mathbf{\Psi}^{(n)}(\lambda;x,m)$ a matrix that is normalized to the identity matrix as $\lambda\to\infty$.  Therefore, we consider the matrix $\mathbf{Y}^{(n)}(\lambda;x,m)$
defined by a small modification of the left-hand side of \eqref{eq:psi-general-infty}:
\begin{equation*}
\mathbf{Y}^{(n)}(\lambda;x,m):=\mathbf{\Psi}^{(n)}(\lambda;x,m)\OurPower{\lambda}{\Theta_\infty\sigma_3/2}\ee^{-\ii x(\lambda-\lambda^{-1})\sigma_3/2}
\end{equation*}
where $\Theta_\infty$ is given by \eqref{eq:Thetas-m-n}.  It is easy to check that if it exists for a given $x\in\mathbb{C}$, this matrix satisfies the conditions of Riemann-Hilbert Problem~\ref{rhp:renormalized}.
Recalling the expansions \eqref{eq:Y-expand-infty}--\eqref{eq:Y-expand-zero}, 
the coefficients $\mathbf{Y}_1^\infty(x)$ and $\mathbf{Y}_0^0(x)$ are related to the expansions of $\mathbf{\Psi}^{(n)}(\lambda;x,m)$ by
\begin{equation}
\mathbf{\Psi}_1^\infty(x)=\mathbf{Y}_1^\infty(x)-\frac{\ii x}{2}\sigma_3\quad\text{and}\quad
\mathbf{\Psi}_0^0(x)=\mathbf{Y}_0^0(x),
\label{eq:Y-Psi-coefficients}
\end{equation}
and therefore combining \eqref{eq:u-recover}, \eqref{e:5}, and \eqref{eq:Y-Psi-coefficients}, the rational solution $u_n(x;m)$ of the Painlev\'e-III equation \eqref{eq:PIII} is given by \eqref{eq:u-n-from-Y-formula}.\smallskip

It is a consequence of the cyclic relation \eqref{eq:cyclic} that at this point we may take the contour $L$ to be arbitrary subject to the restrictions indicated in  Subsection \ref{master}.  Such a modified form of $L$ can always be connected with the original $L$ by a homotopy that moves the intersection point but maintains the increment of arguments as specified by \eqref{eq:increment-argument-red}--\eqref{eq:increment-argument-blue}, and throughout which the power functions $\OurPower{\lambda}{p}$ appearing in the jump conditions \eqref{eq:Yjump-1}--\eqref{eq:Yjump-4} are deformed in a natural way by analytic continuation.  This completes the proof of Theorem~\ref{thm:RH-representation}.

%\textcolor{magenta}{This means that all of the jumps have common triangularity when $m$ is a half-integer.  Therefore, the problem is basically additive rather than multiplicative in this case, and it can be solved by the Plemelj formula, with asymptotic normalization built in by imposing vanishing moments.  This is like the Fokas-Its-Kitaev problem for orthogonal polynomials.}

%
%$\LZeroRed$ begins at $\lambda=\lNaught$ on the left side of $\LInftyBlue\cup\LZeroBlue$ and terminates at $\lambda=0$ in a direction such that $\ee^{\ii x (\lambda-\lambda^{-1})}$ is exponentially decaying,
%\item $\LInftyRed$ emanates from $\lambda=\infty$ in a direction such that $\ee^{\ii x (\lambda-\lambda^{-1})}$ is exponentially decaying and terminates at $\lambda=\lNaught$ on the right side of $\LInftyBlue\cup\LZeroBlue$,
%\item $\LInftyBlue$ emanates from $\lambda=\infty$ in a direction such that $\ee^{-\ii x(\lambda-\lambda^{-1})}$ is exponentially decaying and terminates at $\lambda=\lNaught$ on the left side of $\LInftyRed\cup\LZeroRed$, and
%\item $\LZeroBlue$ begins at $\lambda=\lNaught$ on the right side of $\LInftyRed\cup\LZeroRed$ and terminates at $\lambda=0$ in a direction such that $\ee^{-\ii x(\lambda-\lambda^{-1})}$ is exponentially decaying,
%\end{itemize}
%and the four arcs of $L$ do not otherwise intersect.  

\section{Algebraic Solution of Riemann-Hilbert Problem~\ref{rhp:renormalized} for $m\in\mathbb{Z}+\tfrac{1}{2}$}
\label{sec:m-half-integer}
Note that the jump matrices on $\LInftyRed\cup\LZeroRed$ reduce to the identity if $m=\tfrac{1}{2},\tfrac{3}{2},\tfrac{5}{2},\dots$.   Likewise, the jump matrices on $\LInftyBlue\cup\LZeroBlue$ reduce to the identity if $m=-\tfrac{1}{2},-\tfrac{3}{2},-\tfrac{5}{2},\dots$.  This observation results in an algebraic solution technique for half-integer values of $m$ that we will now describe.\bigskip

Suppose first that $m=\tfrac{1}{2}+k$, $k\in\mathbb{Z}_{\geq 0}$.  Then according to Riemann-Hilbert Problem~\ref{rhp:renormalized}, $\mathbf{Y}^{(n)}(\lambda;x,m)$ is analytic for $\mathbb{C}\setminus L$ where now we may take $L=\LZeroBlue\cup\LInftyBlue$ because the jump matrices on $\LZeroRed\cup\LInftyRed$ reduce to the identity so analyticity follows by Morera's theorem.  Moreover, the jump condition on $L$ takes the form
\begin{equation}
\mathbf{Y}^{(n)}_+(\lambda;x,\tfrac{1}{2}+k)=\mathbf{Y}^{(n)}_-(\lambda;x,\tfrac{1}{2}+k)\begin{bmatrix}1 & 0\\
\displaystyle\frac{\sqrt{2\pi}}{k!}(\OurPower{\lambda}{k/2+3/4})_+(\OurPower{\lambda}{k/2+3/4})_-\lambda^{-n}\ee^{-\ii x(\lambda-\lambda^{-1})} & 1\end{bmatrix},\quad\lambda\in L,
\quad k\in\mathbb{Z}_{\geq 0}.
\label{eq:Y-jump-positive-half-integer}
\end{equation}
A similar Morera argument therefore implies that the second column of $\mathbf{Y}^{(n)}(\lambda;x,\tfrac{1}{2}+k)$ has no jump across $L$ and hence is analytic for $\lambda\in\mathbb{C}\setminus\{0\}$.  Applying the normalization condition at $\lambda=\infty$ yields $Y^{(n)}_{12}(\lambda;x,\tfrac{1}{2}+k)=O(\lambda^{-1})$ and $Y^{(n)}_{22}(\lambda;x,\tfrac{1}{2}+k)=1+O(\lambda^{-1})$ as $\lambda\to \infty$, while $Y^{(n)}_{j2}(\lambda;x,\tfrac{1}{2}+k)=O(\lambda^{k+1})$ as $\lambda\to 0$ for $j=1,2$.  It follows by Liouville's theorem that
\begin{equation*}
Y^{(n)}_{12}(\lambda;x,\tfrac{1}{2}+k)=\sum_{j=1}^{k+1}a^{(n,k)}_j(x)\lambda^{-j}\quad\text{and}\quad
Y^{(n)}_{22}(\lambda;x,\tfrac{1}{2}+k)=1+\sum_{j=1}^{k+1}b^{(n,k)}_j(x)\lambda^{-j}
\end{equation*}
where $a^{(n,k)}_j(x)$ and $b^{(n,k)}_j(x)$ are coefficients to be determined.  The first column of the jump condition \eqref{eq:Y-jump-positive-half-integer} can then be used together with the Plemelj formula and the normalization conditions $Y^{(n)}_{11}(\lambda;x,\tfrac{1}{2}+k)=1+O(\lambda^{-1})$ and $Y^{(n)}_{21}(\lambda;x,\tfrac{1}{2}+k)=O(\lambda^{-1})$ as $\lambda\to\infty$ to express $Y^{(n)}_{j1}(\lambda;x,\tfrac{1}{2}+k)$ explicitly in terms of $Y^{(n)}_{j2}(\lambda;x,\tfrac{1}{2}+k)$:
\begin{equation*}
Y^{(n)}_{11}(\lambda;x,\tfrac{1}{2}+k)=1+\frac{1}{\ii k!\sqrt{2\pi}}\int_L \frac{Y^{(n)}_{12}(\mu;x,\tfrac{1}{2}+k)(\OurPower{\mu}{k/2+3/4})_+(\OurPower{\mu}{k/2+3/4})_-\mu^{-n}\ee^{-\ii x(\mu-\mu^{-1})}}{\mu-\lambda}\,\dd\mu
\end{equation*}
and
\begin{equation*}
Y^{(n)}_{21}(\lambda;x,\tfrac{1}{2}+k)=\frac{1}{\ii k!\sqrt{2\pi}}\int_L \frac{Y^{(n)}_{22}(\mu;x,\tfrac{1}{2}+k)(\OurPower{\mu}{k/2+3/4})_+(\OurPower{\mu}{k/2+3/4})_-\mu^{-n}\ee^{-\ii x(\mu-\mu^{-1})}}{\mu-\lambda}\,\dd\mu.
\end{equation*}
It only remains to enforce the condition that $Y^{(n)}_{j1}(\lambda;x,\tfrac{1}{2}+k)=O(\lambda^{k+1})$ as $\lambda\to 0$ for $j=1,2$.  Expanding $(\mu-\lambda)^{-1}$ for small $\lambda$ in a geometric series and elimination of the second column elements in favor of $a_j^{(n,k)}(x)$ and $b_j^{(n,k)}(x)$, $j=1,\dots,k+1$, yields separate $(k+1)\times (k+1)$ linear systems of Hankel type separately for the $a_j^{(n,k)}(x)$ and the $b_j^{(n,k)}(x)$:  defining coefficients $I^+_{n,k,j}(x)$ by
\begin{equation}
I^+_{n,k,j}(x):=\int_L(\OurPower{\lambda}{k/2+3/4})_+(\OurPower{\lambda}{k/2+3/4})_-\lambda^{-n-j}\ee^{-\ii x(\lambda-\lambda^{-1}))}\,\dd\lambda
\label{eq:Iplus}
\end{equation}
the systems are
\begin{equation*}
\mathbf{H}^+_{n,k}(x)\mathbf{a}^{(n,k)}(x) = -\ii\sqrt{2\pi}k!\mathbf{e}^{(1)}\quad\text{and}\quad
\mathbf{H}^+_{n,k}(x)\mathbf{b}^{(n,k)}(x) = -\mathbf{v}^+_{n,k}(x)
\end{equation*}
where $\mathbf{e}^{(1)}:=(1,0,0,\dots,0)^\top$ denotes the first coordinate unit vector, the unknowns are arranged in vectors as
\begin{equation*}
\quad \mathbf{a}^{(n,k)}(x):=(a^{(n,k)}_1(x),\dots,a^{(n,k)}_{k+1}(x))^\top,\quad \mathbf{b}^{(n,k)}(x):=(b^{(n,k)}_1(x),\dots,b^{(n,k)}_{k+1}(x))^\top,
\end{equation*}
and the Hankel matrix and right-hand side vector for the $\mathbf{b}^{(n,k)}(x)$ system are
\begin{equation*}
\mathbf{H}^+_{n,k}(x):=\{I^+_{n,k,p+q}(x)\}_{p,q=1}^{k+1},\quad\mathbf{v}^+_{n,k}(x):=(I^+_{n,k,1}(x),\dots,I^+_{n,k,k+1}(x))^\top.
\end{equation*}
Therefore, when $m=\tfrac{1}{2}+k$, $k\in\mathbb{Z}_{\geq 0}$, Riemann-Hilbert Problem~\ref{rhp:renormalized} has a solution obtained by linear algebra in dimension $k+1$ provided that $x$ is such that the complex Hankel determinant
\begin{equation*}
D^+_{n,k}(x):=\det(\mathbf{H}_{n,k}^+(x))
\end{equation*}
is nonzero.  From the formula \eqref{eq:u-n-from-Y-formula} we then get the corresponding rational solution $u_n(x;\tfrac{1}{2}+k)$ of the Painlev\'e-III equation \eqref{eq:PIII} for $k=0,1,2,3,\dots$ in the form
\begin{equation}
u_n(x;\tfrac{1}{2}+k)=\frac{\sqrt{2\pi}k!a_1^{(n,k)}(x)}{\displaystyle a_{k+1}^{(n,k)}(x)\sum_{j=1}^{k+1}a_j^{(n,k)}(x)I^+_{n,k,j+k+2}(x)},\quad k\in\mathbb{Z}_{\geq 0}.
\end{equation}
For instance, if $k=0$, then we obtain
\begin{equation*}
a_1^{(n,0)}(x)=-\frac{\ii\sqrt{2\pi}}{D^+_{n,0}(x)}\quad\text{and}\quad
b_1^{(n,0)}(x)=-\frac{1}{D^+_{n,0}(x)}\int_L(\OurPower{\lambda}{3/4})_+(\OurPower{\lambda}{3/4})_-\lambda^{-n-1}\ee^{-\ii x(\lambda-\lambda^{-1})}\,\dd\lambda
\end{equation*}
where
\begin{equation*}
D^+_{n,0}(x):=\int_L(\OurPower{\lambda}{3/4})_+(\OurPower{\lambda}{3/4})_-\lambda^{-n-2}\ee^{-\ii x(\lambda-\lambda^{-1})}\,\dd\lambda.
\end{equation*}
Therefore, assuming that $D^+_{n,0}(x)\neq 0$, the solution of Riemann-Hilbert Problem~\ref{rhp:renormalized}
has been obtained in closed form for arbitrary integer $n$ and for $m=\tfrac{1}{2}$.  The corresponding rational solution of the Painlev\'e-III equation \eqref{eq:PIII} is
\begin{equation}
u_n(x;\tfrac{1}{2})=\ii\frac{\displaystyle\int_{\LInftyBlue\cup\LZeroBlue}(\OurPower{\lambda}{3/4})_+(\OurPower{\lambda}{3/4})_-\lambda^{-(n+2)}\ee^{-\ii x(\lambda-\lambda^{-1})}\,\dd\lambda}{\displaystyle\int_{\LInftyBlue\cup\LZeroBlue}(\OurPower{\lambda}{3/4})_+(\OurPower{\lambda}{3/4})_-\lambda^{-(n+3)}\ee^{-\ii x(\lambda-\lambda^{-1})}\,\dd\lambda}.
\label{eq:u-m-half}
\end{equation}
Assuming that the integrals in the fraction \eqref{eq:u-m-half} have no common zeros, we see that the zeros of $u_n(x;\tfrac{1}{2})$ are the points where Riemann-Hilbert Problem~\ref{rhp:renormalized} has no solution for $m=\tfrac{1}{2}$, while the poles of $u_n(x;\tfrac{1}{2})$ are regular points for $\mathbf{Y}^{(n)}(\lambda;x,\tfrac{1}{2})$.\bigskip

Next assume that $m=-(\tfrac{1}{2}+k)$, $k\in\mathbb{Z}_{\geq 0}$. Then according to Riemann-Hilbert Problem~\ref{rhp:renormalized}, the matrix $\mathbf{Y}^{(n)}(\lambda;x,-\tfrac{1}{2}-k)$ is analytic for $\lambda\in\mathbb{C}\setminus L$, where we may now take $L$ to be the contour $L=\LInftyRed\cup\LZeroRed$, across which we may write the jump condition in the form
\begin{equation*}
\mathbf{Y}^{(n)}_+(\lambda;x,-\tfrac{1}{2}-k)=\mathbf{Y}^{(n)}_-(\lambda;x,-\tfrac{1}{2}-k)
\begin{bmatrix}1 & \displaystyle\frac{\sqrt{2\pi}}{k!}(\OurPower{\lambda}{k-1/2})_\infty\lambda^n\ee^{\ii x(\lambda-\lambda^{-1})}\\0 & 1\end{bmatrix},\quad\lambda\in L,\quad k\in\mathbb{Z}_{\geq 0},
\end{equation*}
where $(\OurPower{\lambda}{k-1/2})_\infty$ denotes the function
\begin{equation*}
(\OurPower{\lambda}{k-1/2})_\infty:=\begin{cases}\OurPower{\lambda}{k-1/2},&\quad \lambda\in\LInftyRed,\\
-\OurPower{\lambda}{k-1/2},&\quad\lambda\in\LZeroRed.\end{cases}
\end{equation*}
Note that $(\OurPower{\lambda}{k-1/2})_\infty$ is continuous at the junction point between $\LZeroRed$ and $\LInftyRed$ because $\OurPower{\lambda}{k-1/2}$ changes sign across its jump contour of $\LZeroBlue\cup\LInftyBlue$.  Obviously, it is now the first column of $\mathbf{Y}^{(n)}(\lambda;x,-\tfrac{1}{2}-k)$ that is analytic for $\lambda\in\mathbb{C}\setminus\{0\}$, and from the normalization conditions $Y^{(n)}_{11}(\lambda;x,-\tfrac{1}{2}-k)=1+O(\lambda^{-1})$ and $Y^{(n)}_{21}(\lambda;x,-\tfrac{1}{2}-k)=O(\lambda^{-1})$ as $\lambda\to\infty$ while $Y^{(n)}_{j1}(\lambda;x,-\tfrac{1}{2}-k)=O(\lambda^{-k})$ as $\lambda\to 0$, we see that the entries of the first column necessarily take the form
\begin{equation*}
Y^{(n)}_{11}(\lambda;x,-\tfrac{1}{2}-k)=1+\sum_{j=1}^kc^{(n,k)}_j(x)\lambda^{-j}\quad\text{and}\quad
Y^{(n)}_{21}(\lambda;x,-\tfrac{1}{2}-k)=\sum_{j=1}^kd^{(n,k)}_j(x)\lambda^{-j}
\end{equation*}
where $c^{(n,k)}_j(x)$ and $d^{(n,k)}_j(x)$ are coefficients to be determined.  The jump condition together with the normalization condition that $Y^{(n)}_{12}(\lambda;x,-\tfrac{1}{2}-k)=O(\lambda^{-1})$ and $Y^{(n)}_{22}(\lambda;x,-\tfrac{1}{2}-k)=1+O(\lambda^{-1})$ as $\lambda\to\infty$ then determines the second column from the first:
\begin{equation*}
Y_{12}^{(n)}(\lambda;x,-\tfrac{1}{2}-k)=\frac{1}{\ii k!\sqrt{2\pi}}\int_L\frac{Y_{11}^{(n)}(\mu;x,-\tfrac{1}{2}-k)(\OurPower{\mu}{k-1/2})_\infty\mu^n\ee^{\ii x(\mu-\mu^{-1})}}{\mu-\lambda}\,\dd\mu
\end{equation*}
and
\begin{equation*}
Y_{22}^{(n)}(\lambda;x,-\tfrac{1}{2}-k)=1+\frac{1}{\ii k!\sqrt{2\pi}}\int_L\frac{Y_{21}^{(n)}(\mu;x,-\tfrac{1}{2}-k)(\OurPower{\mu}{k-1/2})_\infty\mu^n\ee^{\ii x(\mu-\mu^{-1})}}{\mu-\lambda}\,\dd\mu.
\end{equation*}
Then demanding that $Y_{j2}^{(n)}(\lambda;x,-\tfrac{1}{2}-k)=O(\lambda^k)$ as $\lambda\to 0$ yields two Hankel systems on the coefficients $c_j^{(n,k)}(x)$ and $d_j^{(n,k)}(x)$.  Setting
\begin{equation*}
I^-_{n,k,j}(x):=\int_L(\OurPower{\lambda}{k-1/2})_\infty\lambda^{n-j}\ee^{\ii x(\lambda-\lambda^{-1})}\,\dd\lambda,
\end{equation*}
these systems take the form
\begin{equation*}
\mathbf{H}^-_{n,k}(x)\mathbf{c}^{(n,k)}(x)=-\mathbf{v}^-_{n,k}(x)\quad\text{and}\quad
\mathbf{H}^-_{n,k}(x)\mathbf{d}^{(n,k)}(x)=-\ii k!\sqrt{2\pi}\mathbf{e}^{(1)}
\end{equation*}
where
\begin{equation*}
\mathbf{c}^{(n,k)}(x):=(c_1^{(n,k)}(x),\dots,c_k^{(n,k)}(x))^\top,\quad
\mathbf{d}^{(n,k)}(x):=(d_1^{(n,k)}(x),\dots,d_k^{(n,k)}(x))^\top,
\end{equation*}
and the Hankel matrix and right-hand side vector for the $\mathbf{c}^{(n,k)}(x)$ system are
\begin{equation*}
\mathbf{H}^-_{n,k}(x):=\{I^-_{n,k,p+q}(x)\}_{p,q=1}^k,\quad \mathbf{v}^-_{n,k}(x):=(I^-_{n,k,1}(x),\dots,I^-_{n,k,k}(x))^\top.
\end{equation*}
Therefore, if $k\in\mathbb{Z}_{\geq 1}$ and $m=-\tfrac{1}{2}-k$, then Riemann-Hilbert Problem~\ref{rhp:renormalized} has a solution obtained by $k\times k$ linear algebra, provided that the Hankel determinant
\begin{equation*}
D^-_{n,k}(x):=\det(\mathbf{H}^-_{n,k}(x))
\end{equation*}
is nonzero given $x$.  From \eqref{eq:u-n-from-Y-formula} we get the corresponding rational solution of the Painlev\'e-III equation \eqref{eq:PIII} in the form
\begin{equation}
u_n(x;-\tfrac{1}{2}-k)=\frac{\displaystyle \ii I^-_{n,k,0}(x)+\ii\sum_{j=1}^kc_j^{(n,k)}(x)I^-_{n,k,j}(x)}{\displaystyle c_k^{(n,k)}(x)I^-_{n,k,k+1}(x)+c_k^{(n,k)}(x)\sum_{j=1}^kc_j^{(n,k)}(x)I^-_{n,k,j+k+1}(x)},\quad k\in\mathbb{Z}_{\geq 1}.
\end{equation}
Note that if $k=0$, the linear algebra system is trivial and hence Riemann-Hilbert Problem~\ref{rhp:renormalized} \emph{always} has a solution when $m=-\tfrac{1}{2}$:
\begin{equation*}
\mathbf{Y}^{(n)}(\lambda;x,-\tfrac{1}{2})=\begin{bmatrix}1 & \displaystyle\frac{1}{\ii \sqrt{2\pi}}
\int_L\frac{(\OurPower{\mu}{-1/2})_\infty\mu^n\ee^{\ii x(\mu-\mu^{-1})}}{\mu-\lambda}\,\dd\mu\\0 & 1\end{bmatrix}.
\end{equation*}
The corresponding rational solution of the Painlev\'e-III equation \eqref{eq:PIII} is
\begin{equation*}
u_n(x;-\tfrac{1}{2})=\ii\frac{\displaystyle\int_{\LInftyRed\cup\LZeroRed}(\OurPower{\lambda}{-1/2})_\infty\lambda^n\ee^{\ii x(\lambda-\lambda^{-1})}\,\dd\lambda}{\displaystyle\int_{\LInftyRed\cup\LZeroRed}(\OurPower{\lambda}{-1/2})_\infty\lambda^{n-1}\ee^{\ii x(\lambda-\lambda^{-1})}\,\dd\lambda}.
\end{equation*}

\begin{rem}
We remark that in both cases the solution becomes more complicated as $|m|$ increases.  This is similar to the situation with the explicit solution of the Fokas-Its-Kitaev Riemann-Hilbert problem for orthogonal polynomials \cite{FokasIK91}. Significantly however, the large parameter $n$ appears explicitly in the (algebraic) solution of the Hankel system corresponding to any fixed half-integral value of $m$.  It is this latter feature that enables a direct large-$n$ asymptotic analysis by classical steepest descent methods \cite{BothnerM18}.
\end{rem}

Another observation is that the formula \eqref{eq:u-m-half} can be written in terms of Bessel functions.
Indeed, we may write this formula in simplified form as
\begin{equation*}
u_n(x;\tfrac{1}{2})=\ii\frac{\displaystyle\int_0^\infty\lambda^{-n-1/2}\ee^{-\ii x(\lambda-\lambda^{-1})}\,\dd\lambda}{\displaystyle\int_0^\infty\lambda^{-n-3/2}\ee^{-\ii x(\lambda-\lambda^{-1})}\,\dd\lambda}
\end{equation*}
where in both integrals the path of integration is the same, chosen (depending on $x$) so that the integrals are convergent at $\lambda=0,\infty$, and also the branch of $\lambda^{-n-1/2}$ is arbitrary as long as it is analytic along the contour of integration and taken to be the same in both integrals. By the substitution $\lambda=e^t$ and comparison with \cite[Equation 10.9.18]{DLMF} we then find that if $\mathrm{Im}(x)>0$, then
\begin{equation*}
u_n(x;\tfrac{1}{2})=\ii\frac{H^{(2)}_{n-1/2}(-\tfrac{\ii}{2}x)}{H^{(2)}_{n+1/2}(-\tfrac{\ii}{2} x)}
\end{equation*}
where $H^{(2)}_\nu(z)$ denotes a Hankel function.  This formula admits meromorphic continuation to the whole complex $x$-plane.  The same formula can then be expressed in terms of spherical Bessel functions of the second kind \cite[10.47(ii)]{DLMF} as
\begin{equation*}
u_n(x;\tfrac{1}{2})=\ii\frac{\mathsf{h}^{(2)}_{n-1}(-\tfrac{\ii}{2} x)}{\mathsf{h}^{(2)}_n(-\tfrac{\ii}{2} x)}.
\end{equation*}
The functions $\ee^{\ii z}\mathsf{h}_n^{(2)}(z)$ are explicit polynomials in $z^{-1}$ \cite[Equation 10.49.7]{DLMF} and this in turn leads to the explicit formula
\begin{equation*}
u_n(x;\tfrac{1}{2})=\frac{\displaystyle\sum_{j=1}^n\frac{(2n-j-1)!}{(n-j)!(j-1)!}x^j}{\displaystyle\sum_{j=0}^n\frac{(2n-j)!}{(n-j)!j!}x^j}.
\end{equation*}
The identification of $u(x;\tfrac{1}{2})$ with ratios of Bessel polynomials was also noted in \cite{ClarksonLL16}.  More generally, from \cite[Equation 10.9.18]{DLMF} it is clear that the integrals $I^\pm_{n,k,j}(x)$ are proportional to Hankel functions, and hence the expression for $u_n(x;\pm (\tfrac{1}{2}+k))$ can always be written in terms of ratios of Hankel-type determinants whose entries are Bessel functions.  More important from the point of view of asymptotic analysis in the large-$n$ limit however is the fact that the coefficients are integrals that may be analyzed by classical steepest descent methods; see \cite{BothnerM18}.


\begin{thebibliography}{99}
\bibitem{Bertola12}
M. Bertola, ``On the location of poles for the Ablowitz-Segur family of solutions to the second Painlev\'e equation,'' \textit{Nonlinearity} \textbf{25}, 1179--1185, 2012.
\bibitem{BertolaB15}
M. Bertola and T. Bothner, ``Zeros of large degree Vorob'ev-Yablonski polynomials via a Hankel determinant identity,'' \textit{Int.\@ Math.\@ Res.\@ Not.\@} \textbf{2015}, 9330--9399, 2015.
\bibitem{BothnerM18}
T. Bothner and P. D. Miller, ``Rational solutions of the Painlev\'e-III equation:  large parameter asymptotics,'' in preparation, 2018.
\bibitem{BuckinghamM14}
R. J. Buckingham and P. D. Miller, ``Large-degree asymptotics of rational Painlev\'e-II functions:  noncritical behaviour,'' \textit{Nonlinearity} \textbf{27}, 2489--2577, 2014.
\bibitem{BuckinghamM15}
R. J. Buckingham and P. D. Miller, ``Large-degree asymptotics of rational Painlev\'e-II functions:  critical behaviour,'' \textit{Nonlinearity} \textbf{28}, 1539--1596, 2015.
\bibitem{Buckingham17}
R. J. Buckingham, ``Large-degree asymptotics of rational Painlev\'e-IV functions associated to generalized Hermite polynomials,'' \texttt{arXiv:1706.09005}, 2017.
\bibitem{Clarkson03}
P. A. Clarkson, ``The third Painlev\'e equation and associated special polynomials,'' \textit{J. Phys. A:  Math. Gen.} \textbf{36}, 9507--9532, 2003.
\bibitem{ClarksonLL16}
P. A. Clarkson, C.-K. Law, C.-H. Lin, ``An algebraic proof for the Umemura polynomials for the third Painlev\'e equation,'' \texttt{arxiv:1609.00495}, 2016.
%\bibitem{Dubrovin81} B. A. Dubrovin, ``Theta functions and non-linear equations,'' \textit{Russ.\@ Math.\@ Surveys} \textbf{36}, 11--92, 1981.
\bibitem{CostinHT14}
O. Costin, M. Huang, and S. Tanveer, ``Proof of the Dubrovin conjecture and analysis of the tritronqu\'ee solutions of $P_I$,'' \textit{Duke Math.\@ J.} \textbf{163}, 665--704, 2014.
\bibitem{FokasIK91} A. S. Fokas, A. R. Its, and A. V. Kitaev, ``Discrete Painlev\'e equations and their appearance in quantum gravity,'' \textit{Comm.\@ Math.\@ Phys.\@} \textbf{142}, 313--344, 1991.
\bibitem{FokasIKN06}
A. S. Fokas, A. R. Its, A. A. Kapaev, and V. Yu.\@ Novokshenov, \textit{Painlev\'e Transcendents:  The Riemann-Hilbert Approach}, Volume 128, Mathematical Surveys and Monographs, American Mathematical Society, Providence, 2006.
\bibitem{GamayunIL13} O. Gamayun, N. Iorgov, and O. Lisovyy, ``How instanton combinatorics solves Painlev\'e VI, V, and III's,'' \textit{J. Phys. A:  Math. Theor.} \textbf{46}, 335203, 2013.
\bibitem{Gromak73}
V. I. Gromak, ``The solutions of Painlev\'e's third equation,'' \textit{Differencial'nye Uravnenija} \textbf{9}, 2082--2083, 1973 (in Russian).
% Gromak was evidently the first person to derive the B\"acklund transformation for PIII.
\bibitem{JimboM81b}
M. Jimbo and T. Miwa, ``Monodromy preserving deformation of linear ordinary differential equations with rational coefficients.  II,'' \textit{Physica D} \textbf{2}, 407--448, 1981. 
\bibitem{JoshiM03}
N. Joshi and M. Mazzocco, ``Existence and uniqueness of tri-tronqu\'ee solutions of the second Painlev\'e hierarchy,'' \textit{Nonlinearity} \textbf{16}, 427--439, 2003.
%% Although more concerned with the hierarchy, maybe the best reference for tritronquee solutions of PII.  
%% They exist, apparently for any value of $m$.  No poles asymptotically in a sector of width 4\pi/3.
\bibitem{KajiwaraM99}
K. Kajiwara and T. Masuda, ``On the Umemura polynomials for the Painlev\'e-III equation,'' \textit{Phys.\@ Lett.\@ A} \textbf{260}, 462--467, 1999.
\bibitem{MillerS17}
P. D. Miller and Y. Sheng, ``Rational solutions of the Painlev\'e-II equation revisited,'' \textit{SIGMA} \textbf{13}, 65, 29 pages, 2017.
\bibitem{MilneC93}
A. E. Milne and P. A. Clarkson, ``Rational solutions and B\"acklund transformations for the third Painlev\'e equation,'' in P. A. Clarkson, ed., \textit{Applications of Analytic and Geometric Methods to Nonlinear Differential Equations}, Kluwer Academic Publishers, 341--352, 1993.
\bibitem{MilneCB97}
A. E. Milne, P. A. Clarkson, and A. P. Bassom, ``B\"acklund transformations and solution hierarchies for the third Painlev\'e equation,'' \textit{Stud.\@ Appl.\@ Math.\@} \textbf{98}, 139--194, 1997.
% This paper includes Bessel function formulas that should agree with what we get for the special case of m=1/2.
\bibitem{Novokshenov12}
V. Yu. Novokshenov, ``Tronqu\'ee solutions of the Painlev\'e II equation,'' \textit{Theor. Math. Phys.} \textbf{172}, 1136--1146, 2012.
% Another reference about tronquee solutions of PII.  Here he gives the Stokes constants for tritronquee solutions, which should be useful to us.
\bibitem{DLMF}
F. W. J. Olver, A. B. Olde Daalhuis, D. W. Lozier, B. I. Schneider, R. F. Boisvert, C. W. Clark, B. R. Miller, and B. V. Saunders, eds., NIST Digital Library of Mathematical Functions, \texttt{http://dlmf.nist.gov/}, Release 1.0.14,  2016. 
\bibitem{Umemura99}
H. Umemura, ``Painlev\'e equations in the past 100 years,'' \textit{Amer. Math. Soc. Transl.} (2) Vol. {\bf 204}, 2001
%``100 years of the Painlev\'e equation,'' \textit{S\-ugaku} \textbf{51}, 395--420, 1999 (in Japanese).
\bibitem{Vorobev65} A. P. Vorob'ev, ``On the rational solutions of the second Painlev\'e equation,'' \textit{Differ.\@ Equations} \textbf{1}, 58--59, 1965.
\bibitem{W} W. Wasow, \textit{Asymptotic Expansions for Ordinary Differential Equations}, Interscience-Wiley, New York, 1965.
\bibitem{Yablonskii59} A. I. Yablonskii, ``On rational solutions of the second Painlev\'e equation,'' \textit{Vesti AN BSSR, Ser.\@ Fiz.-Tech.\@ Nauk}, no. 3, 30--35, 1959.
\end{thebibliography}
\end{document}